\documentclass[]{article}

\usepackage[margin=30truemm]{geometry}

\usepackage{ascmac}
\usepackage{fancybox}

\usepackage[cmex10]{amsmath} 
\usepackage{amssymb} 
\usepackage{amsthm}

\usepackage{cases} 

\usepackage{fixltx2e}

\usepackage[subrefformat=parens]{subcaption}
\usepackage{graphicx}
\usepackage{epstopdf}

\usepackage{xcolor}

\usepackage{cite}

\usepackage{algorithm}
\usepackage{algorithmic}

\newtheorem{definition}{Definition}%
\newtheorem{theorem}{Theorem}%
\newtheorem{lemma}{Lemma}%
\newtheorem{proposition}{Proposition}%
\newtheorem{corollary}{Corollary}%
\newtheorem{assumption}{Assumption}%
\newtheorem{remark}{Remark} 
\newtheorem{example}{Example}

\usepackage{pgfplots}
\pgfplotsset{compat=newest}
\usepackage{tikz}
\usetikzlibrary{positioning}
\usetikzlibrary{arrows}

\newcommand{\MyHighlight}[1]{\textbf{#1}}

\newcommand{\VisibleColTwo}[1]{#1}

\begin{document}

	\title{Second Moment Polytopic Systems: Generalization of Uncertain Stochastic Linear Dynamics%
		\thanks{%
			This work has been submitted to the IEEE for possible publication. Copyright may be transferred without notice, after which this version may no longer be accessible.
			This work was partly supported by JSPS KAKENHI Grant Number JP18K04222.
			The materials of this paper have been published in part in conference proceedings \cite{ItoIFACWC20,ItoCDC20}.  
			We would like to thank Editage (www.editage.jp) for the English language editing.
	}} 
	
	\author{Yuji Ito\thanks{Yuji Ito is the corresponding author and with Toyota Central R\&D Labs., Inc., 41-1 Yokomichi, Nagakute-shi, Aichi 480-1192, Japan	(e-mail: ito-yuji@mosk.tytlabs.co.jp).}
		\and
		Kenji Fujimoto\thanks{Kenji Fujimoto is with the Department of Aeronautics and Astronautics, Graduate School of Engineering, Kyoto University, Kyotodaigakukatsura, Nishikyo-ku, Kyoto-shi, Kyoto 615-8540, Japan (e-mail: k.fujimoto@ieee.org).}
	}  
	
	\date{}
	
	\maketitle

\newcommand*{\SymColor}[1]{\textcolor{red}{#1}}%\cR{#1}  %{#1}
\renewcommand{\SymColor}[1]{#1}%\cR{#1}  %{#1}

%%%%operators%%%%%%%%operators%%%%%%%%operators%%%%%%%%operators%%%%%%%%operators%%%%
%%%%operators%%%%%%%%operators%%%%%%%%operators%%%%%%%%operators%%%%%%%%operators%%%%

\newcommand*{\SetSymMat}[1]{\SymColor{\mathbb{\zzR}}_{\mathrm{sym}}^{#1}}%\cR{#1}  %{#1}

\newcommand*{\Identity}[1]{\SymColor{\boldsymbol{\zzI}}_{#1}}

\newcommand*{\vectorWildCard}{\SymColor{\boldsymbol{\bullet}}}

\newcommand*{\MyTRANSPO}{\SymColor{\top}}

\newcommand*{\El}[2]{\SymColor{[}{#1}\SymColor{]}_{#2}}
\newcommand*{\VEC}[2][]{\SymColor{\mathrm{vec }#1(} #2 \SymColor{#1)}}%\overrightarrow{#1}}
\newcommand*{\VECH}[2][]{\SymColor{\mathrm{vech}#1(} #2 \SymColor{#1)}}%{\overrightarrow{#1}}
\newcommand*{\TRACE}[2][]{\SymColor{\mathrm{tr}}#1( #2 #1)}

\newcommand*{\iEig}[2]{\SymColor{\lambda_{#1} (}#2\SymColor{)}}
\newcommand*{\iEigVec}[2]{\SymColor{\boldsymbol{\nu}_{#1} (}#2\SymColor{)}}

\newcommand*{\MyRank}[1]{\SymColor{\mathrm{rank}(}#1\SymColor{)}}

\newcommand*{\PDF}[1]{\SymColor{\zzp(}#1\SymColor{)}}
\newcommand*{\CondPDF}[2]{\SymColor{\zzp(}#1  \SymColor{|} #2 \SymColor{)}}

\newcommand*{\SeqTVUnc}{\TVUnc{\bullet}}
\newcommand*{\Expect}[2][]{\SymColor{\mathrm{E}#1[} #2  \SymColor{#1]}}
\newcommand*{\CondExpectTI}[2][]{     \SymColor{\mathrm{E}#1[} #2 \SymColor{#1|}_{\TIUnc}  \SymColor{#1]}}
\newcommand*{\CondExpectTV}[3][]{     \SymColor{\mathrm{E}#1[} #2 \SymColor{#1|}_{\SeqTVUnc}  \SymColor{#1]}}
\newcommand*{\CondExpectAllTV}[4][]{  \SymColor{\mathrm{E}#1[} #2 \SymColor{#1|}_{\SeqTVUnc}  \SymColor{#1]}}
\newcommand*{\CondExpectAlleqTV}[4][]{\SymColor{\mathrm{E}#1[} #2 \SymColor{#1|}_{\SeqTVUnc}  \SymColor{#1]}}
\newcommand*{\CondCovTI}[2][]{\SymColor{\mathrm{Cov}#1[} #2       \SymColor{#1|}_{\TIUnc}  \SymColor{#1]}}
\newcommand*{\CondCovTV}[3][]{\SymColor{\mathrm{Cov}#1[} #2       \SymColor{#1|}_{\SeqTVUnc}  \SymColor{#1]}}
\newcommand*{\anotherCondExpectTI}[2][]{\SymColor{\mathrm{E}#1[} #2 \SymColor{#1|}_{\anotherTIUnc}  \SymColor{#1]}}
\newcommand*{\anotherCondCovTI}[2][]{\SymColor{\mathrm{Cov}#1[} #2 \SymColor{#1|}_{\anotherTIUnc}  \SymColor{#1]}}%\overrightarrow{#1}}

%%%%operators%%%%%%%%operators%%%%%%%%operators%%%%%%%%operators%%%%%%%%operators%%%%
%%%%operators%%%%%%%%operators%%%%%%%%operators%%%%%%%%operators%%%%%%%%operators%%%%

%%%%for notation%%%%%%%%for notation%%%%%%%%for notation%%%%%%%%for notation%%%%%%%%for notation%%%%
%%%%for notation%%%%%%%%for notation%%%%%%%%for notation%%%%%%%%for notation%%%%%%%%for notation%%%%

\newcommand*{\NotationVec}{\SymColor{\boldsymbol{\zzy}}}
\newcommand*{\NotationMat}{\SymColor{\boldsymbol{\zzX}}}
\newcommand*{\NotationSymMat}{\SymColor{\boldsymbol{\zzY}}}
\newcommand*{\NotationSquMat}{\SymColor{\boldsymbol{\zzX}}}
\newcommand*{\IDNotation}{\SymColor{\zzi}}
\newcommand*{\IDbNotation}{\SymColor{\zzj}}
\newcommand*{\DimANotation}{\SymColor{\zzn}}
\newcommand*{\DimBNotation}{\SymColor{\zzm}}
\newcommand*{\NotationA}{\SymColor{\mathrm{a}}}
\newcommand*{\NotationB}{\SymColor{\mathrm{b}}}

\newcommand*{\OpeNotationMat}{\SymColor{\boldsymbol{\zzY}}}%{\SymColor{\widetilde{\boldsymbol{M}}}}
\newcommand*{\OpeNotationbMat}{\SymColor{\boldsymbol{\zzX}}}
\newcommand*{\OpeNotationSymMat}{\SymColor{\boldsymbol{\zzX}}_{\mathrm{s}}}

%%%%for notation%%%%%%%%for notation%%%%%%%%for notation%%%%%%%%for notation%%%%%%%%for notation%%%%
%%%%for notation%%%%%%%%for notation%%%%%%%%for notation%%%%%%%%for notation%%%%%%%%for notation%%%%

\newcommand*{\IDEl}{\SymColor{\zzi}}
\newcommand*{\IDbEl}{\SymColor{\zzj}}
\newcommand*{\IDcEl}{\SymColor{\zzi^{\prime}}}
\newcommand*{\IDdEl}{\SymColor{\zzj^{\prime}}}

%%%%original systems%%%%%%%%original systems%%%%%%%%original systems%%%%%%%%original systems%%%%%%%%original systems%%%%
%%%%original systems%%%%%%%%original systems%%%%%%%%original systems%%%%%%%%original systems%%%%%%%%original systems%%%%

\newcommand*{\DimX}{\SymColor{         \zzn       }}%n_{\boldsymbol{x}}}}
\newcommand*{\DimU}{\SymColor{         \zzm       }}%n_{\boldsymbol{u}}}}

\newcommand*{\MyT}{\SymColor{\zzt}}
\newcommand*{\MybT}{\SymColor{\zzs}}
\newcommand*{\MycT}{\SymColor{\zzt^{\prime}}}
\newcommand*{\State}[1]{\SymColor{\boldsymbol{\zzx}}_{#1}}
\newcommand*{\Input}[1]{\SymColor{\boldsymbol{\zzu}}_{#1}}

\newcommand*{\FBgain}{\SymColor{\boldsymbol{\zzK}}}

\newcommand*{\NonArgDriftMat}[1]{\SymColor{\boldsymbol{\zzA}_{#1}}} 
\newcommand*{\NonArgInMat}[1]{\SymColor{\boldsymbol{\zzB}_{#1}}}
\newcommand*{\DriftMat}[1]{\NonArgDriftMat{#1}( \TVUnc{#1}  )} %{\SymColor{\boldsymbol{A}( \TVUnc{#1} , \TVSto{#1}  )}}
\newcommand*{\InMat}[1]{\NonArgInMat{#1}( \TVUnc{#1}   )} %{\SymColor{\boldsymbol{B}( \TVUnc{#1}  , \TVSto{#1}  )}}
\newcommand*{\CLDriftMat}[1]{\NonArgDriftMat{\mathrm{cl},#1}   } %( \TVUnc{#1} , \FBgain  ) 

\newcommand*{\TIUnc}{\SymColor{\boldsymbol{\theta}}}
\newcommand*{\TVUnc}[1]{\TIUnc_{#1}}
\newcommand*{\DomTVUnc}{\SymColor{\mathbb{\zzS}_{\TIUnc}}}
\newcommand*{\DimTVUnc}{\SymColor{     \zzd_{\TIUnc}           }}%n_{\NonArgTVSto}}}

\newcommand*{\NonArgvecAB}[1]{\SymColor{\boldsymbol{\zzv}_{#1}}} 
\newcommand*{\vecTVAB}[1]{\NonArgvecAB{#1}  ( \TVUnc{#1}   )  }
\newcommand*{\vecTIAB}[1]{\NonArgvecAB{#1}  ( \TIUnc   )  }

%%%%original systems%%%%%%%%original systems%%%%%%%%original systems%%%%%%%%original systems%%%%%%%%original systems%%%%
%%%%original systems%%%%%%%%original systems%%%%%%%%original systems%%%%%%%%original systems%%%%%%%%original systems%%%%

%%%%Definition of SMP systems%%%%%%%%Definition of SMP systems%%%%%%%%Definition of SMP systems%%%%
%%%%Definition of SMP systems%%%%%%%%Definition of SMP systems%%%%%%%%Definition of SMP systems%%%%

\newcommand*{\NonArgMapPolyW}{\SymColor{ {\boldsymbol{\phi}}} } %_{\MyT} 
\newcommand*{\MapPolyW}[1]{\NonArgMapPolyW \SymColor{(} #1 \SymColor{)}} %_{\MyT} 
\newcommand*{\DomPolytope}{\SymColor{\mathbb{\zzP}_{\NUMpoly}}}%\mathrm{poly}

\newcommand*{\polySMvecAB}[1]{\SymColor{ \boldsymbol{\zzM}^{(#1)}  }} %_{\mathrm{v}}
\newcommand*{\blockpolySMvecAB}[3]{\SymColor{ \boldsymbol{\zzM}_{{#2},{#3}}^{(#1)}  }} %\mathrm{v},

\newcommand*{\NUMpoly}{\SymColor{\zzN}}
\newcommand*{\IDpoly}{\SymColor{\zzk}}
\newcommand*{\IDbpoly}{\SymColor{\zzk^{\prime}}}

%%%%Definition of SMP systems%%%%%%%%Definition of SMP systems%%%%%%%%Definition of SMP systems%%%%
%%%%Definition of SMP systems%%%%%%%%Definition of SMP systems%%%%%%%%Definition of SMP systems%%%%

%%%%exponential stability%%%%
\newcommand*{\EMSrate}{\SymColor{\beta}}
\newcommand*{\EMScoef}{\SymColor{\alpha}}
\newcommand*{\exEMSrate}{\SymColor{\widetilde{\beta}}}
\newcommand*{\exEMScoef}{\SymColor{\widetilde{\alpha}}}
%%%%exponential stability%%%%

%%%%Definition of expanded systems%%%%%%%%Definition of expanded systems%%%%%%%%Definition of expanded systems%%%%
%%%%Definition of expanded systems%%%%%%%%Definition of expanded systems%%%%%%%%Definition of expanded systems%%%%
\newcommand*{\DimexX}{\SymColor{\widetilde{\zzn}}}%{\SymColor{(\DimX(\DimX+1)/2)}}

\newcommand*{\exState}[1]{\SymColor{\widetilde{\boldsymbol{\zzx}}}_{#1}}

\newcommand*{\exTIUnc}{\SymColor{\widetilde{\boldsymbol{\theta}}}}
\newcommand*{\exTVUnc}[1]{\exTIUnc_{#1}}

\newcommand*{\ImagePolyW}{\widetilde{\DomTVUnc}}

\newcommand*{\exTICLMat}[1]{    \SymColor{\boldsymbol{\zzF}( \exTIUnc     , #1 )}}
\newcommand*{\DETAILexCLMat}[2]{\SymColor{\boldsymbol{\zzF}( \exTVUnc{#2} , #1 )}}
\newcommand*{\exCLMat}[1]{\DETAILexCLMat{#1}{\MyT}}

\newcommand*{\expolyCLMat}[2]{\SymColor{\boldsymbol{\zzF}}^{ \SymColor{(} #1 \SymColor{)} }  \SymColor{(} #2 \SymColor{)}  }

\newcommand*{\expolyAAMat}[1]{\SymColor{\boldsymbol{\zzF}}_{\SymColor{\mathrm{aa}}}^{ \SymColor{(} #1 \SymColor{)} }  }
\newcommand*{\expolyABMat}[1]{\SymColor{\boldsymbol{\zzF}}_{\SymColor{\mathrm{ab}}}^{ \SymColor{(} #1 \SymColor{)} }  }
\newcommand*{\expolyBAMat}[1]{\SymColor{\boldsymbol{\zzF}}_{\SymColor{\mathrm{ba}}}^{ \SymColor{(} #1 \SymColor{)} }  }
\newcommand*{\expolyBBMat}[1]{\SymColor{\boldsymbol{\zzF}}_{\SymColor{\mathrm{bb}}}^{ \SymColor{(} #1 \SymColor{)} }  }

\newcommand*{\DupMat}{\SymColor{\boldsymbol{\zzC}_{\mathrm{d}}}}
\newcommand*{\EliMat}{\SymColor{\boldsymbol{\zzC}_{\mathrm{e}}}}
\newcommand*{\Elimi}[1]{\SymColor{\mathcal{\zzC}} \SymColor{(}    #1\SymColor{)}}
\newcommand*{\NonArgElimi}{\SymColor{\mathcal{\zzC}} }

%%%%Definition of expanded systems%%%%%%%%Definition of expanded systems%%%%%%%%Definition of expanded systems%%%%
%%%%Definition of expanded systems%%%%%%%%Definition of expanded systems%%%%%%%%Definition of expanded systems%%%%

%%%%CMIs,QMIs,LMIs%%%%%%%%CMIs,QMIs,LMIs%%%%%%%%CMIs,QMIs,LMIs%%%%%%%%CMIs,QMIs,LMIs%%%%%%%%CMIs,QMIs,LMIs%%%%
%%%%CMIs,QMIs,LMIs%%%%%%%%CMIs,QMIs,LMIs%%%%%%%%CMIs,QMIs,LMIs%%%%%%%%CMIs,QMIs,LMIs%%%%%%%%CMIs,QMIs,LMIs%%%%

\newcommand*{\itemMultiPERSCMIs}{(C1.1)}%\ref{thm:MSstab_TISMP}
\newcommand*{\itemMultiPRSCMIs}{(C1.2)}

\newcommand*{\itemMultiPERSQMIs}{(C2.1)}%\ref{thm:stability_QMI}
\newcommand*{\itemMultiPRSQMIs}{(C2.2)}

\newcommand*{\itemMultiPERSSDPs}{(C3.1)}%\ref{thm:stability_QMIwithRankOne}
\newcommand*{\itemMultiPRSSDPs}{(C3.2)}

\newcommand*{\CMIterms}[5]{   \SymColor{\boldsymbol{\zzS}_{\mathrm{cmi}}^{(#1)}} \SymColor{(} #2 , #3 , #4 , #5 \SymColor{)}    }
\newcommand*{\QMIterms}[5]{   \SymColor{\boldsymbol{\zzS}_{\mathrm{qmi}}^{(#1)}} \SymColor{(} #2 , #3 , #4 , #5 \SymColor{)}    }
\newcommand*{\SDPLMIterms}[4]{\SymColor{\boldsymbol{\zzS}_{\mathrm{lmi}}^{(#1)}} \SymColor{(} #2 , #3 , #4  \SymColor{)}    }

\newcommand*{\QMILowLeftblock}[2]{  \SymColor{\boldsymbol{\zzF}}_{\SymColor{\mathrm{qmi}}}^{ \SymColor{(} #1 \SymColor{)} }\SymColor{(} #2 \SymColor{)}   }%{  \SymColor{\boldsymbol{M}}_{\SymColor{\mathrm{QMI}}}^{ \SymColor{(} #1 \SymColor{)} }\SymColor{(} #2 \SymColor{)}   }
\newcommand*{\LMILowLeftblock}[2]{  \SymColor{\boldsymbol{\zzF}}_{\SymColor{\mathrm{lmi}}}^{ \SymColor{(} #1 \SymColor{)} }\SymColor{(} #2 \SymColor{)}   }

\newcommand*{\dummyHHfunc}[1]{\SymColor{\boldsymbol{\zzF}_{\mathrm{hh}}(}#1\SymColor{)} }
\newcommand*{\dummyHMfunc}[1]{\SymColor{\boldsymbol{\zzF}_{\mathrm{hl}}(}#1\SymColor{)} }
\newcommand*{\dummyMHfunc}[1]{\SymColor{\boldsymbol{\zzF}_{\mathrm{lh}}(}#1\SymColor{)} }
\newcommand*{\dummyMMfunc}[1]{\SymColor{\boldsymbol{\zzF}_{\mathrm{ll}}(}#1\SymColor{)} }

\newcommand*{\expolyVMat}[1]{\SymColor{\boldsymbol{\zzP}}^{ \SymColor{(} #1 \SymColor{)} }  }
\newcommand*{\UNIexpolyVMat}{\SymColor{\boldsymbol{\zzP}} }
\newcommand*{\LYAPexpolyVMat}{\SymColor{\boldsymbol{\zzP}_{\ast}} }
\newcommand*{\exAddMat}{\SymColor{\boldsymbol{\zzG}}}

\newcommand*{\AIMexpolyVMat}[1]{\SymColor{\boldsymbol{\zzQ}}^{ \SymColor{(} #1 \SymColor{)} }  }
\newcommand*{\UNIAIMexpolyVMat}{\SymColor{\boldsymbol{\zzQ}} }

\newcommand*{\AIMFBgain}{\SymColor{\boldsymbol{\zzL}}}
\newcommand*{\AddInvMat}{\SymColor{\boldsymbol{\zzH}}}

\newcommand*{\BlockHMMat}[2]{\SymColor{\boldsymbol{\zzZ}}_{#1,#2}}

%%%%CMIs,QMIs,LMIs%%%%%%%%CMIs,QMIs,LMIs%%%%%%%%CMIs,QMIs,LMIs%%%%%%%%CMIs,QMIs,LMIs%%%%%%%%CMIs,QMIs,LMIs%%%%
%%%%CMIs,QMIs,LMIs%%%%%%%%CMIs,QMIs,LMIs%%%%%%%%CMIs,QMIs,LMIs%%%%%%%%CMIs,QMIs,LMIs%%%%%%%%CMIs,QMIs,LMIs%%%%

%%%%iterative SDP%%%%%%%%iterative SDP%%%%%%%%iterative SDP%%%%%%%%iterative SDP%%%%%%%%iterative SDP%%%%
%%%%iterative SDP%%%%%%%%iterative SDP%%%%%%%%iterative SDP%%%%%%%%iterative SDP%%%%%%%%iterative SDP%%%%

\newcommand*{\smallmarginLMIs}{\SymColor{\eta}}

\newcommand*{\rankoneErr}[1]{\SymColor{\varepsilon( }  #1  \SymColor{)}  }
\newcommand*{\rankoneApproxErr}[2]{\SymColor{\widehat{\varepsilon}( }  #1 , #2  \SymColor{)}  }

\newcommand*{\dummyPreHMMat}{\SymColor{\boldsymbol{\zzZ}^{\prime}}}
\newcommand*{\dummyIteHMMat}[1]{\SymColor{\boldsymbol{\zzZ}^{(#1)}}}

\newcommand*{\IDite}{\SymColor{\ell}}

\newcommand*{\approxdummyHMMat}{\SymColor{\widehat{\boldsymbol{\zzZ}}_{\ast}}}
\newcommand*{\optdummyHMMat}{   \SymColor{         \boldsymbol{\zzZ}_{\ast}} }
\newcommand*{\dummyHMMat}{\SymColor{\boldsymbol{\zzZ}}}

\newcommand*{\UBdummyHMMat}{\SymColor{\zzZ_{\mathrm{ub}}}}

\newcommand*{\MultiSDPterminationVal}{\SymColor{\delta}}

%%%%iterative SDP%%%%%%%%iterative SDP%%%%%%%%iterative SDP%%%%%%%%iterative SDP%%%%%%%%iterative SDP%%%%
%%%%iterative SDP%%%%%%%%iterative SDP%%%%%%%%iterative SDP%%%%%%%%iterative SDP%%%%%%%%iterative SDP%%%%

%%%%Demo of SMP systems%%%%%%%%Demo of SMP systems%%%%%%%%Demo of SMP systems%%%%%%%%Demo of SMP systems%%%%
%%%%Demo of SMP systems%%%%%%%%Demo of SMP systems%%%%%%%%Demo of SMP systems%%%%%%%%Demo of SMP systems%%%%

\newcommand*{\polyvecAB}[1]{\SymColor{\boldsymbol{\zzv}^{(#1)}}} 
\newcommand*{\polyRandvecAB}[1]{\SymColor{\boldsymbol{\zzv}_{\MyT}^{(#1)}}}

\newcommand*{\ConstMeanvecAB}{\SymColor{\boldsymbol{\mu}}}
\newcommand*{\polyMeanvecAB}[1]{\SymColor{\boldsymbol{\mu}^{(#1)}}} 
\newcommand*{\polyCovvecAB}[1]{\SymColor{\boldsymbol{\Sigma}^{(#1)}}}

\newcommand*{\anotherTIUnc}{\overline{\TIUnc}}
\newcommand*{\anothervecTIAB}[1]{\overline{\NonArgvecAB{#1}}  ( \anotherTIUnc   )  }
\newcommand*{\anotherDomTVUnc}{\overline{\DomTVUnc}}
\newcommand*{\anotherDimTVUnc}{\overline{\DimTVUnc}}

%%%%Demo of SMP systems%%%%%%%%Demo of SMP systems%%%%%%%%Demo of SMP systems%%%%%%%%Demo of SMP systems%%%%
%%%%Demo of SMP systems%%%%%%%%Demo of SMP systems%%%%%%%%Demo of SMP systems%%%%%%%%Demo of SMP systems%%%%

%%%%for proofs%%%%%%%%for proofs%%%%%%%%for proofs%%%%%%%%for proofs%%%%%%%%for proofs%%%%%%%%for proofs%%%%
%%%%for proofs%%%%%%%%for proofs%%%%%%%%for proofs%%%%%%%%for proofs%%%%%%%%for proofs%%%%%%%%for proofs%%%%

\newcommand*{\tempAmat}{\SymColor{\boldsymbol{\zzA}}}
\newcommand*{\tempBmat}{\SymColor{\boldsymbol{\zzB}}}
\newcommand*{\tempABvec}{\SymColor{\boldsymbol{\zzv}}}
\newcommand*{\tempSMABvec}[2]{\SymColor{\boldsymbol{\zzM}}_{#1,#2}}

\newcommand*{\coefASchwarz}{\NotationVec_{\mathrm{a}}}
\newcommand*{\coefBSchwarz}{\NotationVec_{\mathrm{b}}}

\newcommand*{\PFStateIni}[2]{            {\State{ 0}^{(#1,#2)}}    }   %\SymColor{\boldsymbol{\zzx}}_{0,#1,#2}}
\newcommand*{\vhxx}[2]{                {\exState{ 0}^{(#1,#2)}}   }%     \SymColor{\widetilde{\boldsymbol{\zzx}}_{0,#1,#2}}}
\newcommand*{\SpetialexStatevhxx}[3]{   \exState{#1}^{(#2,#3)}  } %   {\exState{#1}} |_{ {\exState{0}}={\vhxx{#2}{#3}}  }     }
\newcommand*{\weightPFexState}[2]{\SymColor{\zzw}_{#1,#2}}

\newcommand*{\pfCovNonDiagVal}[1]{\SymColor{\sigma}_{#1}}

\newcommand*{\CMIstrictCoef}{\SymColor{\gamma}}
\newcommand*{\BRIEFexCLMat}[1]{\SymColor{\boldsymbol{\zzF}}_{#1}}

%%%%for proofs%%%%%%%%for proofs%%%%%%%%for proofs%%%%%%%%for proofs%%%%%%%%for proofs%%%%%%%%for proofs%%%%
%%%%for proofs%%%%%%%%for proofs%%%%%%%%for proofs%%%%%%%%for proofs%%%%%%%%for proofs%%%%%%%%for proofs%%%%

\newcommand*{\zzR}{R}
\newcommand*{\zzI}{I}
\newcommand*{\zzp}{p}

\newcommand*{\zzi}{i}
\newcommand*{\zzj}{j}

\newcommand*{\zzd}{d}

\newcommand*{\zzn}{n}
\newcommand*{\zzm}{m}

\newcommand*{\zzt}{t}
\newcommand*{\zzs}{s}

\newcommand*{\zzx}{x}
\newcommand*{\zzu}{u}

\newcommand*{\zzK}{K}
\newcommand*{\zzA}{A}
\newcommand*{\zzB}{B}

\newcommand*{\zzS}{S}
\newcommand*{\zzN}{N}

\newcommand*{\zzM}{M}
\newcommand*{\zzk}{k}

\newcommand*{\zzF}{F}
\newcommand*{\zzC}{C}

\newcommand*{\zzP}{P}
\newcommand*{\zzG}{G}
\newcommand*{\zzQ}{Q}
\newcommand*{\zzL}{L}
\newcommand*{\zzH}{H}
\newcommand*{\zzZ}{Z}

\newcommand*{\zzv}{v}

\newcommand*{\zzw}{w}

\newcommand*{\zzy}{y}
\newcommand*{\zzY}{Y}
\newcommand*{\zzX}{X}

%%%%%%%%%%%%%%%%%%%%%%%%%%%%%%%%%%%%%%%%%%%%%%%%%%%%%%%%%%%%%%%%%%%%%%%%%%%%%%%%%%%%%%%%%%%%%%%%%%%%%%%%%%%%%%%%%%%%%%%%%%%%%%%%%%%%%%%
%%%%%%%%%%%%%%%%%%%%%%%%%%%%%%%%%%%%%%%%%%%%%%%%%%%%%%%%%%%%%%%%%%%%%%%%%%%%%%%%%%%%%%%%%%%%%%%%%%%%%%%%%%%%%%%%%%%%%%%%%%%%%%%%%%%%%%%
%%%%%%%%%%%%%%%%%%%%%%%%%%%%%%%%%%%%%%%%%%%%%%%%%%%%%%%%%%%%%%%%%%%%%%%%%%%%%%%%%%%%%%%%%%%%%%%%%%%%%%%%%%%%%%%%%%%%%%%%%%%%%%%%%%%%%%%
\iffalse\section{abstract}\fi
\begin{abstract}
	This paper presents a new paradigm to stabilize uncertain stochastic linear systems.
	Herein, second moment polytopic (SMP) systems are proposed that generalize systems with both uncertainty and randomness.
	The SMP systems are characterized by second moments of the stochastic system matrices and the uncertain parameters.
	Further, a fundamental theory for guaranteeing stability of the SMP systems is established.
	It is challenging to analyze the SMP systems owing to both the uncertainty and randomness.
	An idea to overcome this difficulty is to expand the SMP systems and exclude the randomness.
	Because the expanded systems contain only the uncertainty, their stability can be analyzed via robust stability theory.
	The stability of the expanded systems is equivalent to statistical stability of the SMP systems.	
	These facts provide sufficient conditions for the stability of the SMP systems as linear matrix inequalities (MIs).
	In controller design for the SMP systems, the linear MIs reduce to cubic MIs whose solutions correspond to feedback gains.
	The cubic MIs are transformed into simpler quadratic MIs that can be solved using optimization techniques. 
	Moreover, solving such non-convex MIs is relaxed into the iteration of a convex optimization.
	Solutions to the iterative optimization provide feedback gains that stabilize the SMP systems.
	As demonstrated here, the SMP systems represent linear dynamics with uncertain mean and covariance and other existing systems such as independently identically distributed dynamics and random polytopes.
	Finally, a numerical simulation shows the effectiveness of the proposed method.
\end{abstract}
\section{Introduction}\label{sec_intro}

Uncertainties and randomness are present in various dynamical systems such as semi-autonomous vehicles with individual differences among human drivers \cite{Saleh13} and nanoscale receivers with  manufacturing variations \cite{ItoACCESS19}.
Stabilization of such systems is a crucial task for implementation of control systems in the real world.
Thus, this study focuses on stability analysis and controller design for uncertain stochastic linear systems.

Previously, various robust control approaches have been proposed to treat uncertainties in system dynamics.
Stability and control of uncertain systems have been discussed using matrix inequalities \cite{Boyd94,Oliveira99,Grman05,Peres05,Zhang10}, boundary mapping \cite{Mutlu18,Rick19}, and variational methods \cite{Okura15}.
Probabilistic methods  have guaranteed stability of uncertain systems in a probabilistic sense \cite{TempoCDC96,Polyak01,Fujisaki03}.
If the uncertainties are probabilistic, they are treated as time-invariant (TI) stochastic parameters \cite{Fisher09,ItoCDC16}.
Such a representation of uncertainties is efficient for the improvement of statistical control performance.
Controller design associated with stochastic parameters reduces to stochastic optimal control problems \cite{Fisher09,Templeton12,Bhattacharya14,ItoCDC16,Wan2017,Shen2017}, for which efficient tools such as polynomial chaos expansions \cite{Xiu02} have been established.

Further, various types of time-varying (TV) stochastic systems that are associated with stability analysis and controller design have been analyzed.
TV stochastic parameters can express multiplicative noises included in system dynamics \cite{Koning82} rather than external additive noises \cite{Anderson89}.
Independently identically distributed (i.i.d.) stochastic parameters are common because their Markov properties are tractable in control problems.
Stability of systems with i.i.d. parameters has been guaranteed via several approaches such as using stochastic Riccati equations \cite{Koning82,ItoACC16} and Kronecker products \cite{Hibey96,Ogura13}.
Continuous-time cases with Wiener processes have been analyzed \cite{Luo19TAC,Zhang12TAC}.
Multiple control problems have been addressed, such as optimal control \cite{Koning82}, variance suppression \cite{Fujimoto11FTOP,Fujimoto11ITOP}, risk-sensitive control \cite{ItoACC16,ItoTAC19}.

A crucial challenge is to handle combination of uncertainties and stochastic parameters, which increases the expressiveness of systems.
Stability and control of random polytopes have been analyzed based on sampling-based matrix inequalities \cite{HosoeTAC18} and S-variable approaches \cite{HosoeAutoma20}.
Other systems involving such combinations have also been analyzed, e.g., \cite{Tabarraie16,Gershon18}.
Some variations of stochastic systems are summarized in \cite{Mesbah16}.
Although these promising results have successfully treated the uncertainties and stochastic parameters simultaneously, system representations are still limited to specific types.

To generalize systems involving both uncertainties and stochastic randomness,
this study establishes a fundamental theory to stabilize various types of uncertain stochastic linear systems in a unified manner.
The systems are generalized as second moment polytopic (SMP) systems.
Herein, stability conditions and controller design methods for the SMP systems are derived.
The main contributions of this study are summarized as follows.

\begin{enumerate}
	\item
	We present a novel class of uncertain stochastic linear systems, called SMP systems, which can represent various types of systems with both uncertainties and stochastic randomness.
	This class is easy to utilize because it consists of a second moment of system matrices without requiring the mean dynamics.  
	Moreover, a compression operator is employed to reduce the dimensions of the expanded systems. 
	
	\item
	We show that statistical stability of the SMP systems is equivalent to stability of their expanded systems.
	The proposed expanded systems are included in deterministic polytopes, which are compatible with existing methods such as \cite{Oliveira99}.
	This simplifies the stability analysis for the SMP systems.
		
	\item 
	We derive sufficient conditions for the statistical stability of SMP systems.
	The conditions are expressed by linear matrix inequalities (LMIs) if the systems are autonomous or controllers are given.
	
	\item
	We propose a method to design linear feedback controllers that stabilize the SMP system.
	It is first shown that solutions to quadratic matrix inequalities (QMIs) are the feedback gains.
	Next, we relax solving the QMIs as an iteration of solving a semidefinite program (SDP) that is convex and easy to solve.

\end{enumerate}

This paper is a substantially extended version of our conference papers \cite{ItoIFACWC20,ItoCDC20}.
Although only random polytope systems are considered in \cite{ItoIFACWC20,ItoCDC20}, this study treats the SMP systems that generalize uncertain stochastic systems.
Moreover, whereas the paper \cite{ItoCDC20} tackles solving a non-convex program, this study presents a method to relax such a non-convex program as an iteration of a convex one. 
A novel analysis and demonstration are presented to show the contributions of this study.
Moreover, important theoretical points have been reviewed to improve the technical soundness and readability.

The remainder of this paper is organized as follows.
Section \ref{sec_notation} describes the notation used in this paper.
In Section \ref{sec_problem}, we propose the SMP systems associated with two main problems.
Our approach and solutions to the main problems are presented in Section \ref{sec_method}.
Section \ref{sec_sim} demonstrates the applicability and effectiveness of the proposed method.
Finally, Section \ref{sec_conclusion} concludes this study.

%%%%%%%%%%%%%%%%%%%%%%%%%%%%%%%%%%%%%%%%%%%%%%%%%%%%%%%%%%%%%%%%%%%%%%%%%%%%%%%%%%%%%%%%%%%%%%%%%%%%%%%%%%%%%%%%%%%%%%%%%%%%%%%%%%%%%%%
%%%%%%%%%%%%%%%%%%%%%%%%%%%%%%%%%%%%%%%%%%%%%%%%%%%%%%%%%%%%%%%%%%%%%%%%%%%%%%%%%%%%%%%%%%%%%%%%%%%%%%%%%%%%%%%%%%%%%%%%%%%%%%%%%%%%%%%
%%%%%%%%%%%%%%%%%%%%%%%%%%%%%%%%%%%%%%%%%%%%%%%%%%%%%%%%%%%%%%%%%%%%%%%%%%%%%%%%%%%%%%%%%%%%%%%%%%%%%%%%%%%%%%%%%%%%%%%%%%%%%%%%%%%%%%%
\section{Notation}\label{sec_notation}

This paper uses the following notation.
\begin{itemize}
	
	\item
	$\mathbb{R}^{\DimANotation \times \DimBNotation}$: the set of $\DimANotation \times \DimBNotation$ real-valued matrices 
	
	\item
	${\SetSymMat{\DimANotation}}$: the set of $\DimANotation \times \DimANotation$ real-valued symmetric matrices

	\item 
	$\Identity{\DimANotation}$: the $\DimANotation \times \DimANotation$ identity matrix
	\item
	$\El{\NotationVec}{\IDNotation}$: the $\IDNotation$-th component of a vector $\NotationVec \in \mathbb{R}^{\DimANotation}$
	\item
	$\El{\NotationMat}{\IDNotation,\IDbNotation}$: the component in the $\IDNotation$-th row and $\IDbNotation$-th column  of a matrix $\NotationMat \in \mathbb{R}^{\DimANotation \times \DimBNotation}$
	
	\item
	$\El{\NotationMat}{\vectorWildCard,\IDbNotation}$: the $\IDbNotation$-th column vector of a matrix $\NotationMat \in \mathbb{R}^{\DimANotation \times \DimBNotation}$
	
	\item
	$\VEC{\NotationMat}:=[ 
	\El{\NotationMat}{1,1} , \dots, \El{\NotationMat}{\DimANotation,1} , 
	\El{\NotationMat}{1,2} , \dots, \El{\NotationMat}{\DimANotation,2} , 
	\dots ,
	$ $	
	\El{\NotationMat}{1,\DimBNotation} , \dots, \El{\NotationMat}{\DimANotation,\DimBNotation}
	]^{\MyTRANSPO} $: 
	the vectorization of the components of a matrix $\NotationMat \in \mathbb{R}^{\DimANotation \times \DimBNotation}$
	
	\item
	$\VECH{\NotationSquMat}:=[ 
	\El{\NotationSquMat}{1,1} , \dots, \El{\NotationSquMat}{\DimANotation,1} , 
	\El{\NotationSquMat}{2,2} , \dots, \El{\NotationSquMat}{\DimANotation,2} , 
	\dots ,
	$ $
	\El{\NotationSquMat}{\IDbNotation,\IDbNotation} , \dots, \El{\NotationSquMat}{\DimANotation,\IDbNotation} , 
	\dots ,
	\El{\NotationSquMat}{\DimANotation,\DimANotation}	]^{\MyTRANSPO} $: 
	the half vectorization of the lower triangular components of a square matrix $\NotationSquMat \in \mathbb{R}^{\DimANotation \times \DimANotation}$

	\item
	$\NotationMat_{\NotationA} \otimes \NotationMat_{\NotationB} \in \mathbb{R}^{\DimANotation_{\NotationA} \DimANotation_{\NotationB} \times \DimBNotation_{\NotationA} \DimBNotation_{\NotationB} }$: the Kronecker product of matrices $\NotationMat_{\NotationA} \in \mathbb{R}^{\DimANotation_{\NotationA} \times \DimBNotation_{\NotationA} }$ and $\NotationMat_{\NotationB} \in \mathbb{R}^{\DimANotation_{\NotationB} \times \DimBNotation_{\NotationB} }$, given by
	\begin{align}
	%& 
	\NotationMat_{\NotationA} \otimes \NotationMat_{\NotationB} =
	\begin{bmatrix}
	\El{\NotationMat_{\NotationA}}{1,1} \NotationMat_{\NotationB} & \hdots & \El{\NotationMat_{\NotationA}}{1,\DimBNotation_{\NotationA}} \NotationMat_{\NotationB} \\
	\vdots & \ddots & \vdots \\
	\El{\NotationMat_{\NotationA}}{\DimANotation_{\NotationA},1} \NotationMat_{\NotationB} & \hdots & \El{\NotationMat_{\NotationA}}{\DimANotation_{\NotationA},\DimBNotation_{\NotationA}} \NotationMat_{\NotationB} 
	\end{bmatrix}	
	\nonumber
	\end{align}	
	
	\item
	$\MyRank{\NotationMat}$: the rank of a matrix $\NotationMat \in \mathbb{R}^{\DimANotation \times \DimBNotation}$

	\item
	$\NotationSymMat \succ 0$ (resp. $\prec 0$): the positive (resp. negative) definiteness of a symmetric{\footnote{In this paper, a positive/negative definite/semidefinite matrix means a positive/negative definite/semidefinite symmetric matrix.}} matrix $\NotationSymMat \in {\SetSymMat{\DimANotation}}$
	\item
	$\NotationSymMat \succeq 0$ (resp. $\preceq 0$): the positive (resp. negative) semidefiniteness of a symmetric matrix $\NotationSymMat \in {\SetSymMat{\DimANotation}}$

	\item
	$\iEig{\IDEl}{\NotationSymMat}$: $\IDEl$-th eigenvalue of a symmetric matrix $\NotationSymMat \in {\SetSymMat{\DimANotation}}$ such that $\iEig{1}{\NotationSymMat} \geq \iEig{2}{\NotationSymMat} \geq \dots \geq \iEig{\DimANotation}{\NotationSymMat}$ 
	
	\item
	$\iEigVec{\IDEl}{\NotationSymMat}$: $\IDEl$-th unit eigenvector corresponding to $\iEig{\IDEl}{\NotationSymMat}$ of a symmetric matrix $\NotationSymMat \in {\SetSymMat{\DimANotation}}$, i.e., 
	$\NotationSymMat \iEigVec{\IDEl}{\NotationSymMat} = \iEig{\IDEl}{\NotationSymMat} \iEigVec{\IDEl}{\NotationSymMat}$
	and $ \iEigVec{\IDEl}{\NotationSymMat}^{\MyTRANSPO} \iEigVec{\IDEl}{\NotationSymMat} =1$
	
	\item
	${\Expect[]{   \NotationVec( {\NonArgvecAB{}} )  }}$:
	the expectation $\int  \NotationVec( {\NonArgvecAB{}} ) \PDF{\NonArgvecAB{}} \mathrm{d} {\NonArgvecAB{}} $ 
	of a function  $\NotationVec( {\NonArgvecAB{}} )$ 
	with respect to a random vector ${\NonArgvecAB{}} $ 
	obeying a probability density function (PDF) $\PDF{\NonArgvecAB{}}$

	\item
	${\CondExpectTI[]{    \NotationVec( {\vecTIAB{}}  )   }}$:
	the conditional expectation $\int   \NotationVec( {\vecTIAB{}} )   {\CondPDF{\NonArgvecAB{}}{\TIUnc}} \mathrm{d} {\NonArgvecAB{}} $ 
	with respect to $ {\vecTIAB{}}  $ 
	obeying a conditional PDF $ {\CondPDF{\NonArgvecAB{}}{\TIUnc}}$
	given ${\TIUnc}$
	
	\item
	${\CondExpectAllTV[]{    \NotationVec( {\NonArgvecAB{}}({\TVUnc{0}},{\TVUnc{1}},\dots) )   }{\MybT}{\MyT}}$:
	the conditional expectation $\int  \NotationVec( {\NonArgvecAB{}}({\TVUnc{0}},{\TVUnc{1}},\dots) ) {\CondPDF{\NonArgvecAB{}}{  {\TVUnc{0}},{\TVUnc{1}},\dots  }} \mathrm{d} {\NonArgvecAB{}} $ 
	with respect to ${\NonArgvecAB{}}({\TVUnc{0}},{\TVUnc{1}},\dots) $ 
	obeying a conditional PDF $ {\CondPDF{\NonArgvecAB{}}{ {\TVUnc{0}},{\TVUnc{1}},\dots }}$
	given a sequence $({\TVUnc{0}},{\TVUnc{1}},\dots)$
	
	\item
	${\CondCovTI[]{    \NotationVec( {\vecTIAB{}} )   }}
	:={\CondExpectTI[]{    \NotationVec( {\vecTIAB{}} )   \NotationVec( {\vecTIAB{}} )^{\MyTRANSPO}   }}
	-
	{\CondExpectTI[]{  \NotationVec( {\vecTIAB{}} )    }}
	{\CondExpectTI[]{  \NotationVec( {\vecTIAB{}} )    }}^{\MyTRANSPO}  
	$:
	the conditional covariance 
	given ${\TIUnc}$

\end{itemize}

%%%%%%%%%%%%%%%%%%%%%%%%%%%%%%%%%%%%%%%%%%%%%%%%%%%%%%%%%%%%%%%%%%%%%%%%%%%%%%%%%%%%%%%%%%%%%%%%%%%%%%%%%%%%%%%%%%%%%%%%%%%%%%%%%%%%%%%
%%%%%%%%%%%%%%%%%%%%%%%%%%%%%%%%%%%%%%%%%%%%%%%%%%%%%%%%%%%%%%%%%%%%%%%%%%%%%%%%%%%%%%%%%%%%%%%%%%%%%%%%%%%%%%%%%%%%%%%%%%%%%%%%%%%%%%%
%%%%%%%%%%%%%%%%%%%%%%%%%%%%%%%%%%%%%%%%%%%%%%%%%%%%%%%%%%%%%%%%%%%%%%%%%%%%%%%%%%%%%%%%%%%%%%%%%%%%%%%%%%%%%%%%%%%%%%%%%%%%%%%%%%%%%%%
\section{Second moment polytopic systems with problem settings}\label{sec_problem}

%%%%%%%%%%%%%%%%%%%%%%%%%%%%%%%%%%%%%%%%%%%%%%%%%%%%%%%%%%%%%%%%%%%%%%%%%%%%%%%%%%%%%%%%%%%%%%%%%%%%%%%%%%%%%%%%%%%%%%%%%%%%%%%%%%%%%%%
%%%%%%%%%%%%%%%%%%%%%%%%%%%%%%%%%%%%%%%%%%%%%%%%%%%%%%%%%%%%%%%%%%%%%%%%%%%%%%%%%%%%%%%%%%%%%%%%%%%%%%%%%%%%%%%%%%%%%%%%%%%%%%%%%%%%%%%
%%%%%%%%%%%%%%%%%%%%%%%%%%%%%%%%%%%%%%%%%%%%%%%%%%%%%%%%%%%%%%%%%%%%%%%%%%%%%%%%%%%%%%%%%%%%%%%%%%%%%%%%%%%%%%%%%%%%%%%%%%%%%%%%%%%%%%%
\subsection{Target systems described by second moment polytopes}

Consider the following uncertain stochastic linear system: 
\begin{align}
{\State{\MyT+1}} 
&\;
= {\DriftMat{\MyT}} {\State{\MyT}}
+ {\InMat{\MyT}} {\Input{\MyT}}
, \label{eq:def_sys}
\\
{\vecTVAB{\MyT}}
&:= \VEC{   [{\DriftMat{\MyT}}, {\InMat{\MyT}}]  }
\sim {\CondPDF{\NonArgvecAB{\MyT}}{\TVUnc{\MyT}}}
,
\end{align}
where ${\State{\MyT}} \in \mathbb{R}^{\DimX}$, ${\Input{\MyT}} \in \mathbb{R}^{\DimU}$, and  ${\TVUnc{\MyT}} \in \DomTVUnc \subset \mathbb{R}^{\DimTVUnc}$ denote the state, control input, and TV uncertain parameter for the discrete time $\MyT \in \{0, 1, 2, \dots \}$, respectively.
The initial state ${\State{0}}$ is deterministic, and the uncertain parameter ${\TVUnc{\MyT}}$ can be stochastic or deterministic.
Let ${\vecTVAB{\MyT}} \in  \mathbb{R}^{\DimX(\DimX+\DimU)}  $ be the vectorization of the stochastic system matrices ${\DriftMat{\MyT}} \in  \mathbb{R}^{\DimX \times \DimX}$ and ${\InMat{\MyT}} \in \mathbb{R}^{\DimX \times \DimU}$.
The stochastic parameter ${\vecTVAB{\MyT}}$ obeys a PDF ${\CondPDF{\NonArgvecAB{\MyT}}{\TVUnc{\MyT}}}$ independently with respect to $\MyT$.
The PDF is uncertain because it depends on the uncertain parameter ${\TVUnc{\MyT}}$.

To characterize this general uncertain stochastic system \eqref{eq:def_sys},
we propose the notion of a second moment polytope.
The notion focuses on the second moment of the stochastic parameter ${\vecTVAB{\MyT}}$ that is included in a polytope. 

%%%%%%%%%%%%%%%%%%%%%%%%%%%%%%%%%%%%%%%%%%%%%%%%%%%%%%%%%%%%%%%%%%%%%%%
%%%%%%%%%%%%%%%%%%%%%%%%%%%%%%%%%%%%%%%%%%%%%%%%%%%%%%%%%%%%%%%%%%%%%%%	
\iffalse\subsubsection{Second moment polytope}\fi
\begin{definition}[{\MyHighlight{Second moment polytope}}]\label{def:TVSMP}
The system \eqref{eq:def_sys} is said to be second moment polytopic (SMP) if 
there exist a positive integer $\NUMpoly$, vertices ${\polySMvecAB{\IDpoly}} \in  {\SetSymMat{\DimX(\DimX+\DimU)}}$ for $\IDpoly \in \{ 1,2,\dots,\NUMpoly\}$, and a function $\NonArgMapPolyW: \DomTVUnc \to \DomPolytope$
that satisfy
\begin{align}
&\forall \MyT,\;
\forall {\TVUnc{0}},{\TVUnc{1}} , {\dots} \in \DomTVUnc
,\;
\nonumber\\
&
\quad
{\CondExpectTV[\big]{   {\vecTVAB{\MyT}}{\vecTVAB{\MyT}^{\MyTRANSPO}}    }{\MyT}}
=
\sum_{\IDpoly=1}^{\NUMpoly}  {\El{\MapPolyW{\TVUnc{\MyT}}}{\IDpoly}}  {\polySMvecAB{\IDpoly}}
,
\label{eq:SMP_AA_AB_BB}
\end{align}
where the codomain $\DomPolytope$ is the $\NUMpoly$-dimensional set:
\begin{align}
\DomPolytope
& :=
\Big\{ 
\NonArgMapPolyW \in \mathbb{R}^{\NUMpoly} 
\Big|
\forall \IDpoly, 
{\El{\NonArgMapPolyW}{\IDpoly}}
\geq 0
,
\sum_{\IDpoly=1}^{\NUMpoly} {\El{\NonArgMapPolyW}{\IDpoly}}  = 1
\Big\}
.\label{eq:def_DomPolytope}
\end{align}
		
\end{definition}
%%%%%%%%%%%%%%%%%%%%%%%%%%%%%%%%%%%%%%%%%%%%%%%%%%%%%%%%%%%%%%%%%%%%%%%
%%%%%%%%%%%%%%%%%%%%%%%%%%%%%%%%%%%%%%%%%%%%%%%%%%%%%%%%%%%%%%%%%%%%%%%	

%%%%%%%%%%%%%%%%%%%%%%%%%%%%%%%%%%%%%%%%%%%%%%%%%%%%%%%%%%%%%%%%%%%%%%%
%%%%%%%%%%%%%%%%%%%%%%%%%%%%%%%%%%%%%%%%%%%%%%%%%%%%%%%%%%%%%%%%%%%%%%%	
\iffalse\subsubsection{Time-invariant/varying SMP}\fi
\begin{definition}[{\MyHighlight{Time-invariant/varying SMP}}]\label{def:TISMP}
	An SMP system is said to be TI SMP if 
	${\TVUnc{\MyT}}$ is TI, that is, ${\TVUnc{\MyT}}=\TIUnc$ holds for all $\MyT$ with a constant $\TIUnc$.		
	Otherwise, the system is said to be TV SMP.
\end{definition}
%%%%%%%%%%%%%%%%%%%%%%%%%%%%%%%%%%%%%%%%%%%%%%%%%%%%%%%%%%%%%%%%%%%%%%%
%%%%%%%%%%%%%%%%%%%%%%%%%%%%%%%%%%%%%%%%%%%%%%%%%%%%%%%%%%%%%%%%%%%%%%%	

We introduce one simple example of SMP systems while other various examples are derived in Section \ref{sec_variety_SMP}.
%%%%%%%%%%%%%%%%%%%%%%%%%%%%%%%%%%%%%%%%%%%%%%%%%%%%%%%%%%%%%%%%%%%%%%%
%%%%%%%%%%%%%%%%%%%%%%%%%%%%%%%%%%%%%%%%%%%%%%%%%%%%%%%%%%%%%%%%%%%%%%%	
\iffalse\subsubsection{Example}\fi
\begin{example}[{\MyHighlight{Simple example of SMP systems}}]\label{ex:SMP}
Suppose that ${\vecTVAB{\MyT}}$, $\DomTVUnc$, and $\NUMpoly$ satisfy
\begin{align}
	{\CondExpectTV{\vecTVAB{\MyT}}{\MyT}}
	&
	= \ConstMeanvecAB 
	,
	\\
	{\CondCovTV[\big]{ \vecTVAB{\MyT} }{\MyT}} 
	&
	= \sum_{\IDpoly =1}^{\DimTVUnc} {\El{\TVUnc{\MyT}}{\IDpoly}}  {\polyCovvecAB{\IDpoly}} 
	,
	\\
	\DomTVUnc	& \subseteq	\DomPolytope
	,
	\\
	\NUMpoly&=\DimTVUnc
	,  \label{eq:setting_uncCov_NUMpoly}
\end{align}		
where $\ConstMeanvecAB \in \mathbb{R}^{\DimX ( \DimX + \DimU) } $ and ${\polyCovvecAB{\IDpoly}} \in {\SetSymMat{\DimX ( \DimX + \DimU)}}$ are constants.
Then, the system \eqref{eq:def_sys} is SMP with
\begin{align}	
	{\MapPolyW{\TVUnc{\MyT}}}&={\TVUnc{\MyT}}
	,\label{eq:setting_uncCov_MapPolyW}
	\\
	{\polySMvecAB{\IDpoly}}&=\ConstMeanvecAB \ConstMeanvecAB^{\MyTRANSPO} + {\polyCovvecAB{\IDpoly}}
	.  
\end{align}	
\end{example}
%%%%%%%%%%%%%%%%%%%%%%%%%%%%%%%%%%%%%%%%%%%%%%%%%%%%%%%%%%%%%%%%%%%%%%%
%%%%%%%%%%%%%%%%%%%%%%%%%%%%%%%%%%%%%%%%%%%%%%%%%%%%%%%%%%%%%%%%%%%%%%%	

To discuss stability and controller design for the system \eqref{eq:def_sys}, the following assumptions are used throughout this paper.

%%%%%%%%%%%%%%%%%%%%%%%%%%%%%%%%%%%%%%%%%%%%%%%%%%%%%%%%%%%%%%%%%%%%%%%
%%%%%%%%%%%%%%%%%%%%%%%%%%%%%%%%%%%%%%%%%%%%%%%%%%%%%%%%%%%%%%%%%%%%%%%	
\iffalse\subsubsection{Stochastic parameters}\fi
\begin{assumption}[{\MyHighlight{Second moment polytope}}]\label{ass:StoParam}
	{~}
	\begin{enumerate}
		\item\label{item_ass_SMP}
		The system \eqref{eq:def_sys} is SMP.

		\item\label{item_ass_KnownSMPVertices}
		Vertices ${\polySMvecAB{\IDpoly}}$ and $\NUMpoly$ satisfying \eqref{eq:SMP_AA_AB_BB}
		are known
		although it is not required that the PDF ${\CondPDF{\NonArgvecAB{\MyT}}{\TVUnc{\MyT}}}$ is known.
	
		\item\label{item_ass_unknown_sto_param}	
		For any $\MyT$, the values of ${\TVUnc{\MyT}}$ and ${\vecTVAB{\MyT}}$ are unknown.

		\item\label{item_ass_independentPDF}	
		For any $\MyT$, the PDF ${\CondPDF{\NonArgvecAB{\MyT}}{\TVUnc{\MyT}}}$
		is Lebesgue measurable on $ \mathbb{R}^{\DimX(\DimX+\DimU)}$ and the following independence holds:
		\begin{align}
		&
		\forall \MybT\geq 1,\;
		{\CondPDF{  {\NonArgvecAB{0}},{\NonArgvecAB{1}}, \dots, {\NonArgvecAB{\MybT}} }{   {\TVUnc{0}}, {\TVUnc{1}}, \dots           }}
		=
		\prod_{\MyT=0}^{\MybT}
		{\CondPDF{\NonArgvecAB{\MyT}}{\TVUnc{\MyT}}}
		.
		\end{align}

	\end{enumerate}		
\end{assumption}
%%%%%%%%%%%%%%%%%%%%%%%%%%%%%%%%%%%%%%%%%%%%%%%%%%%%%%%%%%%%%%%%%%%%%%%
%%%%%%%%%%%%%%%%%%%%%%%%%%%%%%%%%%%%%%%%%%%%%%%%%%%%%%%%%%%%%%%%%%%%%%%	

To justify Assumption \ref{ass:StoParam} \ref{item_ass_SMP} and \ref{item_ass_KnownSMPVertices}, Section \ref{sec_variety_SMP} presents how to transform various systems into SMP forms.
Assumption \ref{ass:StoParam} \ref{item_ass_unknown_sto_param} and \ref{item_ass_independentPDF} are formal descriptions of the problem setting.

%%%%%%%%%%%%%%%%%%%%%%%%%%%%%%%%%%%%%%%%%%%%%%%%%%%%%%%%%%%%%%%%%%%%%%%%%%%%%%%%%%%%%%%%%%%%%%%%%%%%%%%%%%%%%%%%%%%%%%%%%%%%%%%%%%%%%%%
%%%%%%%%%%%%%%%%%%%%%%%%%%%%%%%%%%%%%%%%%%%%%%%%%%%%%%%%%%%%%%%%%%%%%%%%%%%%%%%%%%%%%%%%%%%%%%%%%%%%%%%%%%%%%%%%%%%%%%%%%%%%%%%%%%%%%%%
%%%%%%%%%%%%%%%%%%%%%%%%%%%%%%%%%%%%%%%%%%%%%%%%%%%%%%%%%%%%%%%%%%%%%%%%%%%%%%%%%%%%%%%%%%%%%%%%%%%%%%%%%%%%%%%%%%%%%%%%%%%%%%%%%%%%%%%
\subsection{Problem statements}

A linear feedback controller is applied to the system \eqref{eq:def_sys}:
\begin{align}
{\Input{\MyT}} &= - \FBgain {\State{\MyT}}
, \label{eq:def_FBcontroller}
\\
{\State{\MyT+1}} 
&=( {\DriftMat{\MyT}} - {\InMat{\MyT}} \FBgain ) {\State{\MyT}}
, \label{eq:def_CLsys}
\end{align}	
where $\FBgain \in \mathbb{R}^{ \DimU \times \DimX } $ is a feedback gain.
This study focuses on the following two types of statistical stability for the feedback system {\eqref{eq:def_CLsys}} that is SMP because of Assumption \ref{ass:StoParam}.

%%%%%%%%%%%%%%%%%%%%%%%%%%%%%%%%%%%%%%%%%%%%%%%%%%%%%%%%%%%%%%%%%%%%%%%
%%%%%%%%%%%%%%%%%%%%%%%%%%%%%%%%%%%%%%%%%%%%%%%%%%%%%%%%%%%%%%%%%%%%%%%	
\begin{definition}[{\MyHighlight{Robust mean-square stability}}]\label{def:RMSstable}
	The SMP system {\eqref{eq:def_CLsys}} is said to be robustly mean-square (MS) stable if
	\begin{align}
	%&
	\forall {\State{0}} \in  \mathbb{R}^{\DimX}
	,\;
	\forall {\TVUnc{0}},{\TVUnc{1}} , {\dots} \in \DomTVUnc
	,\;\;\;
	\lim_{\MyT \to \infty}
	{\CondExpectAllTV[\big]{  \|    {\State{\MyT}}    \|^{2}   }{\MybT}{\MyT}}
	=0
	.\label{eq:def_RMSstable}
	\end{align}
\end{definition}
%%%%%%%%%%%%%%%%%%%%%%%%%%%%%%%%%%%%%%%%%%%%%%%%%%%%%%%%%%%%%%%%%%%%%%%
%%%%%%%%%%%%%%%%%%%%%%%%%%%%%%%%%%%%%%%%%%%%%%%%%%%%%%%%%%%%%%%%%%%%%%%	
\begin{definition}[{\MyHighlight{Exponential robust mean-square stability}}]\label{def:ERMSstable}
	The SMP system {\eqref{eq:def_CLsys}} is said to be exponentially robustly MS stable if there exist  $\EMScoef \in (0,\infty)$ and $\EMSrate \in (0,1)$ such that 
	\begin{align}
	&\forall {\State{0}} \in  \mathbb{R}^{\DimX}
	,\;	
	\forall {\TVUnc{0}},{\TVUnc{1}} , {\dots} \in \DomTVUnc
	,\;
	\forall \MyT 
	,\;
	\VisibleColTwo{\qquad\qquad
	\nonumber\\&\qquad\qquad}
	\sqrt{
	{\CondExpectAllTV[\big]{  \|    {\State{\MyT}}    \|^{2}   }{\MybT}{\MyT}}
	}
	\leq \EMScoef  \|{\State{0}}\| \EMSrate^{\MyT}
	.\label{eq:def_ERMSstable}
	\end{align}
\end{definition}
%%%%%%%%%%%%%%%%%%%%%%%%%%%%%%%%%%%%%%%%%%%%%%%%%%%%%%%%%%%%%%%%%%%%%%%
%%%%%%%%%%%%%%%%%%%%%%%%%%%%%%%%%%%%%%%%%%%%%%%%%%%%%%%%%%%%%%%%%%%%%%%	

We state the following two main problems:

\textbf{\textit{Problem 1 (stability analysis).}}
Find necessary and/or sufficient conditions that the SMP system {\eqref{eq:def_CLsys}} is (exponentially) robustly MS stable for a given feedback gain $\FBgain$.

\textbf{\textit{Problem 2 (controller design).}}
Design a feedback gain $\FBgain$ such that the SMP system {\eqref{eq:def_CLsys}} is (exponentially) robustly MS stable.

%%%%%%%%%%%%%%%%%%%%%%%%%%%%%%%%%%%%%%%%%%%%%%%%%%%%%%%%%%%%%%%%%%%%%%%%%%%%%%%%%%%%%%%%%%%%%%%%%%%%%%%%%%%%%%%%%%%%%%%%%%%%%%%%%%%%%%%
%%%%%%%%%%%%%%%%%%%%%%%%%%%%%%%%%%%%%%%%%%%%%%%%%%%%%%%%%%%%%%%%%%%%%%%%%%%%%%%%%%%%%%%%%%%%%%%%%%%%%%%%%%%%%%%%%%%%%%%%%%%%%%%%%%%%%%%
%%%%%%%%%%%%%%%%%%%%%%%%%%%%%%%%%%%%%%%%%%%%%%%%%%%%%%%%%%%%%%%%%%%%%%%%%%%%%%%%%%%%%%%%%%%%%%%%%%%%%%%%%%%%%%%%%%%%%%%%%%%%%%%%%%%%%%%
\section{Proposed method}\label{sec_method}

%%%%%%%%%%%%%%%%%%%%%%%%%%%%%%%%%%%%%%%%%%%%%%%%%%%%%%%%%%%%%%%%%%%%%%%%%%%%%%%%%%%%%%%%%%%%%%%%%%%%%%%%%%%%%%%%%%%%%%%%%%%%%%%%%%%%%%%
%%%%%%%%%%%%%%%%%%%%%%%%%%%%%%%%%%%%%%%%%%%%%%%%%%%%%%%%%%%%%%%%%%%%%%%%%%%%%%%%%%%%%%%%%%%%%%%%%%%%%%%%%%%%%%%%%%%%%%%%%%%%%%%%%%%%%%%
%%%%%%%%%%%%%%%%%%%%%%%%%%%%%%%%%%%%%%%%%%%%%%%%%%%%%%%%%%%%%%%%%%%%%%%%%%%%%%%%%%%%%%%%%%%%%%%%%%%%%%%%%%%%%%%%%%%%%%%%%%%%%%%%%%%%%%%
\subsection{Overview}\label{sec_overview}

We solve Problems 1 and 2 associated with the SMP system {\eqref{eq:def_CLsys}}.
An analysis of the SMP system suffers from two factors: the uncertainty of $\TVUnc{\MyT}$ and the randomness of ${\vecTVAB{\MyT}}$ given $\TVUnc{\MyT}$.	
Our key idea to overcome this difficulty is to expand the SMP system so that the randomness is excluded.
The expanded system with only the uncertain $\TVUnc{\MyT}$ is included in a deterministic polytopic system.
Such an exclusion simplifies the stability analysis and controller design for the SMP system.
These details are described in Section \ref{sec_expandedsys}.

Section \ref{sec_stability} presents solutions to Problem 1. 
We show that the (exponential) robust MS stability of the SMP system reduces to stability of the expanded system.
Stability conditions of the expanded system can be derived based on existing results for deterministic polytopic systems.
We obtain LMI-based sufficient conditions that the SMP system is stable for a given feedback gain $\FBgain$.

Sections \ref{sec_control_QMIs} and \ref{sec_control_SDPs} provide solutions to Problem 2, starting from the solutions to Problem 1.
Unfortunately, the derived LMI-based conditions become cubic matrix inequalities (CMIs) if the feedback gain is not given but is to be designed.
In Section \ref{sec_control_QMIs}, we transform the CMIs into simpler QMIs.
The QMIs can be solved via some optimization techniques.
Moreover, in Section \ref{sec_control_SDPs}, the QMIs are relaxed as an iterative convex program because  they are still non-convex problems.
We show that the QMIs are equivalent to LMIs with a rank-one constraint.
Solutions to the constrained LMIs are approximately obtained via the iteration of a convex SDP.
Finally, the solutions provide stabilizing feedback gains $\FBgain$.

Figure \ref{fig:overview} illustrates an overview of the proposed method, which consists of certain theorems and corollaries.
Corollaries \ref{thm:solution_to_Problem1} and \ref{thm:solution_to_Problem2} summarize solutions to Problems 1 and 2, respectively.

\newcommand{\MyFigSIZEoverview}{0.63}

%%%%%%%%%%%%%%%%%%%%%%%%%%%%%%%%%%%%%%%%%%%%%%%%%%%%%%%%%%%%%%%%%%%%%%%%%%%%%%%%%%%%%%%%%%%%%%%%%%%%%%%%%%%%%%%%%%%%%
%%%%%%%%%%%%%%%%%%%%%%%%%%%%%%%%%%%%%%%%%%%%%%%%%%%%%%%%%%%%%%%%%%%%%%%%%%%%%%%%%%%%%%%%%%%%%%%%%%%%%%%%%%%%%%%%%%%%%
%%%%%%%%%%%%%%%%%%%%%%%%%%%%%%%%%%%%%%%%%%%%%%%%%%%%%%%%%%%%%%%%%%%%%%%%%%%%%%%%%%%%%%%%%%%%%%%%%%%%%%%%%%%%%%%%%%%%%
\begin{figure}[!t]
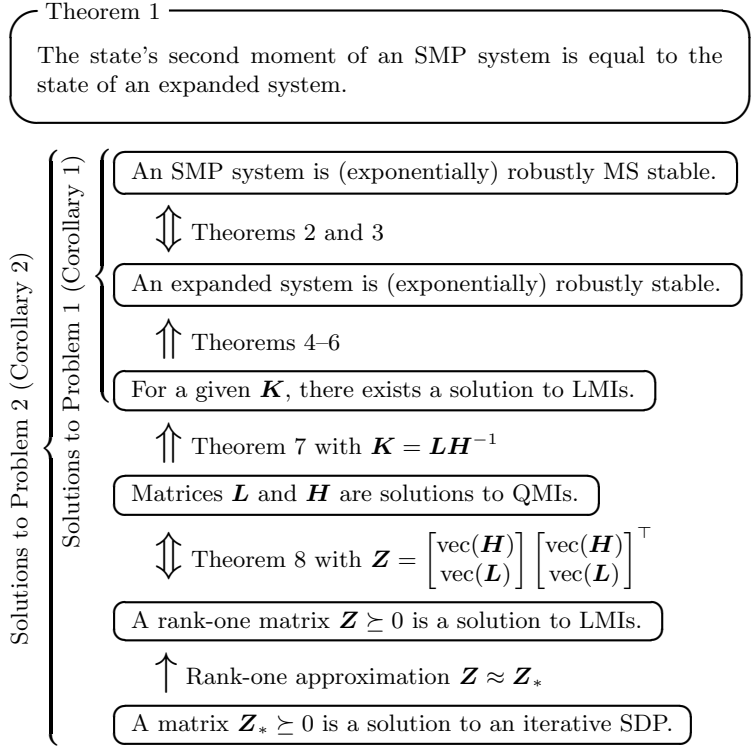
		
	\centering	
\begin{minipage}{\MyFigSIZEoverview\linewidth}
{\small

\begin{itembox}[l]{Theorem \ref{thm:sys_to_exsys}}	
The state's second moment of an SMP system is equal to the state of an expanded system. 
\end{itembox}

\vspace*{0.5\baselineskip}

\begin{minipage}{0.07\linewidth}
\rotatebox{90}{
\shortstack{Solutions to Problem 2 (Corollary \ref{thm:solution_to_Problem2})
\\ 
$\overbrace{\hspace*{11.5\linewidth}}$}
}
\end{minipage}%
\begin{minipage}{0.07\linewidth}
\rotatebox{90}{
\hspace*{3.6\linewidth}
\shortstack{Solutions to Problem 1 (Corollary \ref{thm:solution_to_Problem1}) 
\\ 
$ \hspace*{2.9\linewidth} \overbrace{\hspace*{4.9\linewidth}}$}
}
\end{minipage}%
\begin{minipage}{0.84\linewidth}

\Ovalbox{
An SMP system is (exponentially) robustly MS stable.
}

\vspace*{0.3\baselineskip}
{\LARGE$\quad\Updownarrow$}
Theorems \ref{thm:Equivarence_RMS} and \ref{thm:Equivarence_ERMS}
\vspace*{0.3\baselineskip}

\Ovalbox{
An expanded system is (exponentially) robustly stable.
}

\vspace*{0.3\baselineskip}
{\LARGE$\quad\Uparrow$}
Theorems \ref{thm:MSstab_TISMP}--\ref{thm:ERS_RS_TVexpanded}
\vspace*{0.3\baselineskip}

\Ovalbox{
For a given $\FBgain$, there exists a solution to LMIs.
}

\vspace*{0.3\baselineskip}
{\LARGE$\quad\Uparrow$}
Theorem \ref{thm:stability_QMI} with $\FBgain =  \AIMFBgain \AddInvMat^{-1} $
\vspace*{0.3\baselineskip}

\Ovalbox{
Matrices $\AIMFBgain$ and $\AddInvMat $ are solutions to QMIs.
}

\vspace*{0.3\baselineskip}
{\LARGE$\quad\Updownarrow$}
Theorem \ref{thm:stability_QMIwithRankOne} with
$\dummyHMMat
=
\begin{bmatrix}
\VEC{\AddInvMat} \\	 \VEC{\AIMFBgain} 
\end{bmatrix}
\begin{bmatrix}
\VEC{\AddInvMat} \\	 \VEC{\AIMFBgain} 
\end{bmatrix}^{\MyTRANSPO}$
\vspace*{0.3\baselineskip}

\Ovalbox{
A rank-one matrix $\dummyHMMat\succeq 0$ is a solution to LMIs.
}

\vspace*{0.3\baselineskip}
{\LARGE$\quad\uparrow$}
Rank-one approximation
$
\dummyHMMat \approx \optdummyHMMat
$
\vspace*{0.3\baselineskip}

\Ovalbox{
	A matrix $\optdummyHMMat \succeq 0$ is a solution to an iterative SDP.
}

\end{minipage}
}%\small
\end{minipage}
\caption{Overview of the proposed method.}
\label{fig:overview}
\end{figure}
%%%%%%%%%%%%%%%%%%%%%%%%%%%%%%%%%%%%%%%%%%%%%%%%%%%%%%%%%%%%%%%%%%%%%%%%%%%%%%%%%%%%%%%%%%%%%%%%%%%%%%%%%%%%%%%%%%%%%
%%%%%%%%%%%%%%%%%%%%%%%%%%%%%%%%%%%%%%%%%%%%%%%%%%%%%%%%%%%%%%%%%%%%%%%%%%%%%%%%%%%%%%%%%%%%%%%%%%%%%%%%%%%%%%%%%%%%%
%%%%%%%%%%%%%%%%%%%%%%%%%%%%%%%%%%%%%%%%%%%%%%%%%%%%%%%%%%%%%%%%%%%%%%%%%%%%%%%%%%%%%%%%%%%%%%%%%%%%%%%%%%%%%%%%%%%%%

%%%%%%%%%%%%%%%%%%%%%%%%%%%%%%%%%%%%%%%%%%%%%%%%%%%%%%%%%%%%%%%%%%%%%%%%%%%%%%%%%%%%%%%%%%%%%%%%%%%%%%%%%%%%%%%%%%%%%%%%%%%%%%%%%%%%%%%
%%%%%%%%%%%%%%%%%%%%%%%%%%%%%%%%%%%%%%%%%%%%%%%%%%%%%%%%%%%%%%%%%%%%%%%%%%%%%%%%%%%%%%%%%%%%%%%%%%%%%%%%%%%%%%%%%%%%%%%%%%%%%%%%%%%%%%%
%%%%%%%%%%%%%%%%%%%%%%%%%%%%%%%%%%%%%%%%%%%%%%%%%%%%%%%%%%%%%%%%%%%%%%%%%%%%%%%%%%%%%%%%%%%%%%%%%%%%%%%%%%%%%%%%%%%%%%%%%%%%%%%%%%%%%%%
%%%%%%%%%%%%%%%%%%%%%%%%%%%%%%%%%%%%%%%%%%%%%%%%%%%%%%%%%%%%%%%%%%%%%%%%%%%%%%%%%%%%%%%%%%%%%%%%%%%%%%%%%%%%%%%%%%%%%%%%%%%%%%%%%%%%%%%
\subsection{Key idea: Development of expanded systems}\label{sec_expandedsys} 

In this subsection, we propose the expanded system to the SMP system {\eqref{eq:def_CLsys}}.
First, let us introduce an operator for developing the expanded system.
%%%%%%%%%%%%%%%%%%%%%%%%%%%%%%%%%%%%%%%%%%%%%%%%%%%%%%%%%%%%%%%%%%%%%%%
%%%%%%%%%%%%%%%%%%%%%%%%%%%%%%%%%%%%%%%%%%%%%%%%%%%%%%%%%%%%%%%%%%%%%%%
\iffalse\subsubsection{Compression operator}\fi
\begin{definition}[{\MyHighlight{Compression operator $\NonArgElimi$}}]
	For any $\OpeNotationMat  \in \mathbb{R}^{\DimX^{2} \times \DimX^{2}}$, 
	the compression operator $\NonArgElimi: \mathbb{R}^{\DimX^{2} \times \DimX^{2}} \to \mathbb{R}^{\DimexX \times \DimexX}$ is defined by
	\begin{align}
	{\Elimi{   \OpeNotationMat   }}&:=\EliMat \OpeNotationMat  \DupMat	
	,
	\\
	\DimexX &:= \DimX(\DimX+1)/2 
	,\label{eq:def_DimexX}
	\end{align}	
	where the elimination matrix $\EliMat \in \mathbb{R}^{ \DimexX \times \DimX^{2} }$ and duplication matrix $\DupMat \in \mathbb{R}^{ \DimX^{2} \times \DimexX }$ are defined such that 
	$\EliMat \VEC{\OpeNotationbMat}=\VECH{\OpeNotationbMat}$ for any $\OpeNotationbMat \in \mathbb{R}^{\DimX \times \DimX}$
	and
	$\DupMat \VECH{\OpeNotationSymMat}=\VEC{\OpeNotationSymMat}$ for any $\OpeNotationSymMat \in {\SetSymMat{\DimX}}$ hold, respectively.
\end{definition}
%%%%%%%%%%%%%%%%%%%%%%%%%%%%%%%%%%%%%%%%%%%%%%%%%%%%%%%%%%%%%%%%%%%%%%%
%%%%%%%%%%%%%%%%%%%%%%%%%%%%%%%%%%%%%%%%%%%%%%%%%%%%%%%%%%%%%%%%%%%%%%%
\begin{remark}[{\MyHighlight{Details of $\NonArgElimi$}}]
	The details of their definitions are described in \cite[Definiitons 3.1a, 3.1b, 3.2a, and 3.2b]{Magnus80}.
	As an example, $\EliMat$ and $\DupMat$ for $\DimX=2$ are given as follows:
	\begin{align}
	\EliMat
	&=
	\begin{bmatrix}
		1 & 0 & 0 & 0 \\
		0 & 1 & 0 & 0 \\
		0 & 0 & 0 & 1 
	\end{bmatrix}
	, \label{eq:ex_EliMat}
	\\	
	\DupMat
	&=
	\begin{bmatrix}
	1 & 0 & 0 \\
	0 & 1 & 0 \\
	0 & 1 & 0 \\
	0 & 0 & 1 
	\end{bmatrix}
	. \label{eq:ex_DupMat}	
	\end{align}
\end{remark}
%%%%%%%%%%%%%%%%%%%%%%%%%%%%%%%%%%%%%%%%%%%%%%%%%%%%%%%%%%%%%%%%%%%%%%%
%%%%%%%%%%%%%%%%%%%%%%%%%%%%%%%%%%%%%%%%%%%%%%%%%%%%%%%%%%%%%%%%%%%%%%%

We define the expanded system, using the vertices ${\polySMvecAB{\IDpoly}}$ and $\NUMpoly$ of the SMP system and the compression operator $\NonArgElimi$.
%%%%%%%%%%%%%%%%%%%%%%%%%%%%%%%%%%%%%%%%%%%%%%%%%%%%%%%%%%%%%%%%%%%%%%%
%%%%%%%%%%%%%%%%%%%%%%%%%%%%%%%%%%%%%%%%%%%%%%%%%%%%%%%%%%%%%%%%%%%%%%%
\iffalse\subsubsection{Expanded system}\fi
\begin{definition}[{\MyHighlight{Expanded system}}]
The $\DimexX$-dimensional expanded system to the SMP system \eqref{eq:def_CLsys} is defined by
\begin{align}
{\exState{\MyT+1}} 
&= \Elimi{\exCLMat{\FBgain}}  {\exState{\MyT}}
, \label{eq:def_exsys}
\\
{\exCLMat{\FBgain}}
&:= \sum_{\IDpoly=1}^{\NUMpoly}
{\El{\exTVUnc{\MyT}}{\IDpoly}}
{\expolyCLMat{\IDpoly}{\FBgain}}
. \label{eq:def_exCLMat}
\end{align}	
The symbols ${\exState{\MyT}} \in \mathbb{R}^{ \DimexX }$ and  ${\exTVUnc{\MyT}} \in \ImagePolyW$ denote the expanded state and expanded uncertain parameter, respectively, at the time $\MyT$, where 
$\ImagePolyW$ is the image of $\NonArgMapPolyW$ in Definition \ref{def:TVSMP}:
\begin{align}
\ImagePolyW
& :=
\big\{ 
{\MapPolyW{\TIUnc}}
\big|
 {\TIUnc} \in \DomTVUnc
\big\}
\subseteq \DomPolytope
\subset \mathbb{R}^{\NUMpoly} 
.\label{eq:def_DompolyW}
\end{align}
For each $\IDpoly\in\{1,\dots,\NUMpoly\}$,  ${\expolyCLMat{\IDpoly}{\FBgain}} \in \mathbb{R}^{ \DimX^{2} \times \DimX^{2} }$  in \eqref{eq:def_exCLMat} are defined as follows: 
\begin{align}
{\expolyCLMat{\IDpoly}{\FBgain}}
&:=
{\expolyAAMat{\IDpoly}}
-{\expolyABMat{\IDpoly}} ({\Identity{\DimX}} \otimes \FBgain)
\VisibleColTwo{\nonumber\\&\quad}
-{\expolyBAMat{\IDpoly}} (\FBgain \otimes {\Identity{\DimX}})
+{\expolyBBMat{\IDpoly}} (\FBgain \otimes \FBgain)
.
\end{align}	
The matrices 
${\expolyAAMat{\IDpoly}} \in \mathbb{R}^{\DimX^{2} \times \DimX^{2}}$,
${\expolyABMat{\IDpoly}} \in \mathbb{R}^{\DimX^{2} \times \DimX\DimU}$,
${\expolyBAMat{\IDpoly}} \in \mathbb{R}^{\DimX^{2} \times \DimX\DimU}$, and 
${\expolyBBMat{\IDpoly}} \in \mathbb{R}^{\DimX^{2} \times \DimU^{2}}$ are given by 
\begin{align}
{\El{\expolyAAMat{\IDpoly}}{\vectorWildCard, \DimX(\IDbEl-1) +\IDEl   }}&:=\VEC{ {\blockpolySMvecAB{\IDpoly}{\IDEl       }{\IDbEl}}   }
,\label{eq:def_expolyAAMat}
\\
{\El{\expolyABMat{\IDpoly}}{\vectorWildCard, \DimU(\IDbEl-1) +\IDcEl  }}&:=\VEC{ {\blockpolySMvecAB{\IDpoly}{\DimX+\IDcEl}{\IDbEl}}   }
,\\
{\El{\expolyBAMat{\IDpoly}}{\vectorWildCard, \DimX(\IDdEl-1) +\IDEl   }}&:=\VEC{ {\blockpolySMvecAB{\IDpoly}{\IDEl       }{\DimX+\IDdEl}}   }
,\\
{\El{\expolyBBMat{\IDpoly}}{\vectorWildCard, \DimU(\IDdEl-1) +\IDcEl  }}&:=\VEC{ {\blockpolySMvecAB{\IDpoly}{\DimX+\IDcEl}{\DimX+\IDdEl}}   }
,\label{eq:def_expolyBBMat}
\end{align}
for $\IDEl , \IDbEl \in \{1,\dots, \DimX\}$ and $\IDcEl, \IDdEl \in \{1,\dots, \DimU\}$,
using the block matrix form of ${\polySMvecAB{\IDpoly}}$:
\begin{align}
{\polySMvecAB{\IDpoly}}
&=:
\begin{bmatrix}
{\blockpolySMvecAB{\IDpoly}{1}{1}} & \cdots & {\blockpolySMvecAB{\IDpoly}{1}{\DimX+\DimU}} \\
\vdots & \ddots & \vdots \\
{\blockpolySMvecAB{\IDpoly}{\DimX+\DimU}{1}} & \cdots & {\blockpolySMvecAB{\IDpoly}{\DimX+\DimU}{\DimX+\DimU}} \\
\end{bmatrix}
,
\end{align}
where ${\blockpolySMvecAB{\IDpoly}{\IDEl}{\IDbEl}}$ for $\IDEl , \IDbEl \in \{1,\dots, \DimX+\DimU\}$ are $\DimX \times \DimX$ matrices.
\end{definition}
%%%%%%%%%%%%%%%%%%%%%%%%%%%%%%%%%%%%%%%%%%%%%%%%%%%%%%%%%%%%%%%%%%%%%%%
%%%%%%%%%%%%%%%%%%%%%%%%%%%%%%%%%%%%%%%%%%%%%%%%%%%%%%%%%%%%%%%%%%%%%%%
%%%%%%%%%%%%%%%%%%%%%%%%%%%%%%%%%%%%%%%%%%%%%%%%%%%%%%%%%%%%%%%%%%%%%%%
%%%%%%%%%%%%%%%%%%%%%%%%%%%%%%%%%%%%%%%%%%%%%%%%%%%%%%%%%%%%%%%%%%%%%%%	
\iffalse\subsubsection{Time-invariant/varying SMP}\fi
\begin{definition}[{\MyHighlight{Time-invariant/varying expanded systems}}]\label{def:TITVexpanded}
	The expanded system \eqref{eq:def_exsys} is said to be TI (resp. TV) if 
	the corresponding SMP system \eqref{eq:def_CLsys} is TI (resp. TV).
\end{definition}
%%%%%%%%%%%%%%%%%%%%%%%%%%%%%%%%%%%%%%%%%%%%%%%%%%%%%%%%%%%%%%%%%%%%%%%
%%%%%%%%%%%%%%%%%%%%%%%%%%%%%%%%%%%%%%%%%%%%%%%%%%%%%%%%%%%%%%%%%%%%%%%	

%%%%%%%%%%%%%%%%%%%%%%%%%%%%%%%%%%%%%%%%%%%%%%%%%%%%%%%%%%%%%%%%%%%%%%%
%%%%%%%%%%%%%%%%%%%%%%%%%%%%%%%%%%%%%%%%%%%%%%%%%%%%%%%%%%%%%%%%%%%%%%%
\begin{example}[{\MyHighlight{Demonstration of an expanded system}}]
For the SMP system in Example \ref{ex:SMP} with $\DimX=2$ and $\DimU=1$, 
we develop the expanded system, where $\DimexX=3$ is given by \eqref{eq:def_DimexX}.	
Because of ${\polySMvecAB{\IDpoly}}=\ConstMeanvecAB \ConstMeanvecAB^{\MyTRANSPO} + {\polyCovvecAB{\IDpoly}} \in {\SetSymMat{6}}$,
for any $\IDEl ,\IDbEl \in \{1,2,3\}$,
we obtain
\begin{align}
{\blockpolySMvecAB{\IDpoly}{\IDEl}{\IDbEl}}
&=
\begin{bmatrix}
{\El{\polySMvecAB{\IDpoly}}{ 2\IDEl -1 , 2\IDbEl -1 }} & {\El{\polySMvecAB{\IDpoly}}{ 2\IDEl -1 , 2\IDbEl  }} \\
{\El{\polySMvecAB{\IDpoly}}{ 2\IDEl    , 2\IDbEl -1 }} & {\El{\polySMvecAB{\IDpoly}}{ 2\IDEl    , 2\IDbEl  }}
\end{bmatrix}
.
\end{align}
Because of \eqref{eq:def_expolyAAMat},
${\expolyAAMat{\IDpoly}}$ is given by
\begin{align}
{\expolyAAMat{\IDpoly}}
&
=
[
\VEC{ {\blockpolySMvecAB{\IDpoly}{1}{1}} },
\VEC{ {\blockpolySMvecAB{\IDpoly}{2}{1}} },
\VEC{ {\blockpolySMvecAB{\IDpoly}{1}{2}} },
\VEC{ {\blockpolySMvecAB{\IDpoly}{2}{2}} }
]
\nonumber \\&
=
\begin{bmatrix}
{\El{\polySMvecAB{\IDpoly}}{ 1, 1 }} & {\El{\polySMvecAB{\IDpoly}}{ 3 , 1 }} & {\El{\polySMvecAB{\IDpoly}}{ 1 , 3 }} & {\El{\polySMvecAB{\IDpoly}}{ 3 , 3 }} \\
{\El{\polySMvecAB{\IDpoly}}{ 2, 1 }} & {\El{\polySMvecAB{\IDpoly}}{ 4 , 1 }} & {\El{\polySMvecAB{\IDpoly}}{ 2 , 3 }} & {\El{\polySMvecAB{\IDpoly}}{ 4 , 3 }} \\
{\El{\polySMvecAB{\IDpoly}}{ 1, 2 }} & {\El{\polySMvecAB{\IDpoly}}{ 3 , 2 }} & {\El{\polySMvecAB{\IDpoly}}{ 1 , 4 }} & {\El{\polySMvecAB{\IDpoly}}{ 3 , 4 }}  \\
{\El{\polySMvecAB{\IDpoly}}{ 2, 2 }} & {\El{\polySMvecAB{\IDpoly}}{ 4 , 2 }} & {\El{\polySMvecAB{\IDpoly}}{ 2 , 4 }} & {\El{\polySMvecAB{\IDpoly}}{ 4 , 4 }} 
\end{bmatrix}
.
\end{align}
In a similar manner, we derive 
\begin{align}
{\expolyABMat{\IDpoly}}
&
=
[
\VEC{ {\blockpolySMvecAB{\IDpoly}{3}{1}} },
\VEC{ {\blockpolySMvecAB{\IDpoly}{3}{2}} }
]
\VisibleColTwo{\nonumber \\&}
=
\begin{bmatrix}
{\El{\polySMvecAB{\IDpoly}}{ 5, 1 }} & {\El{\polySMvecAB{\IDpoly}}{ 5 , 3 }} \\
{\El{\polySMvecAB{\IDpoly}}{ 6, 1 }} & {\El{\polySMvecAB{\IDpoly}}{ 6 , 3 }} \\
{\El{\polySMvecAB{\IDpoly}}{ 5, 2 }} & {\El{\polySMvecAB{\IDpoly}}{ 5 , 4 }}  \\
{\El{\polySMvecAB{\IDpoly}}{ 6, 2 }} & {\El{\polySMvecAB{\IDpoly}}{ 6 , 4 }} 
\end{bmatrix}
,
\\
{\expolyBAMat{\IDpoly}}
&
=
[
\VEC{ {\blockpolySMvecAB{\IDpoly}{1}{3}} },
\VEC{ {\blockpolySMvecAB{\IDpoly}{2}{3}} }
]
\VisibleColTwo{\nonumber \\&}
=
\begin{bmatrix}
{\El{\polySMvecAB{\IDpoly}}{ 1, 5 }} & {\El{\polySMvecAB{\IDpoly}}{ 3 , 5 }} \\
{\El{\polySMvecAB{\IDpoly}}{ 2, 5 }} & {\El{\polySMvecAB{\IDpoly}}{ 4 , 5 }} \\
{\El{\polySMvecAB{\IDpoly}}{ 1, 6 }} & {\El{\polySMvecAB{\IDpoly}}{ 3 , 6 }}  \\
{\El{\polySMvecAB{\IDpoly}}{ 2, 6 }} & {\El{\polySMvecAB{\IDpoly}}{ 4 , 6 }} 
\end{bmatrix}
,
\\
{\expolyBBMat{\IDpoly}}
&
=\VEC{ {\blockpolySMvecAB{\IDpoly}{3}{3}} }
\VisibleColTwo{\nonumber \\&}
=
[
{\El{\polySMvecAB{\IDpoly}}{ 5, 5 }}, 
{\El{\polySMvecAB{\IDpoly}}{ 6, 5 }},
{\El{\polySMvecAB{\IDpoly}}{ 5, 6 }},
{\El{\polySMvecAB{\IDpoly}}{ 6, 6 }}
]^{\MyTRANSPO}
.
\end{align}
Subsequently, the $4 \times 4$ matrices ${\expolyCLMat{\IDpoly}{\FBgain}}$ and ${\exCLMat{\FBgain}}$ are defined.
Finally, using \eqref{eq:ex_EliMat} and \eqref{eq:ex_DupMat}, the $3 \times 3$ closed-loop matrix $ \Elimi{\exCLMat{\FBgain}}$ of the expanded system \eqref{eq:def_exsys} is obtained.
\end{example}
%%%%%%%%%%%%%%%%%%%%%%%%%%%%%%%%%%%%%%%%%%%%%%%%%%%%%%%%%%%%%%%%%%%%%%%
%%%%%%%%%%%%%%%%%%%%%%%%%%%%%%%%%%%%%%%%%%%%%%%%%%%%%%%%%%%%%%%%%%%%%%%

We derive a key connection between the expanded system \eqref{eq:def_exsys} and the SMP system \eqref{eq:def_CLsys}.

%%%%%%%%%%%%%%%%%%%%%%%%%%%%%%%%%%%%%%%%%%%%%%%%%%%%%%%%%%%%%%%%%%%%%%%
%%%%%%%%%%%%%%%%%%%%%%%%%%%%%%%%%%%%%%%%%%%%%%%%%%%%%%%%%%%%%%%%%%%%%%%
\iffalse\subsubsection{Expanded system}\fi
\begin{theorem}[{\MyHighlight{Expanded system}}]\label{thm:sys_to_exsys}
	For any ${\State{0}} \in \mathbb{R}^{\DimX}$ and any ${\TVUnc{\MyT}}  \in \DomTVUnc $ for $\MyT \in \{0,1,2,\dots\}$, suppose the following relations:
	\begin{align}
	{\exState{0}} &= \VECH{{\State{0}} \State{0}^{\MyTRANSPO} }
	, \label{eq:ass_ini_condition_State}
	\\
	{\exTVUnc{\MyT}}  &= {\MapPolyW{\TVUnc{\MyT}}}
	. \label{eq:ass_ini_condition_TISto}
	\end{align}
Then, the following property holds for all $\MyT \in \{0,1,2,\dots\}$: 	
	\begin{align}
	{\exState{\MyT}} =  
	{\CondExpectAllTV[\big]{  \VECH{{\State{\MyT}} \State{\MyT}^{\MyTRANSPO} }   }{\MybT}{\MyT}}
	. \label{eq:exState_is_2ndM_State}
	\end{align}		
\end{theorem}
%%%%%%%%%%%%%%%%%%%%%%%%%%%%%%%%%%%%%%%%%%%%%%%%%%%%%%%%%%%%%%%%%%%%%%%
%%%%%%%%%%%%%%%%%%%%%%%%%%%%%%%%%%%%%%%%%%%%%%%%%%%%%%%%%%%%%%%%%%%%%%%
\begin{proof}
	The proof is described in Appendix \ref{pf:sys_to_exsys}.
\end{proof}
%%%%%%%%%%%%%%%%%%%%%%%%%%%%%%%%%%%%%%%%%%%%%%%%%%%%%%%%%%%%%%%%%%%%%%%
%%%%%%%%%%%%%%%%%%%%%%%%%%%%%%%%%%%%%%%%%%%%%%%%%%%%%%%%%%%%%%%%%%%%%%%
\begin{remark}[{\MyHighlight{Contribution of Theorem \ref{thm:sys_to_exsys}}}]
The state's second moment ${\CondExpectAllTV[\big]{  \VECH{{\State{\MyT}} \State{\MyT}^{\MyTRANSPO} }   }{\MybT}{\MyT}}$ of the SMP system \eqref{eq:def_CLsys} is represented by the state ${\exState{\MyT}}$ of the expanded system \eqref{eq:def_exsys}.
The expanded state ${\exState{\MyT}}$ is easier to analyze than the second moment because the expanded system is included in a deterministic polytope without suffering from the randomness of ${\vecTVAB{\MyT}}$ given $\TVUnc{\MyT}$.
The expanded system is constructed based on the Kronecker product $\otimes$ of the system matrices of the SMP system (the details are found in the proof in Appendix \ref{pf:sys_to_exsys}).
Although existing methods \cite{Hibey96,Ogura13} have used such a technique, they do not treat uncertain parameters $\TVUnc{\MyT}$.

\end{remark}
%%%%%%%%%%%%%%%%%%%%%%%%%%%%%%%%%%%%%%%%%%%%%%%%%%%%%%%%%%%%%%%%%%%%%%%
%%%%%%%%%%%%%%%%%%%%%%%%%%%%%%%%%%%%%%%%%%%%%%%%%%%%%%%%%%%%%%%%%%%%%%%
\begin{remark}[{\MyHighlight{Contribution of the compression operator $\NonArgElimi$}}]

	The operator $\NonArgElimi$ makes the expanded system \eqref{eq:def_exsys} low-dimensional.
	Whereas a straightforward expression of the second moment is ${\CondExpectAllTV[\big]{  \VEC{{\State{\MyT}} \State{\MyT}^{\MyTRANSPO} }   }{\MybT}{\MyT}}$ similar to \cite{Hibey96,Ogura13}, this includes duplicated components
	because of ${\El{{{\State{\MyT}}{\State{\MyT}}^{\MyTRANSPO} }}{\IDEl,\IDbEl}}={\El{{{\State{\MyT}}{\State{\MyT}}^{\MyTRANSPO} }}{\IDbEl,\IDEl}}$.
	The proposed expanded system with $\NonArgElimi$ handles the second moment by using the half vectorization $\VECH{{\State{\MyT}} \State{\MyT}^{\MyTRANSPO} } $ without the duplication.
	This expression is computationally efficient for $\DimX\geq 2$ because the dimension $\DimexX=\DimX(\DimX+1)/2$ of $\VECH{{\State{\MyT}}{\State{\MyT}}^{\MyTRANSPO} }$ is less than the dimension $\DimX^{2}$ of $\VEC{{\State{\MyT}}{\State{\MyT}}^{\MyTRANSPO} }$.	
	Whereas existing methods, e.g., \cite{Luo19TAC,Zhang12TAC}, have employed such a compression successfully, they do not treat uncertain parameters $\TVUnc{\MyT}$.

\end{remark}
%%%%%%%%%%%%%%%%%%%%%%%%%%%%%%%%%%%%%%%%%%%%%%%%%%%%%%%%%%%%%%%%%%%%%%%
%%%%%%%%%%%%%%%%%%%%%%%%%%%%%%%%%%%%%%%%%%%%%%%%%%%%%%%%%%%%%%%%%%%%%%%	

%%%%%%%%%%%%%%%%%%%%%%%%%%%%%%%%%%%%%%%%%%%%%%%%%%%%%%%%%%%%%%%%%%%%%%%%%%%%%%%%%%%%%%%%%%%%%%%%%%%%%%%%%%%%%%%%%%%%%%%%%%%%%%%%%%%%%%%
%%%%%%%%%%%%%%%%%%%%%%%%%%%%%%%%%%%%%%%%%%%%%%%%%%%%%%%%%%%%%%%%%%%%%%%%%%%%%%%%%%%%%%%%%%%%%%%%%%%%%%%%%%%%%%%%%%%%%%%%%%%%%%%%%%%%%%%
%%%%%%%%%%%%%%%%%%%%%%%%%%%%%%%%%%%%%%%%%%%%%%%%%%%%%%%%%%%%%%%%%%%%%%%%%%%%%%%%%%%%%%%%%%%%%%%%%%%%%%%%%%%%%%%%%%%%%%%%%%%%%%%%%%%%%%%
\subsection{Solutions to Problem 1: Stability analysis}\label{sec_stability}

In this subsection, we solve Problem 1 by using the expanded system \eqref{eq:def_exsys} with Theorem \ref{thm:sys_to_exsys}.
We show that (exponential) robust stability of the expanded system is equivalent to the (exponential) robust MS stability of the SMP system \eqref{eq:def_CLsys}.
Given a feedback gain $\FBgain$, stability conditions for the expanded system are derived as LMIs.
Therefore, the LMIs are solutions to Problem 1.

First, the two stability notions of the expanded system \eqref{eq:def_exsys} are defined below.	
%%%%%%%%%%%%%%%%%%%%%%%%%%%%%%%%%%%%%%%%%%%%%%%%%%%%%%%%%%%%%%%%%%%%%%%
%%%%%%%%%%%%%%%%%%%%%%%%%%%%%%%%%%%%%%%%%%%%%%%%%%%%%%%%%%%%%%%%%%%%%%%	
\begin{definition}[{\MyHighlight{Robust stability}}]
	The expanded system \eqref{eq:def_exsys} is said to be robustly stable if %the following relation holds:
	\begin{align}
	\forall {\exState{0}}  \in  \mathbb{R}^{\DimexX}
	,\;
	\forall {\exTVUnc{0}},{\exTVUnc{1}} , {\dots} \in \ImagePolyW
	,\quad
	\lim_{\MyT \to \infty}
	\|    {\exState{\MyT}}    \|%^{2}
	=0
	.\label{eq:def_Rstable}
	\end{align}
\end{definition}
%%%%%%%%%%%%%%%%%%%%%%%%%%%%%%%%%%%%%%%%%%%%%%%%%%%%%%%%%%%%%%%%%%%%%%%
%%%%%%%%%%%%%%%%%%%%%%%%%%%%%%%%%%%%%%%%%%%%%%%%%%%%%%%%%%%%%%%%%%%%%%%	

%%%%%%%%%%%%%%%%%%%%%%%%%%%%%%%%%%%%%%%%%%%%%%%%%%%%%%%%%%%%%%%%%%%%%%%
%%%%%%%%%%%%%%%%%%%%%%%%%%%%%%%%%%%%%%%%%%%%%%%%%%%%%%%%%%%%%%%%%%%%%%%	
\begin{definition}[{\MyHighlight{Exponential robust stability}}]
	The expanded system {\eqref{eq:def_exsys}} is said to be exponentially robustly stable if there exist $\exEMScoef \in (0,\infty)$ and $\exEMSrate \in (0,1)$ such that %the following relation holds:
	\begin{align}
	\forall {\exState{0}}  \in  \mathbb{R}^{\DimexX}
	,\;	
	\forall {\exTVUnc{0}},{\exTVUnc{1}} , {\dots}  \in \ImagePolyW
	,\;
	\forall \MyT 
	,\quad%\nonumber\\&
	\|    {\exState{\MyT}}    \|
	\leq \exEMScoef  \|{\exState{0}}\| \exEMSrate^{\MyT}
	.\label{eq:def_ERstable}
	\end{align}
\end{definition}
%%%%%%%%%%%%%%%%%%%%%%%%%%%%%%%%%%%%%%%%%%%%%%%%%%%%%%%%%%%%%%%%%%%%%%%
%%%%%%%%%%%%%%%%%%%%%%%%%%%%%%%%%%%%%%%%%%%%%%%%%%%%%%%%%%%%%%%%%%%%%%%	

We derive equivalence of the stability between the expanded system and the SMP system as follows.

%%%%%%%%%%%%%%%%%%%%%%%%%%%%%%%%%%%%%%%%%%%%%%%%%%%%%%%%%%%%%%%%%%%%%%%
%%%%%%%%%%%%%%%%%%%%%%%%%%%%%%%%%%%%%%%%%%%%%%%%%%%%%%%%%%%%%%%%%%%%%%%	
\iffalse\subsubsection{Equivalence of the stability}\fi
\begin{theorem}[{\MyHighlight{Equivalence of the stability}}]\label{thm:Equivarence_RMS}
	The SMP system {\eqref{eq:def_CLsys}}
	is robustly MS stable if and only if 
	the expanded system {\eqref{eq:def_exsys}}
	is robustly stable.
\end{theorem}
%%%%%%%%%%%%%%%%%%%%%%%%%%%%%%%%%%%%%%%%%%%%%%%%%%%%%%%%%%%%%%%%%%%%%%%
%%%%%%%%%%%%%%%%%%%%%%%%%%%%%%%%%%%%%%%%%%%%%%%%%%%%%%%%%%%%%%%%%%%%%%%	
\begin{proof}
	The proof is described in Appendix \ref{pf:Equivarence_RMS}.		
\end{proof}
%%%%%%%%%%%%%%%%%%%%%%%%%%%%%%%%%%%%%%%%%%%%%%%%%%%%%%%%%%%%%%%%%%%%%%%
%%%%%%%%%%%%%%%%%%%%%%%%%%%%%%%%%%%%%%%%%%%%%%%%%%%%%%%%%%%%%%%%%%%%%%%	

%%%%%%%%%%%%%%%%%%%%%%%%%%%%%%%%%%%%%%%%%%%%%%%%%%%%%%%%%%%%%%%%%%%%%%%
%%%%%%%%%%%%%%%%%%%%%%%%%%%%%%%%%%%%%%%%%%%%%%%%%%%%%%%%%%%%%%%%%%%%%%%	
\iffalse\subsubsection{Equivalence of the exponential stability}\fi
\begin{theorem}[{\MyHighlight{Equivalence of the exponential stability}}]\label{thm:Equivarence_ERMS}
	The SMP system \eqref{eq:def_CLsys} is exponentially robustly MS stable with $\EMSrate \in (0,1)$
	if and only if the expanded system \eqref{eq:def_exsys} is exponentially robustly stable with $\exEMSrate \in (0,1)$ satisfying
	\begin{align}
	\exEMSrate
	& = 
	\EMSrate^{2}
	. \label{eq:exEMSrate_to_EMSrate}
	\end{align}
\end{theorem}
%%%%%%%%%%%%%%%%%%%%%%%%%%%%%%%%%%%%%%%%%%%%%%%%%%%%%%%%%%%%%%%%%%%%%%%
%%%%%%%%%%%%%%%%%%%%%%%%%%%%%%%%%%%%%%%%%%%%%%%%%%%%%%%%%%%%%%%%%%%%%%%	
\begin{proof}	
	The proof is described in Appendix \ref{pf:Equivarence_ERMS}.
\end{proof}
%%%%%%%%%%%%%%%%%%%%%%%%%%%%%%%%%%%%%%%%%%%%%%%%%%%%%%%%%%%%%%%%%%%%%%%
%%%%%%%%%%%%%%%%%%%%%%%%%%%%%%%%%%%%%%%%%%%%%%%%%%%%%%%%%%%%%%%%%%%%%%%	

%%%%%%%%%%%%%%%%%%%%%%%%%%%%%%%%%%%%%%%%%%%%%%%%%%%%%%%%%%%%%%%%%%%%%%%
%%%%%%%%%%%%%%%%%%%%%%%%%%%%%%%%%%%%%%%%%%%%%%%%%%%%%%%%%%%%%%%%%%%%%%%	
\begin{remark}[{\MyHighlight{Contributions of Theorems \ref{thm:Equivarence_RMS} and \ref{thm:Equivarence_ERMS}}}]
The stability analysis of the SMP system \eqref{eq:def_CLsys} reduces to that of the expanded system \eqref{eq:def_exsys}.
Recall that the expanded system is included in a deterministic polytope with the uncertain parameter ${\exTVUnc{\MyT}} \in \ImagePolyW \subseteq \DomPolytope$ in \eqref{eq:def_DompolyW}.
Thus, we can employ various stability analyses for polytopic linear systems, e.g., \cite{Oliveira99}.

\end{remark}
%%%%%%%%%%%%%%%%%%%%%%%%%%%%%%%%%%%%%%%%%%%%%%%%%%%%%%%%%%%%%%%%%%%%%%%
%%%%%%%%%%%%%%%%%%%%%%%%%%%%%%%%%%%%%%%%%%%%%%%%%%%%%%%%%%%%%%%%%%%%%%%	

In the following, we derive sufficient conditions for the stability of the expanded system \eqref{eq:def_exsys} based on existing results in \cite{Oliveira99}.
Let us define the following matrix-valued function that is cubic in 
$\UNIexpolyVMat \in  {\SetSymMat{\DimexX}}$,
$\exAddMat \in \mathbb{R}^{\DimexX \times \DimexX}$, and
$\FBgain \in \mathbb{R}^{\DimU \times \DimX}$
 with fixed $\exEMSrate$:
\begin{align}
&{\CMIterms{\IDpoly}{\UNIexpolyVMat}{\exAddMat}{\FBgain}{\exEMSrate}}
\VisibleColTwo{\nonumber\\&}
:=
\begin{bmatrix}
\exEMSrate^{2} 
\UNIexpolyVMat
\;\;&\;\;     
\Elimi{ {\expolyCLMat{\IDpoly}{\FBgain}}  }^{\MyTRANSPO}   \exAddMat
\\
\exAddMat^{\MyTRANSPO}   \Elimi{ {\expolyCLMat{\IDpoly}{\FBgain}}  }  
\;\;&\;\;   
\exAddMat^{\MyTRANSPO} + \exAddMat - \UNIexpolyVMat
\end{bmatrix}
.\label{eq:def_LMIterms}
\end{align}	
Stability conditions of the expanded system are derived below.

%%%%%%%%%%%%%%%%%%%%%%%%%%%%%%%%%%%%%%%%%%%%%%%%%%%%%%%%%%%%%%%%%%%%%%%
%%%%%%%%%%%%%%%%%%%%%%%%%%%%%%%%%%%%%%%%%%%%%%%%%%%%%%%%%%%%%%%%%%%%%%%	
\iffalse\subsubsection{Exp. robust stability of TI expanded systems}\fi
\begin{theorem}[{\MyHighlight{Exp. robust stability of TI expanded systems}}]\label{thm:MSstab_TISMP}
	Suppose that the expanded system \eqref{eq:def_exsys} is TI.
	The system \eqref{eq:def_exsys} is exponentially robustly stable with a given $\exEMSrate \in (0, 1)$ 
	if the following condition {\itemMultiPERSCMIs} holds:	
	\begin{enumerate}
		\setlength{\itemindent}{10pt}
		
		\item[\itemMultiPERSCMIs]
		There exist ${\expolyVMat{\IDpoly}} \succ 0 \in {\SetSymMat{\DimexX}}$ for $\IDpoly \in \{1,\dots, \NUMpoly \}$,  $\exAddMat$, and $\FBgain$ such that
		\begin{align}
		\forall \IDpoly \in \{1,\dots, \NUMpoly \}
		,\quad
		{\CMIterms{\IDpoly}{\expolyVMat{\IDpoly}}{\exAddMat}{\FBgain}{\exEMSrate}}
		\succeq 0
		.\label{eq:expanded_ERS}
		\end{align}	
		
	\end{enumerate}		
	
\end{theorem}
%%%%%%%%%%%%%%%%%%%%%%%%%%%%%%%%%%%%%%%%%%%%%%%%%%%%%%%%%%%%%%%%%%%%%%%
%%%%%%%%%%%%%%%%%%%%%%%%%%%%%%%%%%%%%%%%%%%%%%%%%%%%%%%%%%%%%%%%%%%%%%%	
\begin{proof}
The proof is described in Appendix \ref{pf:MSstab_TISMP}.		
\end{proof}
%%%%%%%%%%%%%%%%%%%%%%%%%%%%%%%%%%%%%%%%%%%%%%%%%%%%%%%%%%%%%%%%%%%%%%%
%%%%%%%%%%%%%%%%%%%%%%%%%%%%%%%%%%%%%%%%%%%%%%%%%%%%%%%%%%%%%%%%%%%%%%%	

%%%%%%%%%%%%%%%%%%%%%%%%%%%%%%%%%%%%%%%%%%%%%%%%%%%%%%%%%%%%%%%%%%%%%%%
%%%%%%%%%%%%%%%%%%%%%%%%%%%%%%%%%%%%%%%%%%%%%%%%%%%%%%%%%%%%%%%%%%%%%%%	
\iffalse\subsubsection{Robust stability of TI expanded systems}\fi
\begin{theorem}[{\MyHighlight{Robust stability of TI expanded systems}}]\label{thm:RS_TIexpanded}
	Suppose that the expanded system \eqref{eq:def_exsys} is TI.
	The system \eqref{eq:def_exsys} is robustly stable 
	if the following {\itemMultiPRSCMIs} holds:		
	\begin{enumerate}
		\setlength{\itemindent}{10pt}
					
		\item[\itemMultiPRSCMIs]
		There exist ${\expolyVMat{\IDpoly}} \succ 0 \in {\SetSymMat{\DimexX}}$ for $\IDpoly \in \{1,\dots, \NUMpoly \}$,  $\exAddMat$, and $\FBgain$ such that
		\begin{align}
		\forall \IDpoly \in \{1,\dots, \NUMpoly \}
		,\quad
		{\CMIterms{\IDpoly}{\expolyVMat{\IDpoly}}{\exAddMat}{\FBgain}{1}}
		\succ 0
		.\label{eq:expanded_RS}
		\end{align}

	\end{enumerate}	
\end{theorem}
%%%%%%%%%%%%%%%%%%%%%%%%%%%%%%%%%%%%%%%%%%%%%%%%%%%%%%%%%%%%%%%%%%%%%%%
%%%%%%%%%%%%%%%%%%%%%%%%%%%%%%%%%%%%%%%%%%%%%%%%%%%%%%%%%%%%%%%%%%%%%%%	
\begin{proof}
	The proof is described in Appendix \ref{pf:RS_TIexpanded}.		
\end{proof}
%%%%%%%%%%%%%%%%%%%%%%%%%%%%%%%%%%%%%%%%%%%%%%%%%%%%%%%%%%%%%%%%%%%%%%%
%%%%%%%%%%%%%%%%%%%%%%%%%%%%%%%%%%%%%%%%%%%%%%%%%%%%%%%%%%%%%%%%%%%%%%%	

%%%%%%%%%%%%%%%%%%%%%%%%%%%%%%%%%%%%%%%%%%%%%%%%%%%%%%%%%%%%%%%%%%%%%%%
%%%%%%%%%%%%%%%%%%%%%%%%%%%%%%%%%%%%%%%%%%%%%%%%%%%%%%%%%%%%%%%%%%%%%%%	
\iffalse\subsubsection{Stability of TV expanded systems}\fi
\begin{theorem}[{\MyHighlight{Stability of TV expanded systems}}]\label{thm:ERS_RS_TVexpanded}
	Even if the expanded system \eqref{eq:def_exsys} is not TI but TV,	
	the statements of Theorems \ref{thm:MSstab_TISMP} and \ref{thm:RS_TIexpanded} hold 
	if ${\expolyVMat{\IDpoly}}$ for $\IDpoly \in \{1,\dots, \NUMpoly \}$ are replaced with an identical matrix $\UNIexpolyVMat \succ 0 \in {\SetSymMat{\DimexX}}$ in  {\itemMultiPERSCMIs} and {\itemMultiPRSCMIs}.	
\end{theorem}
%%%%%%%%%%%%%%%%%%%%%%%%%%%%%%%%%%%%%%%%%%%%%%%%%%%%%%%%%%%%%%%%%%%%%%%
%%%%%%%%%%%%%%%%%%%%%%%%%%%%%%%%%%%%%%%%%%%%%%%%%%%%%%%%%%%%%%%%%%%%%%%	
\begin{proof}
	The proof is described in Appendix \ref{pf:ERS_RS_TVexpanded}.			
\end{proof}
%%%%%%%%%%%%%%%%%%%%%%%%%%%%%%%%%%%%%%%%%%%%%%%%%%%%%%%%%%%%%%%%%%%%%%%
%%%%%%%%%%%%%%%%%%%%%%%%%%%%%%%%%%%%%%%%%%%%%%%%%%%%%%%%%%%%%%%%%%%%%%%	

%%%%%%%%%%%%%%%%%%%%%%%%%%%%%%%%%%%%%%%%%%%%%%%%%%%%%%%%%%%%%%%%%%%%%%%
%%%%%%%%%%%%%%%%%%%%%%%%%%%%%%%%%%%%%%%%%%%%%%%%%%%%%%%%%%%%%%%%%%%%%%%	
\begin{remark}[{\MyHighlight{Contributions of the above statements}}]
Theorems \ref{thm:MSstab_TISMP} and \ref{thm:RS_TIexpanded} provide sufficient stability conditions of the TI expanded systems for a given $\exEMSrate$.
Theorem \ref{thm:ERS_RS_TVexpanded} applies the conditions to TV systems, which focuses the quadratic stability.
Moreover, note that the conditions reduce to LMIs of ${\expolyVMat{\IDpoly}}$ and $\exAddMat$ if $\FBgain$ is given or if $\InMat{\MyT}\FBgain=0$ holds (autonomous cases). 
\end{remark}
%%%%%%%%%%%%%%%%%%%%%%%%%%%%%%%%%%%%%%%%%%%%%%%%%%%%%%%%%%%%%%%%%%%%%%%
%%%%%%%%%%%%%%%%%%%%%%%%%%%%%%%%%%%%%%%%%%%%%%%%%%%%%%%%%%%%%%%%%%%%%%%	

Consequently, solutions to Problem 1 are summarized using Theorems \ref{thm:Equivarence_RMS}--\ref{thm:ERS_RS_TVexpanded}.

%%%%%%%%%%%%%%%%%%%%%%%%%%%%%%%%%%%%%%%%%%%%%%%%%%%%%%%%%%%%%%%%%%%%%%%
%%%%%%%%%%%%%%%%%%%%%%%%%%%%%%%%%%%%%%%%%%%%%%%%%%%%%%%%%%%%%%%%%%%%%%%	
\iffalse\subsubsection{Solutions to Problem 1}\fi
\begin{corollary}[{\MyHighlight{Solutions to Problem 1}}]\label{thm:solution_to_Problem1}
The SMP system {\eqref{eq:def_CLsys}} is exponentially robustly MS stable (resp. robustly MS stable) if
either of the following (i) or (ii) is satisfied:	
\begin{enumerate}
	\item 
	{\itemMultiPERSCMIs} (resp. {\itemMultiPRSCMIs}) holds and the system is TI.
	\item
	{\itemMultiPERSCMIs} (resp. {\itemMultiPRSCMIs}) holds when replacing ${\expolyVMat{\IDpoly}}$ for $\IDpoly \in \{1,\dots, \NUMpoly \}$ with an identical $\UNIexpolyVMat \succ 0$.	
\end{enumerate}
\end{corollary}
%%%%%%%%%%%%%%%%%%%%%%%%%%%%%%%%%%%%%%%%%%%%%%%%%%%%%%%%%%%%%%%%%%%%%%%
%%%%%%%%%%%%%%%%%%%%%%%%%%%%%%%%%%%%%%%%%%%%%%%%%%%%%%%%%%%%%%%%%%%%%%%

%%%%%%%%%%%%%%%%%%%%%%%%%%%%%%%%%%%%%%%%%%%%%%%%%%%%%%%%%%%%%%%%%%%%%%%%%%%%%%%%%%%%%%%%%%%%%%%%%%%%%%%%%%%%%%%%%%%%%%%%%%%%%%%%%%%%%%%
%%%%%%%%%%%%%%%%%%%%%%%%%%%%%%%%%%%%%%%%%%%%%%%%%%%%%%%%%%%%%%%%%%%%%%%%%%%%%%%%%%%%%%%%%%%%%%%%%%%%%%%%%%%%%%%%%%%%%%%%%%%%%%%%%%%%%%%
%%%%%%%%%%%%%%%%%%%%%%%%%%%%%%%%%%%%%%%%%%%%%%%%%%%%%%%%%%%%%%%%%%%%%%%%%%%%%%%%%%%%%%%%%%%%%%%%%%%%%%%%%%%%%%%%%%%%%%%%%%%%%%%%%%%%%%%
\subsection{Solutions to Problem 2: QMI-based controller design}\label{sec_control_QMIs}

In this subsection, we solve Problem 2, starting from the stability conditions {\itemMultiPERSCMIs} and {\itemMultiPRSCMIs}.
Unfortunately, these conditions are CMIs with respect to $(\UNIexpolyVMat,\exAddMat,\FBgain)$ if the feedback gain $\FBgain$ is not given but regarded as a decision variable.
To avoid solving the CMIs directly, we transform them into simpler QMIs in the following.

Let us define the following matrix-valued function that is quadratic in 
$\UNIAIMexpolyVMat \in {\SetSymMat{\DimexX}}$,
$\AddInvMat \in \mathbb{R}^{\DimX \times \DimX}$, and
$\AIMFBgain \in \mathbb{R}^{\DimU \times \DimX}$ with fixed $\exEMSrate$:
\begin{align}
&{\QMIterms{\IDpoly}{\UNIAIMexpolyVMat}{\AddInvMat}{\AIMFBgain}{\exEMSrate}}
\VisibleColTwo{\nonumber\\&}
:=
\begin{bmatrix}
\exEMSrate^{2} 
\UNIAIMexpolyVMat
&  
{\QMILowLeftblock{\IDpoly}{\AIMFBgain,\AddInvMat}^{\MyTRANSPO}  } 
\\
{\QMILowLeftblock{\IDpoly}{\AIMFBgain,\AddInvMat}}
&  
\Elimi{ \AddInvMat \otimes \AddInvMat } + \Elimi{ \AddInvMat \otimes \AddInvMat }^{\MyTRANSPO}   -  \UNIAIMexpolyVMat
\end{bmatrix}
,\label{eq:def_QMIterms}
\end{align}	
where
\begin{align}	
{\QMILowLeftblock{\IDpoly}{\AIMFBgain,\AddInvMat}}
&:=
\NonArgElimi
\Big( 
{\expolyAAMat{\IDpoly}}  (\AddInvMat \otimes \AddInvMat)
-{\expolyABMat{\IDpoly}} (\AddInvMat \otimes \AIMFBgain)
\VisibleColTwo{\nonumber \\&\quad}
-{\expolyBAMat{\IDpoly}} (\AIMFBgain \otimes \AddInvMat)
+{\expolyBBMat{\IDpoly}} (\AIMFBgain \otimes \AIMFBgain)
\Big)
.\label{eq:QMILowLeftblock}
\end{align}		
We derive the following results to simplify stability conditions.

%%%%%%%%%%%%%%%%%%%%%%%%%%%%%%%%%%%%%%%%%%%%%%%%%%%%%%%%%%%%%%%%%%%%%%%
%%%%%%%%%%%%%%%%%%%%%%%%%%%%%%%%%%%%%%%%%%%%%%%%%%%%%%%%%%%%%%%%%%%%%%%	
\iffalse\subsubsection{QMI}\fi
\begin{theorem}[\MyHighlight{Controller design via QMIs}]\label{thm:stability_QMI}
For a given $\exEMSrate \in (0, 1)$, 
the following {\itemMultiPERSQMIs} and {\itemMultiPRSQMIs} imply 
{\itemMultiPERSCMIs} and  {\itemMultiPRSCMIs}, 
respectively, with the settings:
\begin{align}
{\expolyVMat{\IDpoly}}
&=  
\exAddMat^{\MyTRANSPO}
{\AIMexpolyVMat{\IDpoly}} 
\exAddMat
,
 \label{eq:AIMexpolyVMat_to_expolyVMat}
\\
\exAddMat &=\Elimi{ \AddInvMat \otimes \AddInvMat }^{-1} % \exAddInvMat^{-1}
, \label{eq:AddInvMat_to_exAddMat}
\\
\FBgain 
&= 
\AIMFBgain \AddInvMat^{-1} 
. \label{eq:AIMFBgain_to_FBgain}
\end{align}	
\begin{enumerate}
	\setlength{\itemindent}{10pt}

	\item[{\itemMultiPERSQMIs}]
	There exist ${\AIMexpolyVMat{\IDpoly}} \succ 0 \in {\SetSymMat{\DimexX}}$ for $\IDpoly \in \{1,\dots, \NUMpoly \}$, $\AddInvMat$, and $\AIMFBgain$ such that %the following QMIs hold:	
	\begin{align}
	\forall \IDpoly \in \{1,\dots, \NUMpoly \}
	,\quad
	{\QMIterms{\IDpoly}{\AIMexpolyVMat{\IDpoly}}{\AddInvMat}{\AIMFBgain}{\exEMSrate}}
	\succeq 0
	.\label{eq:expanded_ERS_QMI}
	\end{align}

	\item[{\itemMultiPRSQMIs}]
	There exist ${\AIMexpolyVMat{\IDpoly}} \succ 0 \in {\SetSymMat{\DimexX}}$ for $\IDpoly \in \{1,\dots, \NUMpoly \}$, $\AddInvMat$, and $\AIMFBgain$ such that
	\begin{align}
	\forall \IDpoly \in \{1,\dots, \NUMpoly \}
	,\quad
	{\QMIterms{\IDpoly}{\AIMexpolyVMat{\IDpoly}}{\AddInvMat}{\AIMFBgain}{1}}
	\succ 0
	.\label{eq:expanded_RS_QMI}
	\end{align}

\end{enumerate}		
Furthermore, if ${\AIMexpolyVMat{\IDpoly}}$ are identical with respect to ${\IDpoly}$, ${\expolyVMat{\IDpoly}}$ given in \eqref{eq:AIMexpolyVMat_to_expolyVMat} are also identical.
\end{theorem}
%%%%%%%%%%%%%%%%%%%%%%%%%%%%%%%%%%%%%%%%%%%%%%%%%%%%%%%%%%%%%%%%%%%%%%%
%%%%%%%%%%%%%%%%%%%%%%%%%%%%%%%%%%%%%%%%%%%%%%%%%%%%%%%%%%%%%%%%%%%%%%%
\begin{proof}
	The proof is described in Appendix \ref{pf:stability_QMI}.	
\end{proof}
%%%%%%%%%%%%%%%%%%%%%%%%%%%%%%%%%%%%%%%%%%%%%%%%%%%%%%%%%%%%%%%%%%%%%%%
%%%%%%%%%%%%%%%%%%%%%%%%%%%%%%%%%%%%%%%%%%%%%%%%%%%%%%%%%%%%%%%%%%%%%%%

%%%%%%%%%%%%%%%%%%%%%%%%%%%%%%%%%%%%%%%%%%%%%%%%%%%%%%%%%%%%%%%%%%%%%%%
%%%%%%%%%%%%%%%%%%%%%%%%%%%%%%%%%%%%%%%%%%%%%%%%%%%%%%%%%%%%%%%%%%%%%%%
\begin{remark}[{\MyHighlight{Contribution of Theorem \ref{thm:stability_QMI}}}]
The CMIs are transformed into the QMIs with respect to $({\AIMexpolyVMat{\IDpoly}},{\AddInvMat},{\AIMFBgain})$, which are expected to be easier to solve.
Solving the QMIs yields a stabilizing feedback gain $\FBgain$ in \eqref{eq:AIMFBgain_to_FBgain} subsequently, as illustrated in Fig. \ref{fig:overview}.
Furthermore, in Section \ref{sec_control_SDPs}, the QMIs are relaxed as an iterative convex program whereas the CMIs and QMIs are non-convex.
\end{remark}
%%%%%%%%%%%%%%%%%%%%%%%%%%%%%%%%%%%%%%%%%%%%%%%%%%%%%%%%%%%%%%%%%%%%%%%
%%%%%%%%%%%%%%%%%%%%%%%%%%%%%%%%%%%%%%%%%%%%%%%%%%%%%%%%%%%%%%%%%%%%%%%

%%%%%%%%%%%%%%%%%%%%%%%%%%%%%%%%%%%%%%%%%%%%%%%%%%%%%%%%%%%%%%%%%%%%%%%
%%%%%%%%%%%%%%%%%%%%%%%%%%%%%%%%%%%%%%%%%%%%%%%%%%%%%%%%%%%%%%%%%%%%%%%
\begin{remark}[\MyHighlight{Intuition of deriving Theorem \ref{thm:stability_QMI}}]
Multiplying the QMIs
${\QMIterms{\IDpoly}{\AIMexpolyVMat{\IDpoly}}{\AddInvMat}{\AIMFBgain}{\exEMSrate}} $ $\succeq 0$  by a block matrix using $\exAddMat$ in \eqref{eq:AddInvMat_to_exAddMat} yields the CMIs ${\CMIterms{\IDpoly}{\expolyVMat{\IDpoly}}{\exAddMat}{\FBgain}{\exEMSrate}} \succeq 0$.
While such a multiplication is based on robust control design \cite{Oliveira99,HosoeTAC18}, 
we show that this technique can be applied to matrix inequalities involving the compression operator $\NonArgElimi$ and Kronecker product $\otimes$ (the details are found in the proof in Appendix \ref{pf:stability_QMI}).

\end{remark}
%%%%%%%%%%%%%%%%%%%%%%%%%%%%%%%%%%%%%%%%%%%%%%%%%%%%%%%%%%%%%%%%%%%%%%%
%%%%%%%%%%%%%%%%%%%%%%%%%%%%%%%%%%%%%%%%%%%%%%%%%%%%%%%%%%%%%%%%%%%%%%%

%%%%%%%%%%%%%%%%%%%%%%%%%%%%%%%%%%%%%%%%%%%%%%%%%%%%%%%%%%%%%%%%%%%%%%%%%%%%%%%%%%%%%%%%%%%%%%%%%%%%%%%%%%%%%%%%%%%%%%%%%%%%%%%%%%%%%%%
%%%%%%%%%%%%%%%%%%%%%%%%%%%%%%%%%%%%%%%%%%%%%%%%%%%%%%%%%%%%%%%%%%%%%%%%%%%%%%%%%%%%%%%%%%%%%%%%%%%%%%%%%%%%%%%%%%%%%%%%%%%%%%%%%%%%%%%
%%%%%%%%%%%%%%%%%%%%%%%%%%%%%%%%%%%%%%%%%%%%%%%%%%%%%%%%%%%%%%%%%%%%%%%%%%%%%%%%%%%%%%%%%%%%%%%%%%%%%%%%%%%%%%%%%%%%%%%%%%%%%%%%%%%%%%%
%%%%%%%%%%%%%%%%%%%%%%%%%%%%%%%%%%%%%%%%%%%%%%%%%%%%%%%%%%%%%%%%%%%%%%%%%%%%%%%%%%%%%%%%%%%%%%%%%%%%%%%%%%%%%%%%%%%%%%%%%%%%%%%%%%%%%%%
\subsection{Solutions to Problem 2: Convex controller design}\label{sec_control_SDPs}

Solving the QMIs in {\itemMultiPERSQMIs} or {\itemMultiPRSQMIs} is still a non-convex program.
In this subsection, we relax the QMIs as an iterative convex program that is easy to solve.
The QMIs of $\AIMFBgain$ and $\AddInvMat$ reduce to LMIs of a new matrix $\dummyHMMat$ with a rank constraint.
The constrained LMIs are relaxed as an iteration of a convex SDP.
Solving the iterative SDP provides an approximate solution to the QMIs.

First, the quadratic terms of $\AIMFBgain$ and $\AddInvMat$ in the QMIs are replaced with linear terms of a rank-one matrix $\dummyHMMat$.
For any $\dummyHMMat \succeq 0 \in {\SetSymMat{\DimX(\DimX+\DimU)}}$,
if $\MyRank{\dummyHMMat}=1$ holds, there exist $\AddInvMat \in \mathbb{R}^{\DimX \times \DimX}$ and $\AIMFBgain \in \mathbb{R}^{\DimU \times \DimX}$ that satisfy	
\begin{align}
\dummyHMMat
&=
\begin{bmatrix}
\VEC{\AddInvMat} \\	 \VEC{\AIMFBgain} 
\end{bmatrix}
\begin{bmatrix}
\VEC{\AddInvMat} \\	 \VEC{\AIMFBgain} 
\end{bmatrix}^{\MyTRANSPO}
.
\label{eq:RankOneEigendecomposition}
\end{align}		
Using this relation transforms the QMIs in {\itemMultiPERSQMIs} and {\itemMultiPRSQMIs} into LMIs with the rank constraint $\MyRank{\dummyHMMat}=1$.
Let us define the following matrix-valued function that is linear in 
$\UNIAIMexpolyVMat \in {\SetSymMat{\DimexX}}$  and
$\dummyHMMat \in {\SetSymMat{\DimX(\DimX+\DimU)}}$
with fixed $\exEMSrate$:
\begin{align}
&{\SDPLMIterms{\IDpoly}{\UNIAIMexpolyVMat}{\dummyHMMat}{\exEMSrate}}
\VisibleColTwo{\nonumber\\&}
:=
\begin{bmatrix}
\exEMSrate^{2} 
\UNIAIMexpolyVMat
&    
{\LMILowLeftblock{\IDpoly}{\dummyHMMat}^{\MyTRANSPO}    }
\\
{\LMILowLeftblock{\IDpoly}{\dummyHMMat}}
&  
\Elimi{ \dummyHHfunc{\dummyHMMat} } + \Elimi{ \dummyHHfunc{\dummyHMMat} }^{\MyTRANSPO}   -  \UNIAIMexpolyVMat
\end{bmatrix}
.\label{eq:def_SDPLMIterms}
\end{align}	
The function ${\LMILowLeftblock{\IDpoly}{\dummyHMMat}}$ is defined by
\begin{align}
{\LMILowLeftblock{\IDpoly}{\dummyHMMat}}
&:=
\NonArgElimi
%\EliMat
\big( 
{\expolyAAMat{\IDpoly}}  \dummyHHfunc{\dummyHMMat}
-{\expolyABMat{\IDpoly}} \dummyHMfunc{\dummyHMMat}
\VisibleColTwo{\nonumber \\&\qquad}
-{\expolyBAMat{\IDpoly}} \dummyMHfunc{\dummyHMMat}
+{\expolyBBMat{\IDpoly}} \dummyMMfunc{\dummyHMMat}
\big)
%\DupMat
,\label{eq:LMILowLeftblock}
\end{align}
where   
$\dummyHHfunc{\dummyHMMat} \in \mathbb{R}^{\DimX^{2} \times \DimX^{2}}$, 
$\dummyHMfunc{\dummyHMMat} \in \mathbb{R}^{\DimU\DimX \times \DimX^{2}}$,
$\dummyMHfunc{\dummyHMMat} \in \mathbb{R}^{\DimU\DimX \times \DimX^{2}}$, and
$\dummyMMfunc{\dummyHMMat} \in \mathbb{R}^{\DimU^{2} \times \DimX^{2}}$ 
are the following linear functions of $\dummyHMMat \in {\SetSymMat{\DimX(\DimX+\DimU)}}$, using its block matrix form:
\begin{align}
{\El{ \dummyHHfunc{\dummyHMMat} }{\vectorWildCard, \DimX(\IDbEl-1) +\IDEl   }}
&:= \VEC{ {\BlockHMMat{\IDEl}{\IDbEl}}   }
,\label{eq:def2_dummyHHfunc}
\\
{\El{ \dummyHMfunc{\dummyHMMat} }{\vectorWildCard, \DimX(\IDbEl-1) +\IDEl   }}
&:= \VEC{ {\BlockHMMat{\DimX+\IDEl}{\IDbEl}}   }
,\\
{\El{ \dummyMHfunc{\dummyHMMat} }{\vectorWildCard, \DimX(\IDbEl-1) +\IDEl   }}
&:= \VEC{ {\BlockHMMat{\IDEl}{\DimX+\IDbEl}}   }
,\\
{\El{ \dummyMMfunc{\dummyHMMat} }{\vectorWildCard, \DimX(\IDbEl-1) +\IDEl   }}
&:= \VEC{ {\BlockHMMat{\DimX+\IDEl}{\DimX+\IDbEl}}   }
,\label{eq:def2_dummyMMfunc}
\\
\begin{bmatrix}
{\BlockHMMat{1}{1}} & \cdots & {\BlockHMMat{1}{2\DimX}} \\
\vdots & \ddots & \vdots \\
{\BlockHMMat{2\DimX}{1}} & \cdots & {\BlockHMMat{2\DimX}{2\DimX}} \\
\end{bmatrix}
&
:=
\dummyHMMat
,\label{eq:def2_BlockHMMat}
\end{align}
where  
${\BlockHMMat{\IDEl}{\IDbEl}} \in \mathbb{R}^{\DimX \times \DimX}$,
${\BlockHMMat{\DimX+\IDEl}{\IDbEl}} \in \mathbb{R}^{\DimU \times \DimX}$,
${\BlockHMMat{\IDEl}{\DimX+\IDbEl}} \in \mathbb{R}^{\DimX \times \DimU}$, and
${\BlockHMMat{\DimX+\IDEl}{\DimX+\IDbEl}} \in \mathbb{R}^{\DimU \times \DimU}$
for any $\IDEl \in \{1,\dots, \DimX\}$ and $\IDbEl \in \{1,\dots, \DimX\}$.
These functions are defined to satisfy the following relations.

%%%%%%%%%%%%%%%%%%%%%%%%%%%%%%%%%%%%%%%%%%%%%%%%%%%%%%%%%%%%%%%%%%%%%%%
%%%%%%%%%%%%%%%%%%%%%%%%%%%%%%%%%%%%%%%%%%%%%%%%%%%%%%%%%%%%%%%%%%%%%%%	
\iffalse\subsubsection{Properties of the linear functions}\fi
\begin{proposition}[{\MyHighlight{Properties of the linear functions}}]\label{thm:dummyHMfunc}
	The condition \eqref{eq:RankOneEigendecomposition} implies the relations:			
	\begin{align}
	\dummyHHfunc{\dummyHMMat} &= \AddInvMat \otimes \AddInvMat
	,\label{eq:def_dummyHHfunc}
	\\
	\dummyHMfunc{\dummyHMMat} &= \AddInvMat \otimes \AIMFBgain
	,\\
	\dummyMHfunc{\dummyHMMat} &= \AIMFBgain \otimes \AddInvMat
	,\\
	\dummyMMfunc{\dummyHMMat} &= \AIMFBgain \otimes \AIMFBgain
	.\label{eq:def_dummyMMfunc}
	\end{align}
\end{proposition}
%%%%%%%%%%%%%%%%%%%%%%%%%%%%%%%%%%%%%%%%%%%%%%%%%%%%%%%%%%%%%%%%%%%%%%%
%%%%%%%%%%%%%%%%%%%%%%%%%%%%%%%%%%%%%%%%%%%%%%%%%%%%%%%%%%%%%%%%%%%%%%%
\begin{proof}
	The proof is described in Appendix \ref{pf:dummyHMfunc}.
\end{proof}
%%%%%%%%%%%%%%%%%%%%%%%%%%%%%%%%%%%%%%%%%%%%%%%%%%%%%%%%%%%%%%%%%%%%%%%
%%%%%%%%%%%%%%%%%%%%%%%%%%%%%%%%%%%%%%%%%%%%%%%%%%%%%%%%%%%%%%%%%%%%%%%

Based on these definitions and Proposition \ref{thm:dummyHMfunc}, we transform the QMIs in {\itemMultiPERSQMIs} and {\itemMultiPRSQMIs} into constrained LMIs as follows.

%%%%%%%%%%%%%%%%%%%%%%%%%%%%%%%%%%%%%%%%%%%%%%%%%%%%%%%%%%%%%%%%%%%%%%%
%%%%%%%%%%%%%%%%%%%%%%%%%%%%%%%%%%%%%%%%%%%%%%%%%%%%%%%%%%%%%%%%%%%%%%%	
\iffalse\subsubsection{LMIs with the rank constraint}\fi
\begin{theorem}[{\MyHighlight{Controller design via constrained LMIs}}]\label{thm:stability_QMIwithRankOne}
For a given $\exEMSrate \in (0, 1)$,
the following {\itemMultiPERSSDPs} and {\itemMultiPRSSDPs}
are equivalent to 
{\itemMultiPERSQMIs} and  {\itemMultiPRSQMIs}, respectively, with the setting \eqref{eq:RankOneEigendecomposition}:
\begin{enumerate}
	\setlength{\itemindent}{10pt}

	\item[{\itemMultiPERSSDPs}]
	There exist ${\AIMexpolyVMat{\IDpoly}} \succ 0 \in {\SetSymMat{\DimexX}}$ for $\IDpoly \in \{1,\dots, \NUMpoly \}$ and  $\dummyHMMat \succeq 0 \in {\SetSymMat{\DimX(\DimX+\DimU)}}$ such that 
	the following LMIs and $\MyRank{\dummyHMMat}=1$ hold:	
	\begin{align}
	\forall \IDpoly \in \{1,\dots, \NUMpoly \}
	,\quad
	{\SDPLMIterms{\IDpoly}{\AIMexpolyVMat{\IDpoly}}{\dummyHMMat}{\exEMSrate}}
	\succeq 0
	.\label{eq:expanded_ERS_SDPLMI}
	\end{align}

	\item[{\itemMultiPRSSDPs}]
	There exist ${\AIMexpolyVMat{\IDpoly}} \succ 0 \in {\SetSymMat{\DimexX}}$ for $\IDpoly \in \{1,\dots, \NUMpoly \}$ and $\dummyHMMat \succeq 0 \in {\SetSymMat{\DimX(\DimX+\DimU)}}$ such that 
	the following LMIs and $\MyRank{\dummyHMMat}=1$ hold:	
	\begin{align}
	\forall \IDpoly \in \{1,\dots, \NUMpoly \}
	,\quad
	{\SDPLMIterms{\IDpoly}{\AIMexpolyVMat{\IDpoly}}{\dummyHMMat}{1}}
	\succ 0
	.\label{eq:expanded_RS_SDPLMI}
	\end{align}

\end{enumerate}	
\end{theorem}
%%%%%%%%%%%%%%%%%%%%%%%%%%%%%%%%%%%%%%%%%%%%%%%%%%%%%%%%%%%%%%%%%%%%%%%
%%%%%%%%%%%%%%%%%%%%%%%%%%%%%%%%%%%%%%%%%%%%%%%%%%%%%%%%%%%%%%%%%%%%%%%
\begin{proof}
The proof is described in Appendix \ref{pf:stability_QMIwithRankOne}.
\end{proof}
%%%%%%%%%%%%%%%%%%%%%%%%%%%%%%%%%%%%%%%%%%%%%%%%%%%%%%%%%%%%%%%%%%%%%%%
%%%%%%%%%%%%%%%%%%%%%%%%%%%%%%%%%%%%%%%%%%%%%%%%%%%%%%%%%%%%%%%%%%%%%%%

Note that the LMIs in {\itemMultiPERSSDPs} and {\itemMultiPRSSDPs} are just linear in ${\AIMexpolyVMat{\IDpoly}}$ and $\dummyHMMat$.
A remaining challenge is to exclude the rank constraint $\MyRank{\dummyHMMat}=1$ that invokes the non-convexity in solving the LMIs.

In the following, we relax the rank-constrained LMIs as an iterative convex SDP based on the ideas of capped trace norm minimization \cite{ZhangJMLR10,Liu19}.
Let us define the absolute sum of the eigenvalues of symmetric $\dummyHMMat \succeq 0 \in {\SetSymMat{\DimX(\DimX+\DimU)}}$, except the maximum eigenvalue ${\iEig{1}{ {\dummyHMMat} }}$, as follows: 
\begin{align}
{\rankoneErr{\dummyHMMat}}
&:=
\sum_{\IDEl=2}^{\DimX(\DimX+\DimU)}
|{\iEig{\IDEl}{ {\dummyHMMat} }}|
\geq 0
. \label{eq:def_rankoneErr}
\end{align}
Solving the LMIs in {\itemMultiPERSSDPs} and {\itemMultiPRSSDPs} with the rank constraint reduces to minimize ${\rankoneErr{\dummyHMMat}}$ under the LMIs because $\MyRank{\dummyHMMat}=1$ is equivalent to ${\rankoneErr{\dummyHMMat}}=0$ with ${\iEig{1}{ {\dummyHMMat} }}> 0$.
Whereas ${\rankoneErr{\dummyHMMat}}$ is nonlinear in $\dummyHMMat$,
we focus on the fact that the minimization of ${\rankoneErr{\dummyHMMat}}$ has a form similar to the capped trace norm minimization \cite{Liu19}.
This motivates us to employ the following function ${\rankoneApproxErr{\dummyHMMat}{\dummyPreHMMat}}$ of
$\dummyHMMat \in {\SetSymMat{\DimX(\DimX+\DimU)}}$ and $\dummyPreHMMat \in {\SetSymMat{\DimX(\DimX+\DimU)}}$
 for approximating ${\rankoneErr{\dummyHMMat}}$:
\begin{align}
{\rankoneApproxErr{\dummyHMMat}{\dummyPreHMMat}}
:=
\TRACE{
\dummyHMMat
}
-
{\iEigVec{1}{\dummyPreHMMat}}^{\MyTRANSPO} 
\dummyHMMat  
{\iEigVec{1}{\dummyPreHMMat}} 
%}
,
\end{align}
where $\iEigVec{1}{\dummyPreHMMat}$ is the unit eigenvector corresponding to $\iEig{1}{\dummyPreHMMat}$ defined in Section \ref{sec_notation}.
We show that ${\rankoneApproxErr{\dummyHMMat}{\dummyPreHMMat}}$ is an upper bound of ${\rankoneErr{\dummyHMMat}}$, as follows.

%%%%%%%%%%%%%%%%%%%%%%%%%%%%%%%%%%%%%%%%%%%%%%%%%%%%%%%%%%%%%%%%%%%%%%%
%%%%%%%%%%%%%%%%%%%%%%%%%%%%%%%%%%%%%%%%%%%%%%%%%%%%%%%%%%%%%%%%%%%%%%%	
\iffalse\subsubsection{Approximate linear functions}\fi
\begin{lemma}[{\MyHighlight{Approximate linear functions}}]\label{thm:multiSDPfunc}
Suppose that $\dummyPreHMMat \neq 0$ holds.
For any $\dummyHMMat\succeq 0 $ and $\dummyPreHMMat \succeq 0$,
the following relations hold:	
\begin{align}
{\rankoneApproxErr{\dummyHMMat}{\dummyHMMat}}
&=
{\rankoneErr{\dummyHMMat}}
,
\\
{\rankoneApproxErr{\dummyHMMat}{\dummyPreHMMat}}
&\geq
{\rankoneErr{\dummyHMMat}}
.
\end{align}		
\end{lemma}
%%%%%%%%%%%%%%%%%%%%%%%%%%%%%%%%%%%%%%%%%%%%%%%%%%%%%%%%%%%%%%%%%%%%%%%
%%%%%%%%%%%%%%%%%%%%%%%%%%%%%%%%%%%%%%%%%%%%%%%%%%%%%%%%%%%%%%%%%%%%%%%	
\begin{proof}
	The proof is described in Appendix \ref{pf:multiSDPfunc}.
\end{proof}
%%%%%%%%%%%%%%%%%%%%%%%%%%%%%%%%%%%%%%%%%%%%%%%%%%%%%%%%%%%%%%%%%%%%%%%
%%%%%%%%%%%%%%%%%%%%%%%%%%%%%%%%%%%%%%%%%%%%%%%%%%%%%%%%%%%%%%%%%%%%%%%	

In addition to the above relations,
an efficient property is that ${\rankoneApproxErr{\dummyHMMat}{\dummyPreHMMat}}$ is linear in $\dummyHMMat$ with a fixed $\dummyPreHMMat$, which is convex.
Using these properties, we solve the following convex SDP iteratively instead of {\itemMultiPERSSDPs} or {\itemMultiPRSSDPs}.
For each $\IDite$-th iteration,
let ${\dummyIteHMMat{\IDite}}$ be a solution to the SDP:
\begin{align}
\min_{
	\dummyHMMat,  {\AIMexpolyVMat{1}}, \dots , {\AIMexpolyVMat{\NUMpoly}} 
 } 
{\rankoneApproxErr{\dummyHMMat}{\dummyIteHMMat{\IDite-1}}}
\quad 
\mathrm{s.t.}
\qquad 
\VisibleColTwo{&
\nonumber\\}
\forall  \IDpoly \in \{1,\dots, \NUMpoly \},
\quad
{\SDPLMIterms{\IDpoly}{\AIMexpolyVMat{\IDpoly}}{\dummyHMMat}{\exEMSrate}}  &\succeq  \smallmarginLMIs {\Identity{2\DimexX}}
,\nonumber\\
\forall  \IDpoly \in \{1,\dots, \NUMpoly \},
\quad
{\AIMexpolyVMat{\IDpoly}} &\succeq  \smallmarginLMIs {\Identity{\DimexX}}
,\nonumber\\
\dummyHMMat &\succeq 0
,\nonumber \\
\TRACE{ \dummyHMMat } & \leq \UBdummyHMMat
, \label{eq:def_multiSDP}
\end{align}	
where $\smallmarginLMIs \geq 0$ and $\UBdummyHMMat>0$ are free parameters.
The first and second inequalities describe {\itemMultiPERSSDPs} or {\itemMultiPRSSDPs}, where the positive value of $\smallmarginLMIs$ helps satisfy the positive (semi)definiteness even if numerical errors occur.
The inequality $\TRACE{ \dummyHMMat } \leq \UBdummyHMMat$ is employed to avoid divergence of $\dummyHMMat$.
Let ${\dummyIteHMMat{1}}$ be a solution to \eqref{eq:def_multiSDP} when replacing ${\rankoneApproxErr{\dummyHMMat}{\dummyIteHMMat{\IDite}}}$ with $\TRACE{\dummyHMMat}$, which is similar to trace norm minimization \cite{Bach08}.
Because the minimization in \eqref{eq:def_multiSDP} is convex, the following optimality is assumed:
\begin{align}
{\rankoneApproxErr{\dummyIteHMMat{\IDite}}{\dummyIteHMMat{\IDite-1}}}
\leq 
{\rankoneApproxErr{\dummyIteHMMat{\IDite-1}}{\dummyIteHMMat{\IDite-1}}}
.  \label{eq:optimality_multiSDP}
\end{align}
After solving the SDP \eqref{eq:def_multiSDP} iteratively, its solution $\optdummyHMMat={\dummyIteHMMat{\IDite}}$ can be successfully approximated using the rank-one matrix $\approxdummyHMMat$ associated with the maximum eigenvalue ${\iEig{1}{\optdummyHMMat}} $:
\begin{align}
\approxdummyHMMat 
&:=
{\iEig{1}{\optdummyHMMat}} 
{\iEigVec{1}{\optdummyHMMat}} 
{\iEigVec{1}{\optdummyHMMat}}^{\MyTRANSPO}
.\label{eq:RankOneEigendecompositionRank1_approx}
\end{align}
Note that the approximation residual $\optdummyHMMat-\approxdummyHMMat
=
\sum_{\IDEl=2}^{\DimX(\DimX+\DimU)}
{\iEig{\IDEl}{\optdummyHMMat}} 
{\iEigVec{\IDEl}{\optdummyHMMat}} 
{\iEigVec{\IDEl}{\optdummyHMMat}}^{\MyTRANSPO}$ can be negligible if ${\rankoneErr{\optdummyHMMat}}$ in \eqref{eq:def_rankoneErr} is sufficiently decreased.

Algorithm \ref{alg1} summarizes the proposed iterative SDP.
The iteration is terminated when ${\rankoneErr{\dummyIteHMMat{\IDite}}}
\geq {\rankoneErr{\dummyIteHMMat{\IDite-1}}}
-\MultiSDPterminationVal
$ holds, where $\MultiSDPterminationVal\geq 0$ is a small threshold. 
We derive the following result to justify Algorithm \ref{alg1} in the sense that ${\rankoneErr{\dummyIteHMMat{\IDite}}}$ is successfully decreased.

%%%%%%%%%%%%%%%%%%%%%%%%%%%%%%%%%%%%%%%%%%%%%%%%%%%%%%%%%%%%%%%%%%%%%%%
%%%%%%%%%%%%%%%%%%%%%%%%%%%%%%%%%%%%%%%%%%%%%%%%%%%%%%%%%%%%%%%%%%%%%%%	
\begin{theorem}[{\MyHighlight{Monotonically non-increasing iteration}}]\label{{thm:multiSDPmonodec}}
	For each $\IDite\geq 1$, supposing that \eqref{eq:optimality_multiSDP} and  ${\dummyIteHMMat{\IDite-1}} \neq 0$ hold, 
	\begin{align}
	&
	{\rankoneErr{\dummyIteHMMat{\IDite}}}
	\leq
	{\rankoneErr{\dummyIteHMMat{\IDite-1}}}
	,\label{eq:optimality_multiSDP2}
	\end{align}
holds, where the strict inequality of \eqref{eq:optimality_multiSDP} implies that of \eqref{eq:optimality_multiSDP2}.
\end{theorem}
\begin{proof}
	Using Lemma \ref{thm:multiSDPfunc} with \eqref{eq:optimality_multiSDP} yields 
	\begin{align}
	{\rankoneErr{\dummyIteHMMat{\IDite}}}
	\leq
	{\rankoneApproxErr{\dummyIteHMMat{\IDite}}{\dummyIteHMMat{\IDite-1}}}
	\leq 
	{\rankoneApproxErr{\dummyIteHMMat{\IDite-1}}{\dummyIteHMMat{\IDite-1}}}
	%\nonumber\\&
	=
	{\rankoneErr{\dummyIteHMMat{\IDite-1}}}
	.
	\end{align}	
	This completes the proof.
\end{proof}	
%%%%%%%%%%%%%%%%%%%%%%%%%%%%%%%%%%%%%%%%%%%%%%%%%%%%%%%%%%%%%%%%%%%%%%%
%%%%%%%%%%%%%%%%%%%%%%%%%%%%%%%%%%%%%%%%%%%%%%%%%%%%%%%%%%%%%%%%%%%%%%%	

%%%%%%%%%%%%%%%%%%%%%%%%%%%%%%%%%%%%%%%%%%%%%%%%%%%%%%%%%%%%%%%%%%%%%%%
%%%%%%%%%%%%%%%%%%%%%%%%%%%%%%%%%%%%%%%%%%%%%%%%%%%%%%%%%%%%%%%%%%%%%%%	
\iffalse \subsubsection{Algorithm} \fi
\begin{algorithm}[!t]    
	\renewcommand{\algorithmicrequire}{\textbf{Input:}}
	\renewcommand{\algorithmicensure}{\textbf{Output:}}                       
	\caption{Proposed controller design with the iterative SDP}  
	\label{alg1} 
	\begin{algorithmic}[1]                  
		\REQUIRE
		Vertices ${\polySMvecAB{\IDpoly}}$ for $\IDpoly \in \{1,\dots, \NUMpoly \}$ and $\NUMpoly$ of the SMP system,
		$\exEMSrate$, 
		the free parameters $\smallmarginLMIs$ and $\UBdummyHMMat$ used in the SDP \eqref{eq:def_multiSDP},
		and 
		$\MultiSDPterminationVal$
		\ENSURE Rank-one matrix $\approxdummyHMMat$ and ${\AIMexpolyVMat{\IDpoly}}$
				
		\STATE
		Calculate ${\expolyAAMat{\IDpoly}}$,
		${\expolyABMat{\IDpoly}}$,
		${\expolyBAMat{\IDpoly}}$, and 
		${\expolyBBMat{\IDpoly}}$ according to \eqref{eq:def_expolyAAMat}--\eqref{eq:def_expolyBBMat} using ${\polySMvecAB{\IDpoly}}$ and $\NUMpoly$

		\STATE
		Develop the matrix-valued function ${\SDPLMIterms{\IDpoly}{\UNIAIMexpolyVMat}{\dummyHMMat}{\exEMSrate}}$ in \eqref{eq:def_SDPLMIterms} using 
		 ${\expolyAAMat{\IDpoly}}$, ${\expolyABMat{\IDpoly}}$, ${\expolyBAMat{\IDpoly}}$, ${\expolyBBMat{\IDpoly}}$, and $\exEMSrate$		
				
		\STATE	
		Obtain ${\dummyIteHMMat{1}}$ by solving the SDP \eqref{eq:def_multiSDP} when replacing ${\rankoneApproxErr{\dummyHMMat}{\dummyIteHMMat{\IDite}}}$ with $\TRACE{\dummyHMMat}$,
		
		\STATE
		Set $\IDite \gets 1$ 
		
		\REPEAT
		\STATE
		Set $\IDite \gets \IDite+1$ 
		
		\STATE	
		Obtain ${\dummyIteHMMat{\IDite}}$ and ${\AIMexpolyVMat{\IDpoly}}$ by solving the SDP \eqref{eq:def_multiSDP} with substituting ${\dummyIteHMMat{\IDite-1}}$.

		\UNTIL{ 			
			${\rankoneErr{\dummyIteHMMat{\IDite}}}
			\geq {\rankoneErr{\dummyIteHMMat{\IDite-1}}}
			-\MultiSDPterminationVal
			$ 
		} 
		\STATE
		Obtain the approximately rank-one matrix $\optdummyHMMat \gets {\dummyIteHMMat{\IDite}}$

		\STATE
		Obtain the rank-one matrix $\approxdummyHMMat$ by \eqref{eq:RankOneEigendecompositionRank1_approx}
		
	\end{algorithmic}		
\end{algorithm}
%%%%%%%%%%%%%%%%%%%%%%%%%%%%%%%%%%%%%%%%%%%%%%%%%%%%%%%%%%%%%%%%%%%%%%%
%%%%%%%%%%%%%%%%%%%%%%%%%%%%%%%%%%%%%%%%%%%%%%%%%%%%%%%%%%%%%%%%%%%%%%%

If {\itemMultiPERSSDPs} or {\itemMultiPRSSDPs} is satisfied by $\dummyHMMat=\approxdummyHMMat$ in \eqref{eq:RankOneEigendecompositionRank1_approx} and a solution ${\AIMexpolyVMat{\IDpoly}}$ (in the $\IDite$-th iteration) to the SDP,
a stabilizing feedback gain $\FBgain$ is obtained.
Consequently, solutions to Problem 2 are summarized, using Theorems \ref{thm:stability_QMI} and \ref{thm:stability_QMIwithRankOne} and Corollary \ref{thm:solution_to_Problem1}.

%%%%%%%%%%%%%%%%%%%%%%%%%%%%%%%%%%%%%%%%%%%%%%%%%%%%%%%%%%%%%%%%%%%%%%%
%%%%%%%%%%%%%%%%%%%%%%%%%%%%%%%%%%%%%%%%%%%%%%%%%%%%%%%%%%%%%%%%%%%%%%%	
\iffalse\subsubsection{Solutions to Problem 2}\fi
\begin{corollary}[{\MyHighlight{Solutions to Problem 2}}]\label{thm:solution_to_Problem2}
The SMP system {\eqref{eq:def_CLsys}} is exponentially robustly MS stable (resp. robustly MS stable) if at least one of the following (i)--(iv) is satisfied and if $\FBgain$ is given by \eqref{eq:AIMFBgain_to_FBgain} and \eqref{eq:RankOneEigendecomposition}:
		\begin{enumerate}
			\item 
			{\itemMultiPERSQMIs} (resp. {\itemMultiPRSQMIs}) is solved and the system is TI. 
			\item 
			{\itemMultiPERSQMIs} (resp. {\itemMultiPRSQMIs}) is solved when replacing ${\AIMexpolyVMat{\IDpoly}}$ for $\IDpoly \in \{1,\dots, \NUMpoly \}$ with an identical $\UNIAIMexpolyVMat$.
			\item
			Solutions  $\dummyHMMat=\approxdummyHMMat$ and ${\AIMexpolyVMat{\IDpoly}}$ to the iterative SDP \eqref{eq:def_multiSDP}
			satisfy {\itemMultiPERSSDPs} (resp. {\itemMultiPRSSDPs}) and the system is TI.			 
			\item
			Solutions  $\dummyHMMat=\approxdummyHMMat$ and $\UNIAIMexpolyVMat$ to the iterative SDP \eqref{eq:def_multiSDP}
			satisfy {\itemMultiPERSSDPs}  (resp. {\itemMultiPRSSDPs}) when replacing ${\AIMexpolyVMat{\IDpoly}}$ for $\IDpoly \in \{1,\dots, \NUMpoly \}$ with an identical $\UNIAIMexpolyVMat$.
		\end{enumerate}
\end{corollary}
%%%%%%%%%%%%%%%%%%%%%%%%%%%%%%%%%%%%%%%%%%%%%%%%%%%%%%%%%%%%%%%%%%%%%%%
%%%%%%%%%%%%%%%%%%%%%%%%%%%%%%%%%%%%%%%%%%%%%%%%%%%%%%%%%%%%%%%%%%%%%%%

%%%%%%%%%%%%%%%%%%%%%%%%%%%%%%%%%%%%%%%%%%%%%%%%%%%%%%%%%%%%%%%%%%%%%%%%%%%%%%%%%%%%%%%%%%%%%%%%%%%%%%%%%%%%%%%%%%%%%%%%%%%%%%%%%%%%%%%
%%%%%%%%%%%%%%%%%%%%%%%%%%%%%%%%%%%%%%%%%%%%%%%%%%%%%%%%%%%%%%%%%%%%%%%%%%%%%%%%%%%%%%%%%%%%%%%%%%%%%%%%%%%%%%%%%%%%%%%%%%%%%%%%%%%%%%%
%%%%%%%%%%%%%%%%%%%%%%%%%%%%%%%%%%%%%%%%%%%%%%%%%%%%%%%%%%%%%%%%%%%%%%%%%%%%%%%%%%%%%%%%%%%%%%%%%%%%%%%%%%%%%%%%%%%%%%%%%%%%%%%%%%%%%%%
\section{Demonstration with a numerical simulation}\label{sec_sim}

Section \ref{sec_variety_SMP} demonstrates the applicability of the proposed SMP systems.
The proposed method to design stabilizing controllers is evaluated through a numerical example.
Sections \ref{sec_sim_settings} and \ref{sec_sim_results} describe the simulation settings and results, respectively.

%%%%%%%%%%%%%%%%%%%%%%%%%%%%%%%%%%%%%%%%%%%%%%%%%%%%%%%%%%%%%%%%%%%%%%%%%%%%%%%%%%%%%%%%%%%%%%%%%%%%%%%%%%%%%%%%%%%%%%%%%%%%%%%%%%%%%%%
%%%%%%%%%%%%%%%%%%%%%%%%%%%%%%%%%%%%%%%%%%%%%%%%%%%%%%%%%%%%%%%%%%%%%%%%%%%%%%%%%%%%%%%%%%%%%%%%%%%%%%%%%%%%%%%%%%%%%%%%%%%%%%%%%%%%%%%
%%%%%%%%%%%%%%%%%%%%%%%%%%%%%%%%%%%%%%%%%%%%%%%%%%%%%%%%%%%%%%%%%%%%%%%%%%%%%%%%%%%%%%%%%%%%%%%%%%%%%%%%%%%%%%%%%%%%%%%%%%%%%%%%%%%%%%%
\subsection{Variety of second moment polytopes}\label{sec_variety_SMP}

The SMP systems in Definition \ref{def:TVSMP} can describe various classes of uncertain stochastic linear systems, which contain familiar existing systems.
We derive certain examples of such classes as follows.

%%%%%%%%%%%%%%%%%%%%%%%%%%%%%%%%%%%%%%%%%%%%%%%%%%%%%%%%%%%%%%%%%%%%%%%
%%%%%%%%%%%%%%%%%%%%%%%%%%%%%%%%%%%%%%%%%%%%%%%%%%%%%%%%%%%%%%%%%%%%%%%	
\iffalse\subsubsection{examples of TI SMP}\fi
\begin{theorem}[{\MyHighlight{Examples of TI SMP systems}}]\label{thm:TISMP_examples}	
	Suppose that ${\TVUnc{\MyT}}$ is TI (i.e., ${\TVUnc{\MyT}}=\TIUnc$) 
	and that $\DomTVUnc$ is given by
	\begin{align}
	\DomTVUnc
	&=
	\Big\{ 
	\TIUnc \in \mathbb{R}^{\DimTVUnc} 
	\Bigg|
	\forall \IDpoly, 
	{\El{\TIUnc}{\IDpoly}}
	\geq 0
	,
	\sum_{\IDpoly=1}^{\DimTVUnc} {\El{\TIUnc}{\IDpoly}}  = 1
	\Big\}
	. \label{eq:DomTVUnc_is_polytope}
	\end{align}
The following classes are represented as SMP systems.
	\begin{enumerate}
		\item\label{item_iid}
		i.i.d stochastic systems \cite{ItoTAC19}: 
		Suppose that ${\vecTIAB{\MyT}}={\NonArgvecAB{\MyT}}$ is i.i.d. with respect to $\MyT$ and independent of $\TIUnc$.
		The system \eqref{eq:def_sys} is TI SMP with
		\begin{align}
		\NUMpoly&=1
		,\label{eq:iid_NUMpoly}
		\\
		{\MapPolyW{\TIUnc}}&=1
		,\label{eq:iid_MapPolyW}
		\\
		{\polySMvecAB{1}}&= {\Expect[]{ {\NonArgvecAB{\MyT}} {\NonArgvecAB{\MyT}^{\MyTRANSPO}} }}
		.\label{eq:iid_polySMvecAB}
		\end{align}

		\item\label{item_det_polytope}
		Deterministic polytopic systems \cite{Oliveira99}:
		Suppose that ${\vecTIAB{\MyT}}$ and $\DomTVUnc$ are given by
		\begin{align}
		{\vecTIAB{\MyT}}
		& = \sum_{\IDpoly =1}^{\DimTVUnc} {\El{\TIUnc}{\IDpoly}} {\polyvecAB{\IDpoly}}
		,
		\end{align}
		where ${\polyvecAB{\IDpoly}}$ for $\IDpoly \in \{1,\dots, \DimTVUnc\}$ are deterministic vertices.
		Supposing that 
		the expectation
		{\;\;\;\;\;}
		 ${\CondExpectTV{   {\vecTVAB{\MyT}}{\vecTVAB{\MyT}^{\MyTRANSPO}}    }{\MyT}}$ can be replaced with the deterministic value $   {\vecTVAB{\MyT}}{\vecTVAB{\MyT}^{\MyTRANSPO}}   $,
		the system \eqref{eq:def_sys} is  TI SMP with
		\begin{align}
		\NUMpoly&=\DimTVUnc^2
		,\label{eq:setting1_NUMexpoly}
		\\
		{\MapPolyW{\TIUnc}}&= {\VEC{ {\TIUnc}  {\TIUnc^{\MyTRANSPO}}  }}
		,\label{eq:setting1_polyW}
		\\
		{\polySMvecAB{ \DimTVUnc(\IDbpoly-1) + \IDpoly  }}&
		=
		\frac{
		{\polyvecAB{\IDpoly}}   {\polyvecAB{\IDbpoly}}^{\MyTRANSPO}
		+
		{\polyvecAB{\IDbpoly}}   {\polyvecAB{\IDpoly}}^{\MyTRANSPO}
		}{ 2}
		.
		\end{align}

		\item\label{item_random_polytope}
		Random polytopic systems \cite{HosoeTAC18}:
		Suppose that ${\vecTIAB{\MyT}}$ is given by	
		\begin{align}
		{\vecTIAB{\MyT}}
		& = \sum_{\IDpoly =1}^{\DimTVUnc} {\El{\TIUnc}{\IDpoly}} {\polyRandvecAB{\IDpoly}}
		, \label{eq:def_random_polytope}
		\end{align}
		where the stochastic vertices ${\polyRandvecAB{\IDpoly}}$ for $\IDpoly \in \{1,\dots, \DimTVUnc\}$ are i.i.d. with respect to $\MyT$ and independent of $\TIUnc$.
		The system \eqref{eq:def_sys} is  TI SMP with \eqref{eq:setting1_NUMexpoly}, \eqref{eq:setting1_polyW}, and
		\begin{align}
		{\polySMvecAB{ \DimTVUnc(\IDbpoly-1) + \IDpoly  }}&
		=
		\frac{\Expect[]{ 
				{\polyRandvecAB{\IDpoly}}{\polyRandvecAB{\IDbpoly}}^{\MyTRANSPO}
				+
				{\polyRandvecAB{\IDbpoly}}{\polyRandvecAB{\IDpoly}}^{\MyTRANSPO}	
		}}{2}
		.\label{eq:setting1_polySMvecAB}
		\end{align}

		\item\label{item_uncMeanCov_NUMpoly}
		Systems with uncertain mean and covariance:
		Suppose that ${\vecTIAB{\MyT}}$ is given by
		\begin{align}
		{\CondExpectTI{\vecTIAB{\MyT}}}
		&
		%={\MeanvecAB{\TIUnc}}
		= \sum_{\IDpoly =1}^{\DimTVUnc} {\El{\TIUnc}{\IDpoly}} {\polyMeanvecAB{\IDpoly}}
		,\label{eq:mean_UncMeanCovSys}
		\\
		{\CondCovTI[\big]{\vecTIAB{\MyT}}}
		&
		= \sum_{\IDpoly =1}^{\DimTVUnc} {\El{\TIUnc}{\IDpoly}}  {\polyCovvecAB{\IDpoly}} 
		,\label{eq:cov_UncMeanCovSys}
		\end{align}
		where ${\polyMeanvecAB{\IDpoly}}$ and ${\polyCovvecAB{\IDpoly}} \in {\SetSymMat{\DimX(\DimX+\DimU)}}$  for $\IDpoly \in \{1,\dots, \DimTVUnc\}$ are deterministic.
		The system \eqref{eq:def_sys} is  TI SMP with \eqref{eq:setting1_NUMexpoly}, \eqref{eq:setting1_polyW}, and
		\begin{align}
		{\polySMvecAB{  \DimTVUnc(\IDbpoly-1) + \IDpoly   }}
		&=\frac{
		{\polyMeanvecAB{\IDpoly}}{\polyMeanvecAB{\IDbpoly}}^{\MyTRANSPO} 
		+ 
		{\polyMeanvecAB{\IDpoly}}{\polyMeanvecAB{\IDbpoly}}^{\MyTRANSPO}
		}{2}
		 + {\polyCovvecAB{\IDpoly}}
		.
		\end{align}

	\end{enumerate}

\end{theorem}
%%%%%%%%%%%%%%%%%%%%%%%%%%%%%%%%%%%%%%%%%%%%%%%%%%%%%%%%%%%%%%%%%%%%%%%
%%%%%%%%%%%%%%%%%%%%%%%%%%%%%%%%%%%%%%%%%%%%%%%%%%%%%%%%%%%%%%%%%%%%%%%	
\begin{proof}
	The proof is described in Appendix \ref{pf:TISMP_examples}.
\end{proof}	
%%%%%%%%%%%%%%%%%%%%%%%%%%%%%%%%%%%%%%%%%%%%%%%%%%%%%%%%%%%%%%%%%%%%%%%
%%%%%%%%%%%%%%%%%%%%%%%%%%%%%%%%%%%%%%%%%%%%%%%%%%%%%%%%%%%%%%%%%%%%%%%

%%%%%%%%%%%%%%%%%%%%%%%%%%%%%%%%%%%%%%%%%%%%%%%%%%%%%%%%%%%%%%%%%%%%%%%
%%%%%%%%%%%%%%%%%%%%%%%%%%%%%%%%%%%%%%%%%%%%%%%%%%%%%%%%%%%%%%%%%%%%%%%	
\iffalse\subsubsection{examples of TV SMP}\fi
\begin{corollary}[{\MyHighlight{Examples of TV SMP systems}}]\label{thm:TVSMP_examples}	
Suppose that ${\TVUnc{\MyT}}$ is TV and that $\DomTVUnc$ is given by \eqref{eq:DomTVUnc_is_polytope}.
Then, the statements \ref{item_det_polytope}--\ref{item_uncMeanCov_NUMpoly} in Theorem \ref{thm:TVSMP_examples} hold if ${\TIUnc}$, ${\vecTIAB{\MyT}}$, and the TI property are replaced with ${\TVUnc{\MyT}}$, ${\vecTVAB{\MyT}}$, and the TV property, respectively.	
\end{corollary}
%%%%%%%%%%%%%%%%%%%%%%%%%%%%%%%%%%%%%%%%%%%%%%%%%%%%%%%%%%%%%%%%%%%%%%%
%%%%%%%%%%%%%%%%%%%%%%%%%%%%%%%%%%%%%%%%%%%%%%%%%%%%%%%%%%%%%%%%%%%%%%%	
\begin{proof}
By replacing $\TIUnc$ with $\TVUnc{\MyT}$, we can prove the statements in a manner similar to Theorem \ref{thm:TVSMP_examples}. 	
\end{proof}	
%%%%%%%%%%%%%%%%%%%%%%%%%%%%%%%%%%%%%%%%%%%%%%%%%%%%%%%%%%%%%%%%%%%%%%%
%%%%%%%%%%%%%%%%%%%%%%%%%%%%%%%%%%%%%%%%%%%%%%%%%%%%%%%%%%%%%%%%%%%%%%%

While SMP systems represent random polytopes in Theorem \ref{thm:TISMP_examples} \ref{item_random_polytope}, 
the following statement shows a difference between the SMP systems and random polytopes.

%%%%%%%%%%%%%%%%%%%%%%%%%%%%%%%%%%%%%%%%%%%%%%%%%%%%%%%%%%%%%%%%%%%%%%%
%%%%%%%%%%%%%%%%%%%%%%%%%%%%%%%%%%%%%%%%%%%%%%%%%%%%%%%%%%%%%%%%%%%%%%%	
\iffalse\subsubsection{Comparison with random polytopes}\fi
\begin{proposition}[{\MyHighlight{Comparison with random polytopes}}]\label{thm:SMP_vs_random_polytope}
	There exists a TI SMP system with ${\anothervecTIAB{\MyT}}$ that is not equivalent to every random polytope with ${\vecTIAB{\MyT}}$ given in Theorem \ref{thm:TISMP_examples} \ref{item_random_polytope}, 
	in the sense that
	there is no setting of ${\polyRandvecAB{\IDpoly}}$ for $\IDpoly \in \{1,\dots, \DimTVUnc\}$ satisfying the following conditions:
	\begin{align}
	&
	\forall \anotherTIUnc \in \anotherDomTVUnc
	,
	\exists \TIUnc \in \DomTVUnc
	, \;
	{\anotherCondExpectTI{    {\anothervecTIAB{\MyT}} {\anothervecTIAB{\MyT}^{\MyTRANSPO}}        }}
	=
	{\CondExpectTI[\big]{   {\vecTIAB{\MyT}}{\vecTIAB{\MyT}^{\MyTRANSPO}}    }}
	,\label{eq:SMPvsRP_goalA} \\
	&
	\forall  \TIUnc \in \DomTVUnc
	,
	\exists  \anotherTIUnc \in \anotherDomTVUnc
	, \;
	{\anotherCondExpectTI{    {\anothervecTIAB{\MyT}} {\anothervecTIAB{\MyT}^{\MyTRANSPO}}        }}
	=
	{\CondExpectTI[\big]{   {\vecTIAB{\MyT}}{\vecTIAB{\MyT}^{\MyTRANSPO}}    }}
	,\label{eq:SMPvsRP_goalB} 
	\end{align}
	where the notation $\overline{(\bullet)}$ denotes variables different from $(\bullet)$.
\end{proposition}
%%%%%%%%%%%%%%%%%%%%%%%%%%%%%%%%%%%%%%%%%%%%%%%%%%%%%%%%%%%%%%%%%%%%%%%
%%%%%%%%%%%%%%%%%%%%%%%%%%%%%%%%%%%%%%%%%%%%%%%%%%%%%%%%%%%%%%%%%%%%%%%	
\begin{proof}
	The proof is described in Appendix \ref{pf:SMP_vs_random_polytope}.
\end{proof}
%%%%%%%%%%%%%%%%%%%%%%%%%%%%%%%%%%%%%%%%%%%%%%%%%%%%%%%%%%%%%%%%%%%%%%%
%%%%%%%%%%%%%%%%%%%%%%%%%%%%%%%%%%%%%%%%%%%%%%%%%%%%%%%%%%%%%%%%%%%%%%%	

%%%%%%%%%%%%%%%%%%%%%%%%%%%%%%%%%%%%%%%%%%%%%%%%%%%%%%%%%%%%%%%%%%%%%%%%%%%%%%%%%%%%%%%%%%%%%%%%%%%%%%%%%%%%%%%%%%%%%%%%%%%%%%%%%%%%%%%
%%%%%%%%%%%%%%%%%%%%%%%%%%%%%%%%%%%%%%%%%%%%%%%%%%%%%%%%%%%%%%%%%%%%%%%%%%%%%%%%%%%%%%%%%%%%%%%%%%%%%%%%%%%%%%%%%%%%%%%%%%%%%%%%%%%%%%%
%%%%%%%%%%%%%%%%%%%%%%%%%%%%%%%%%%%%%%%%%%%%%%%%%%%%%%%%%%%%%%%%%%%%%%%%%%%%%%%%%%%%%%%%%%%%%%%%%%%%%%%%%%%%%%%%%%%%%%%%%%%%%%%%%%%%%%%
\subsection{Simulation settings}\label{sec_sim_settings}

\newcommand{\simRangeOneMinTISto}{0.5}
\newcommand{\simRangeOneMaxTISto}{1.1}
\newcommand{\simRangeTwoMinTISto}{0.7}
\newcommand{\simRangeTwoMaxTISto}{1.3}

\newcommand{\simPoly}[2]{
	\begin{bmatrix}
		#1 & 0.2 + {\El{\TVSto{\MyT}}{2}} \\
		#2 {\El{\TVSto{\MyT}}{1}}  & 0.9
	\end{bmatrix}	
}
\newcommand{\simVecPoly}[2]{
	[
	#1 								,\;\;
	#2 {\El{\TVSto{\MyT}}{1}}  		,\;\; 
	0.2 + {\El{\TVSto{\MyT}}{2}} 	,\;\;
	0.9								,\;\; 
	1 								,\;\; 
	0
	]^{\MyTRANSPO}
}

We consider the SMP system \eqref{eq:def_CLsys} given by Theorem \ref{thm:TISMP_examples} \ref{item_uncMeanCov_NUMpoly} with 
$\DimTVUnc=3$, $\DimX=2$, $\DimU=1$, and the following settings:
\begin{align}
{\polyMeanvecAB{1}}
&=
[0.90,\; 0,\; 0.2,\;  0.9,\; 0,\; 1.0]^{\MyTRANSPO}
,\\
{\polyMeanvecAB{2}}
&=
[1.00,\; 0,\; 0.2,\; 0.9,\; 0,\; 0.8]^{\MyTRANSPO}
,\\
{\polyMeanvecAB{3}}
&=
[1.05,\; 0,\; 0.2,\; 0.9,\; 0,\; 1.2]^{\MyTRANSPO}
,\\
{\El{\polyCovvecAB{1}}{\IDEl,\IDbEl}}
&=
\begin{cases}
 0.06 & (\IDEl=\IDbEl) \\
-0.01 & (\IDEl\neq\IDbEl)
\end{cases}
,\\
{\El{\polyCovvecAB{2}}{\IDEl,\IDbEl}}
&=
\begin{cases}
0.02 & (\IDEl=\IDbEl) \\
0.01 & (\IDEl\neq\IDbEl)
\end{cases}
,\\
{\El{\polyCovvecAB{3}}{\IDEl,\IDbEl}}
&=
\begin{cases}
0.01 & (\IDEl=\IDbEl) \\
0    & (\IDEl\neq\IDbEl)
\end{cases}
.
\end{align}

Algorithm \ref{alg1} is employed to design a stabilizing controller of the SMP system \eqref{eq:def_CLsys} with the settings described above.
The parameters are set as $\exEMSrate=0.97$, $\smallmarginLMIs=0.1$ $\UBdummyHMMat=10$, and $\MultiSDPterminationVal=1.0 \times 10^{-8}$.
The SDP \eqref{eq:def_multiSDP} was numerically solved using the MATLAB solver, mincx, 
where the initial guess of $\IDite$-th iteration for $\IDite \geq 2$ is set to a solution of $(\IDite-1)$-th iteration.
%using ${\dummyIteHMMat{\IDite-1}}$ as the initial guess of $\IDite$-th iteration for $\IDite \geq 2$.
%\cR{each iteration $\IDite$, initial guess is $XXXXXX$.}

%%%%%%%%%%%%%%%%%%%%%%%%%%%%%%%%%%%%%%%%%%%%%%%%%%%%%%%%%%%%%%%%%%%%%%%%%%%%%%%%%%%%%%%%%%%%%%%%%%%%%%%%%%%%%%%%%%%%%%%%%%%%%%%%%%%%%%%
%%%%%%%%%%%%%%%%%%%%%%%%%%%%%%%%%%%%%%%%%%%%%%%%%%%%%%%%%%%%%%%%%%%%%%%%%%%%%%%%%%%%%%%%%%%%%%%%%%%%%%%%%%%%%%%%%%%%%%%%%%%%%%%%%%%%%%%
%%%%%%%%%%%%%%%%%%%%%%%%%%%%%%%%%%%%%%%%%%%%%%%%%%%%%%%%%%%%%%%%%%%%%%%%%%%%%%%%%%%%%%%%%%%%%%%%%%%%%%%%%%%%%%%%%%%%%%%%%%%%%%%%%%%%%%%
%%%%%%%%%%%%%%%%%%%%%%%%%%%%%%%%%%%%%%%%%%%%%%%%%%%%%%%%%%%%%%%%%%%%%%%%%%%%%%%%%%%%%%%%%%%%%%%%%%%%%%%%%%%%%%%%%%%%%%%%%%%%%%%%%%%%%%%
\subsection{Simulation results}\label{sec_sim_results}

\newcommand{\simFolder}{results_UncMeanCovSysA_20210830}

The SDP \eqref{eq:def_multiSDP} was iteratively solved so that its solution $\optdummyHMMat={\dummyIteHMMat{\IDite}}$ is approximated by a rank-one matrix $\approxdummyHMMat$.
Recall that $\MyRank{\dummyIteHMMat{\IDite}}=1$ holds if the maximum eigenvalue ${\iEig{1}{\dummyIteHMMat{\IDite}}}$ is positive and the absolute sum of the other eigenvalues
${\rankoneErr{\dummyIteHMMat{\IDite}}}
=
\sum_{\IDEl=2}^{\DimX(\DimX+\DimU)}
|{\iEig{\IDEl}{\dummyIteHMMat{\IDite}}}|$ is zero.
These values for each iteration are shown in Fig. \ref{fig:eigenvalues}.
We see that ${\rankoneErr{\dummyIteHMMat{\IDite}}}$ is sufficiently small for $\IDite\geq 2$, indicating that 
the rank-one matrix $\approxdummyHMMat$ is successfully obtained.
Indeed, the corresponding feedback gain $\FBgain$ guarantees the exponential robust MS stability because the stability conditions {\itemMultiPERSCMIs} was satisfied.

\newcommand{\MyFigSIZEeigen}{0.55}
\begin{figure}[!t]	
	%\vspace*{+0.07in}
	\centering
	\begin{minipage}[b]{\MyFigSIZEeigen\linewidth}
		\centering	
		\begin{tikzpicture}[scale=0.9]%
		\begin{axis}[axis y line=none, axis x line=none
		,xtick=\empty, ytick=\empty %grid=major%,	
		,xmin=-10,xmax=10,ymin=-10,ymax=10
		,width=3.1 in,height=1.8 in
		]%, enlarge x limits=false		
		
			%\filldraw[draw=red, fill=gray] (-15,-15) -- (-15,15) -- (15,15) -- (15,-15);		
		
		\node at (0,0.2) { 
		\includegraphics[width=0.7\linewidth]{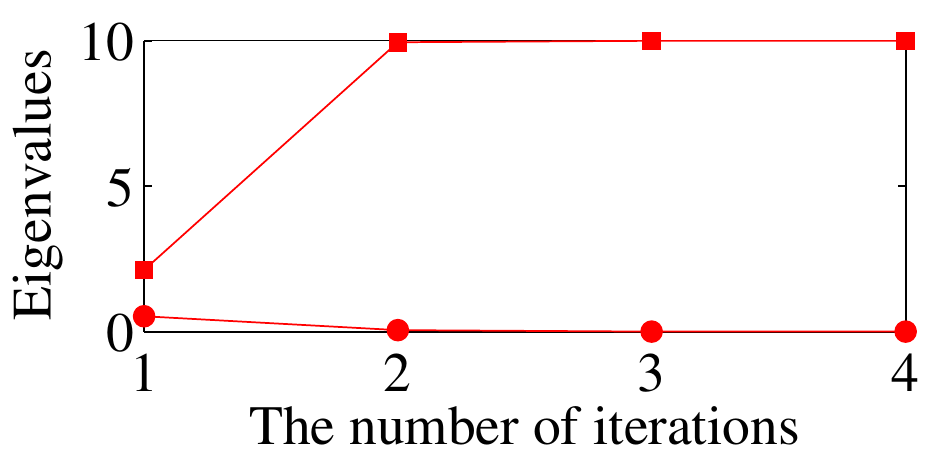}
		};
		\fill[fill=white] (-10,-10) rectangle (-8.2,10);
		\node at (-8.8,1.2) {\rotatebox{90}{
			Value 
		}};
		\fill[fill=white] (-10,-7.4) rectangle (10,-11);
		\node at (1.2,-8.3) {The number $\IDite$ of iterations};		
		\end{axis}			
		\end{tikzpicture}	
	\end{minipage}%	
\caption{Results of solving the iterative SDP.
	The symbols $\blacksquare$ and $\bullet$ denote the maximum eigenvalue ${\iEig{1}{\dummyIteHMMat{\IDite}}}$ and the absolute sum of the other eigenvalues
	${\rankoneErr{\dummyIteHMMat{\IDite}}}
	=
	\sum_{\IDEl=2}^{\DimX(\DimX+\DimU)}
	|{\iEig{\IDEl}{\dummyIteHMMat{\IDite}}}|$, respectively, where ${\iEig{1}{\dummyIteHMMat{\IDite}}} + {\rankoneErr{\dummyIteHMMat{\IDite}}} \leq \UBdummyHMMat = 10$.
}
\label{fig:eigenvalues}
\end{figure}

Next, we simulated the SMP system for various uncertain parameters ${\TIUnc}$ that were randomly generated.
For each $\MyT$, ${\vecTIAB{\MyT}}$ was randomly sampled from ${\CondPDF{\NonArgvecAB{\MyT}}{\TIUnc}}$ that was set to the normal distribution obeying \eqref{eq:mean_UncMeanCovSys} and \eqref{eq:cov_UncMeanCovSys}.
The simulation was performed in both the cases without control ${\Input{\MyT}}=0$ and by employing the designed controller for the initial state ${\State{0}}=[1,1]^{\MyTRANSPO}$.  
These results with the different parameters are plotted in Fig. \ref{fig:TITV_1st_state_forA}.
It can be observed that the systems without control did not converge to the origin in Fig. \ref{fig:TITV_1st_state_forA} \subref{fig:NoControl}.
In contrast, the designed controller stabilized the systems successfully as shown in Fig. \ref{fig:TITV_1st_state_forA} \subref{fig:WithControl}.
These results indicate that the proposed controller design method can be successfully applied to SMP systems.

\newcommand{\MyFigSIZEresults}{0.52}
\begin{figure}[!t]	
	\vspace*{+0.07in}
	\centering
	\begin{minipage}[b]{\MyFigSIZEresults\linewidth}
		\centering	
		\begin{tikzpicture}[scale=0.9]%
		\begin{axis}[axis y line=none, axis x line=none
		,xtick=\empty, ytick=\empty %grid=major%,	
		,xmin=-11,xmax=10,ymin=-10,ymax=10
		,width=3.9 in,height=2.2 in
		]%, enlarge x limits=false		
		
		%	\filldraw[draw=red, fill=gray] (-15,-15) -- (-15,15) -- (15,15) -- (15,-15);		
		
		\node at (0,0) { 
			\includegraphics[width=0.91\linewidth]{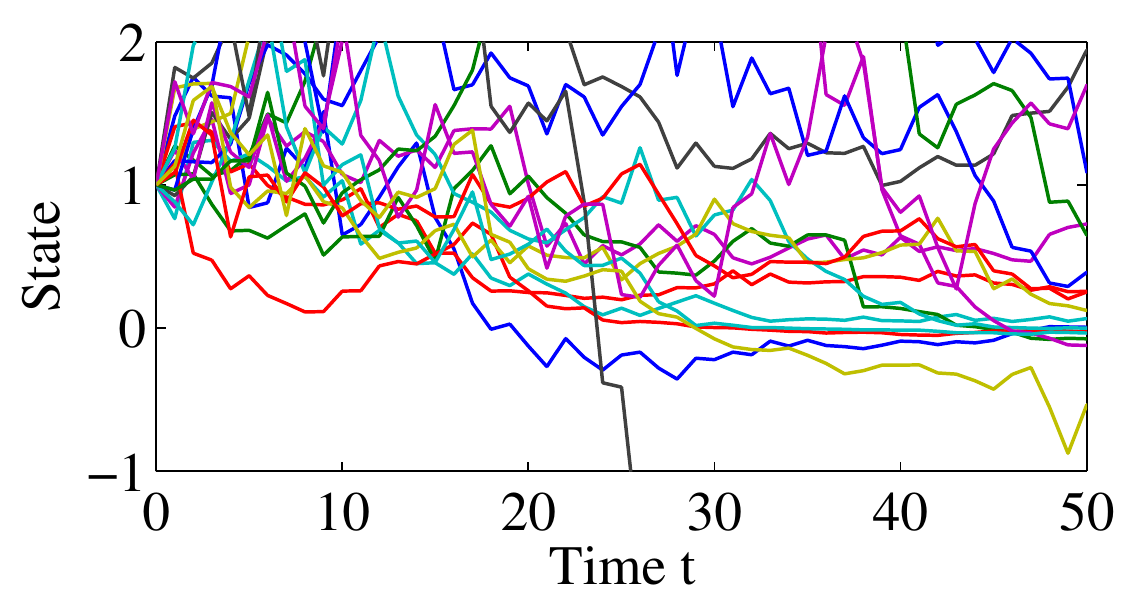} 
		};
		\fill[fill=white] (-10,-6) rectangle (-8.2,6);
		\node at (-9.4,1.0) {\rotatebox{90}{
				\shortstack{First component $\El{\State{\MyT}}{1}$ \\ of the state}
		}};
		\fill[fill=white] (-6,-7.8) rectangle (6,-11);
		\node at (1.2,-9.2) {Time step $\MyT$};		 %- \MarginDist
		\end{axis}			
		\end{tikzpicture}	
		\subcaption{Without control}
		\label{fig:NoControl}	
	\end{minipage}%
	\VisibleColTwo{\\}%
	\begin{minipage}[b]{\MyFigSIZEresults\linewidth}
		\centering	
		\begin{tikzpicture}[scale=0.9]%
		\begin{axis}[axis y line=none, axis x line=none
		,xtick=\empty, ytick=\empty %grid=major%,	
		,xmin=-11,xmax=10,ymin=-10,ymax=10
		,width=3.9 in,height=2.2in
		]%, enlarge x limits=false		
		
		%	\filldraw[draw=red, fill=gray] (-15,-15) -- (-15,15) -- (15,15) -- (15,-15);		
		
		\node at (0,0) { 
			\includegraphics[width=0.91\linewidth]{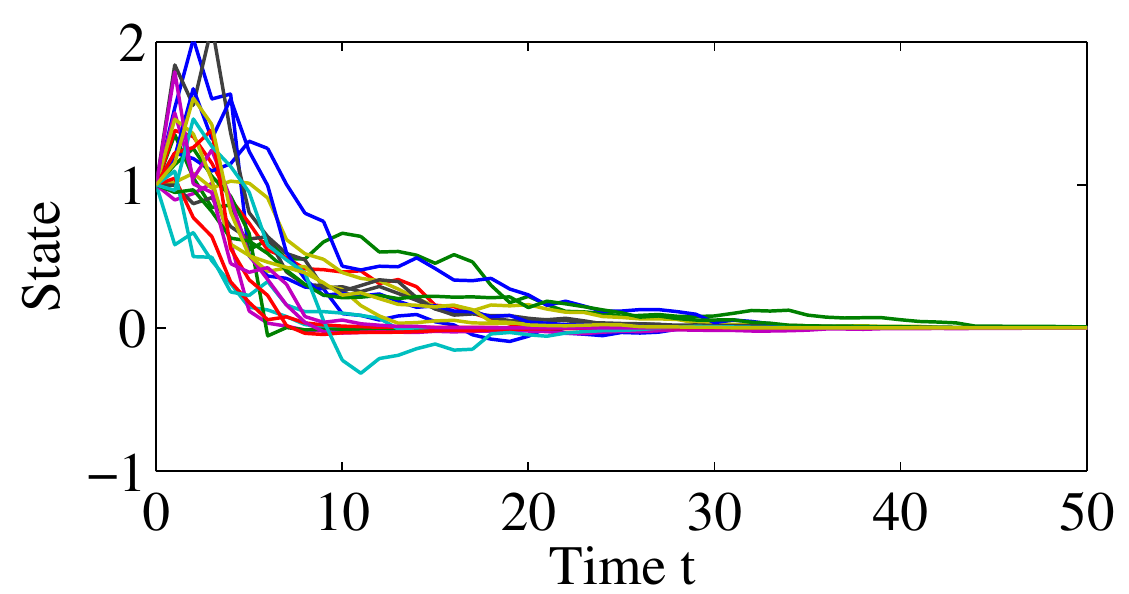} 
		};
		\fill[fill=white] (-10,-6) rectangle (-8.2,6);
		\node at (-9.4,1.0) {\rotatebox{90}{
				\shortstack{First component $\El{\State{\MyT}}{1}$ \\ of the state}
		}};
		\fill[fill=white] (-6,-7.8) rectangle (6,-11);
		\node at (1.2,-9.2) {Time step $\MyT$};		 %- \MarginDist
		\end{axis}			
		\end{tikzpicture}	
		\subcaption{With the designed controller}
		\label{fig:WithControl}			
	\end{minipage}%	
	\caption{Control results without control (a) and by employing the designed feedback controller (b).}
	\label{fig:TITV_1st_state_forA}
\end{figure}
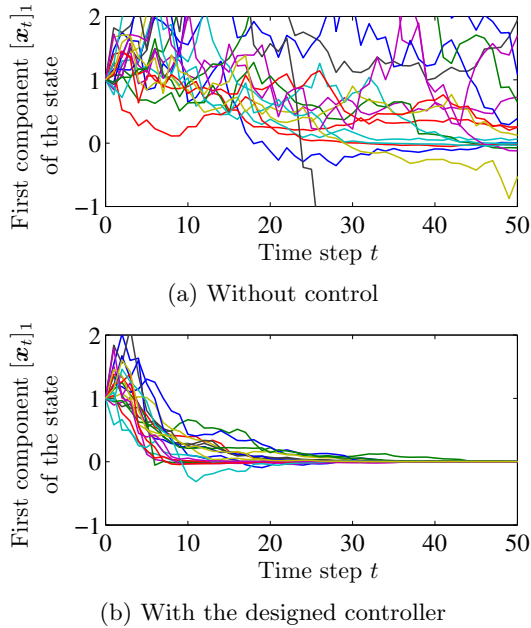

%%%%%%%%%%%%%%%%%%%%%%%%%%%%%%%%%%%%%%%%%%%%%%%%%%%%%%%%%%%%%%%%%%%%%%%%%%%%%%%%%%%%%%%%%%%%%%%%%%%%%%%%%%%%%%%%%%%%%%%%%%%%%%%%%%%%%%%
%%%%%%%%%%%%%%%%%%%%%%%%%%%%%%%%%%%%%%%%%%%%%%%%%%%%%%%%%%%%%%%%%%%%%%%%%%%%%%%%%%%%%%%%%%%%%%%%%%%%%%%%%%%%%%%%%%%%%%%%%%%%%%%%%%%%%%%
%%%%%%%%%%%%%%%%%%%%%%%%%%%%%%%%%%%%%%%%%%%%%%%%%%%%%%%%%%%%%%%%%%%%%%%%%%%%%%%%%%%%%%%%%%%%%%%%%%%%%%%%%%%%%%%%%%%%%%%%%%%%%%%%%%%%%%%
%%%%%%%%%%%%%%%%%%%%%%%%%%%%%%%%%%%%%%%%%%%%%%%%%%%%%%%%%%%%%%%%%%%%%%%%%%%%%%%%%%%%%%%%%%%%%%%%%%%%%%%%%%%%%%%%%%%%%%%%%%%%%%%%%%%%%%%
\section{Conclusion}\label{sec_conclusion}

This paper presented the concept of general uncertain stochastic systems called SMP systems.
The goal of this study was to establish fundamental theory to guarantee stability of the SMP systems. 
A central idea to address this goal is to develop expanded systems with key properties as shown in Section \ref{sec_expandedsys}.
We derived conditions for which the SMP systems are (exponentially) robustly MS stable in Section \ref{sec_stability}.
Based on the derived conditions, a method to design stabilizing feedback gains of linear controllers was proposed.
Section \ref{sec_control_QMIs} showed that feedback gains are solutions to QMIs (into which CMIs are transformed).  
In Section \ref{sec_control_SDPs}, we relaxed the non-convex QMIs as an iterative convex SDP that is easy to solve.

SMP systems can widely represent various uncertain stochastic systems such as i.i.d. systems, random polytope systems, and systems with uncertain mean and covariance, as shown in Theorem \ref{thm:TISMP_examples}.
The effectiveness of the proposed design method was confirmed via a numerical simulation.
In future work, various control problems such as optimal control and output feedback control will be considered for SMP systems.

\appendix

%%%%%%%%%%%%%%%%%%%%%%%%%%%%%%%%%%%%%%%%%%%%%%%%%%%%%%%%%%%%%%%%%%%%%%%%%%%%%%%%%%%%%%%%%%%%%%%%%%%%%%%%%%%%%%%%%%%%%%%%%%%%%%%%%%%%%%%
%%%%%%%%%%%%%%%%%%%%%%%%%%%%%%%%%%%%%%%%%%%%%%%%%%%%%%%%%%%%%%%%%%%%%%%%%%%%%%%%%%%%%%%%%%%%%%%%%%%%%%%%%%%%%%%%%%%%%%%%%%%%%%%%%%%%%%%
%%%%%%%%%%%%%%%%%%%%%%%%%%%%%%%%%%%%%%%%%%%%%%%%%%%%%%%%%%%%%%%%%%%%%%%%%%%%%%%%%%%%%%%%%%%%%%%%%%%%%%%%%%%%%%%%%%%%%%%%%%%%%%%%%%%%%%%
%%%%%%%%%%%%%%%%%%%%%%%%%%%%%%%%%%%%%%%%%%%%%%%%%%%%%%%%%%%%%%%%%%%%%%%%%%%%%%%%%%%%%%%%%%%%%%%%%%%%%%%%%%%%%%%%%%%%%%%%%%%%%%%%%%%%%%%
\section{Proof of Theorem \ref{thm:sys_to_exsys}} \label{pf:sys_to_exsys}

	For brevity of notation,  ${\CLDriftMat{\MyT}}$ denotes the closed-loop matrix in \eqref{eq:def_CLsys}:
	\begin{align}
	{\CLDriftMat{\MyT}} 
	&:= 
	{\DriftMat{\MyT}} - {\InMat{\MyT}} \FBgain 
	. 
	\end{align}
	We show the statement \eqref{eq:exState_is_2ndM_State} using mathematical induction after proving the following relation:
	\begin{align}
	\forall \MyT,\quad
	{\exCLMat{\FBgain}}
	&
	=
	{\CondExpectTV[]{ {\CLDriftMat{\MyT}} \otimes {\CLDriftMat{\MyT}}  }{\MyT}}
	.  \label{eq:pf_exCLMat}
	\end{align}

	First, we show \eqref{eq:pf_exCLMat}.
	For any $\tempAmat \in \mathbb{R}^{\DimX \times \DimX}$ and $\tempBmat \in \mathbb{R}^{\DimX \times \DimU}$,
	let us define $\tempABvec:=\VEC{ [\tempAmat,\tempBmat] }$ and ${\tempSMABvec{\IDEl}{\IDbEl}}\in \mathbb{R}^{\DimX \times \DimX}$ as follows:
	\begin{align}
	\tempABvec \tempABvec^{\MyTRANSPO}
	&
	=\VEC{ [\tempAmat,\tempBmat] }\VEC{ [\tempAmat,\tempBmat] }^{\MyTRANSPO}
	\VisibleColTwo{\nonumber\\&}
	=:
	\begin{bmatrix}
	{\tempSMABvec{1}{1}} & \cdots & {\tempSMABvec{1}{\DimX+\DimU}} \\
	\vdots & \ddots & \vdots \\
	{\tempSMABvec{\DimX+\DimU}{1}} & \cdots & {\tempSMABvec{\DimX+\DimU}{\DimX+\DimU}} \\
	\end{bmatrix}
	.
	\end{align}
	For any $\IDcEl \in \{1,\dots, \DimU\}$ and $\IDbEl \in \{1,\dots, \DimX\}$, we obtain 
	\begin{align}
	{\El{ \tempAmat \otimes \tempBmat }{\vectorWildCard, \DimU(\IDbEl-1) +\IDcEl   }}
	=\begin{bmatrix}
	{\El{\tempAmat}{1,\IDbEl}} {\El{\tempBmat}{\vectorWildCard,\IDcEl}} \\
	\vdots \\
	{\El{\tempAmat}{\DimX,\IDbEl}} {\El{\tempBmat}{\vectorWildCard,\IDcEl}}
	\end{bmatrix}
	&
	=\VEC{   {\El{\tempBmat}{\vectorWildCard,\IDcEl}} {\El{\tempAmat}{\vectorWildCard,\IDbEl}^{\MyTRANSPO}}   }
	\VisibleColTwo{\nonumber\\&}
	=
	\VEC{ {\tempSMABvec{\DimX+\IDcEl}{\IDbEl}}   }
	. \label{eq:relation_M_to_AB}
	\end{align}
	In a manner similar to this relation, the following results are obtained 
	for $\IDEl , \IDbEl \in \{1,\dots, \DimX\}$ and $\IDcEl, \IDdEl \in \{1,\dots, \DimU\}$:
	\begin{align}
	{\El{ \tempAmat \otimes \tempAmat }{\vectorWildCard, \DimX(\IDbEl-1) +\IDEl   }}
	&=
	\VEC{ {\tempSMABvec{\IDEl}{\IDbEl}}   }
	,\label{eq:relation_M_to_AA}
	\\
	{\El{ \tempBmat \otimes \tempAmat }{\vectorWildCard, \DimX(\IDdEl-1) +\IDEl   }}
	&=
	\VEC{ {\tempSMABvec{\IDEl}{\DimX+\IDdEl}}   }
	,\label{eq:relation_M_to_BA}
	\\
	{\El{ \tempBmat \otimes \tempBmat }{\vectorWildCard, \DimU(\IDdEl-1) +\IDcEl   }}
	&=
	\VEC{ {\tempSMABvec{\DimX+\IDcEl}{\DimX+\IDdEl}}   }
	.\label{eq:relation_M_to_BB}
	\end{align}
	These properties \eqref{eq:relation_M_to_AB}--\eqref{eq:relation_M_to_BB} hold even if we replace 
	$\tempAmat$, $\tempBmat$, $\tempABvec$, and ${\tempSMABvec{\IDEl}{\IDbEl}}$ with
	${\DriftMat{\MyT}}$, ${\InMat{\MyT}}$, ${\vecTVAB{\MyT}}$, and 
	$\sum_{\IDpoly=1}^{\NUMpoly}
	{\El{\MapPolyW{\TVUnc{\MyT}}}{\IDpoly}}
	{\blockpolySMvecAB{\IDpoly}{\IDEl       }{\IDbEl}} $ and take the expectations ${\CondExpectTV{  \dots }{\MyT}}$.
	Then, using \eqref{eq:SMP_AA_AB_BB} and \eqref{eq:def_expolyAAMat}, we obtain
	\begin{align}
	&	{\CondExpectTV[\big]{  
			{\El{		 {\DriftMat{\MyT}}  \otimes  {\DriftMat{\MyT}}    }{\vectorWildCard, \DimX(\IDbEl-1) +\IDEl   }}
		}{\MyT}}
	\VisibleColTwo{\nonumber\\&}
	=
	\VEC[\Big]{
	\sum_{\IDpoly=1}^{\NUMpoly}
	{\El{\MapPolyW{\TVUnc{\MyT}}}{\IDpoly}}
	{\blockpolySMvecAB{\IDpoly}{\IDEl       }{\IDbEl}}   }
	\VisibleColTwo{\nonumber\\&}
	=
	\sum_{\IDpoly=1}^{\NUMpoly}
	{\El{\MapPolyW{\TVUnc{\MyT}}}{\IDpoly}}
	{\El{\expolyAAMat{\IDpoly}}{\vectorWildCard, \DimX(\IDbEl-1) +\IDEl   }}	
	.\label{eq:expolyMat_intuition_AA}
	\end{align}
In the same way,
${\CondExpectTV{    {\DriftMat{\MyT}}  \otimes  {\InMat{\MyT}} }{\MyT}}$,
${\CondExpectTV{   {\InMat{\MyT}}   \otimes  {\DriftMat{\MyT}}  }{\MyT}}$,
${\CondExpectTV{   {\InMat{\MyT}}  \otimes  {\InMat{\MyT}}    }{\MyT}}$ are represented using 
${\expolyABMat{\IDpoly}}$, 
${\expolyBAMat{\IDpoly}}$, and 
${\expolyBBMat{\IDpoly}}$,  respectively.
	Note that $(\NotationMat_{1} \otimes \NotationMat_{2})(\NotationMat_{3} \otimes \NotationMat_{4})=
	\NotationMat_{1}\NotationMat_{3} \otimes \NotationMat_{2}\NotationMat_{4}
	$ and $(\NotationMat_{5} + \NotationMat_{6})\otimes \NotationMat_{7}= \NotationMat_{5} \otimes \NotationMat_{7} + \NotationMat_{6} \otimes \NotationMat_{7}$ hold for the appropriate dimensions of these matrices \cite[Section 3.2.9]{Gentle07}.
	Using these properties and \eqref{eq:expolyMat_intuition_AA},  we obtain \eqref{eq:pf_exCLMat} as follows:
	\begin{align}
	&
	{\CondExpectTV{  
			{\CLDriftMat{\MyT}} \otimes {\CLDriftMat{\MyT}}   
		}{\MyT}}
	\nonumber\\&
	=
	{\CondExpectTV{     {\DriftMat{\MyT}}  \otimes  {\DriftMat{\MyT}}      }{\MyT}}
	-{\CondExpectTV{     {\DriftMat{\MyT}}  \otimes ( {\InMat{\MyT}} \FBgain )   }{\MyT}}
	\nonumber\\&\quad
	-{\CondExpectTV{    (  {\InMat{\MyT}} \FBgain )  \otimes  {\DriftMat{\MyT}}     }{\MyT}}
	\VisibleColTwo{\nonumber\\&\quad}
	+{\CondExpectTV{   ( {\InMat{\MyT}} \FBgain )  \otimes ( {\InMat{\MyT}} \FBgain )   }{\MyT}}
	\nonumber\\&
	=
	{\CondExpectTV{  {\DriftMat{\MyT}}  \otimes  {\DriftMat{\MyT}}    }{\MyT}}
	\VisibleColTwo{\nonumber\\&\quad}
	-{\CondExpectTV{  {\DriftMat{\MyT}}  \otimes  {\InMat{\MyT}}  }{\MyT}}  ({\Identity{\DimX}} \otimes \FBgain)
	\nonumber\\&\quad
	-{\CondExpectTV{   {\InMat{\MyT}}   \otimes  {\DriftMat{\MyT}}    }{\MyT}} ( \FBgain \otimes {\Identity{\DimX}} )
	\VisibleColTwo{\nonumber\\&\quad}
	+{\CondExpectTV{ {\InMat{\MyT}}  \otimes  {\InMat{\MyT}}    }{\MyT}} ( \FBgain \otimes \FBgain )
	\nonumber\\&
	=
	\sum_{\IDpoly=1}^{\NUMpoly}
	{\El{\MapPolyW{\TVUnc{\MyT}}}{\IDpoly}}
	\Big(
	{\expolyAAMat{\IDpoly}}
	-{\expolyABMat{\IDpoly}} ({\Identity{\DimX}} \otimes \FBgain)
	\VisibleColTwo{\nonumber\\&\quad\quad\quad\quad\quad\quad\quad}
	-{\expolyBAMat{\IDpoly}} (\FBgain \otimes {\Identity{\DimX}})
	+{\expolyBBMat{\IDpoly}} (\FBgain \otimes \FBgain)
	\Big)
	\nonumber\\&
	=
	\sum_{\IDpoly=1}^{\NUMpoly}
	{\El{\MapPolyW{\TVUnc{\MyT}}}{\IDpoly}}
	{\expolyCLMat{\IDpoly}{\FBgain}}
	\nonumber\\&
	=
	\sum_{\IDpoly=1}^{\NUMpoly}
	{\El{\exTVUnc{\MyT}}{\IDpoly}}
	{\expolyCLMat{\IDpoly}{\FBgain}}
	\nonumber\\&
	=
	{\exCLMat{\FBgain}}
	. 
	\end{align}

	Next, we show the statement \eqref{eq:exState_is_2ndM_State} by using \eqref{eq:pf_exCLMat} and mathematical induction.
	For each $\MycT \in \{0,1,2,\dots\}$, supposing that \eqref{eq:exState_is_2ndM_State} holds for $\MyT=\MycT$, the following relation implies \eqref{eq:exState_is_2ndM_State} for $\MyT=\MycT+1$:
	\begin{align}
	\VisibleColTwo{&}
	{\CondExpectAllTV{  
			\VECH{
				{\State{\MycT+1}}{\State{\MycT+1}^{\MyTRANSPO}} 
			}
		}{\MybT}{\MycT+1}}
	\VisibleColTwo{\nonumber\\}
	&
	=
	{\CondExpectAllTV{   
			\EliMat
			\VEC{
				{\State{\MycT+1}}{\State{\MycT+1}^{\MyTRANSPO} }
			}
		}{\MybT}{\MycT+1}}
	\nonumber\\&
	=
	\EliMat
	{\CondExpectAllTV{   
			\VEC{     {\CLDriftMat{\MycT}}   {\State{\MycT}}{\State{\MycT}^{\MyTRANSPO} } {\CLDriftMat{\MycT}}^{\MyTRANSPO} }
		}{\MybT}{\MycT+1}}
	\nonumber\\&
	=
	\EliMat
	{\CondExpectAllTV{ 
			(  {\CLDriftMat{\MycT}} \otimes {\CLDriftMat{\MycT}}   )
			\VEC{{\State{\MycT}}{\State{\MycT}^{\MyTRANSPO}} }
		}{\MybT}{\MycT+1}}
	\nonumber\\&
	=
	\EliMat
	{\CondExpectAllTV{   
			(  {\CLDriftMat{\MycT}} \otimes {\CLDriftMat{\MycT}}   )
			\DupMat
			\VECH{{\State{\MycT}}{\State{\MycT}^{\MyTRANSPO}} }
		}{\MybT}{\MycT+1}}
	\nonumber\\&
	=
	\Elimi{
		{\CondExpectTV{ 
				(  {\CLDriftMat{\MycT}} \otimes {\CLDriftMat{\MycT}}   )
			}{\MycT}}
	}
	{\CondExpectAllTV{ 
			\VECH{{\State{\MycT}}{\State{\MycT}^{\MyTRANSPO}} }
		}{\MybT}{\MycT}}
	\nonumber\\&
	=
	\Elimi{
		{\CondExpectTV{   
				{\CLDriftMat{\MycT}} \otimes {\CLDriftMat{\MycT}}   
			}{\MycT}}
	}
	{\exState{\MycT}}
	\nonumber\\&
	=
	\Elimi{
		{\DETAILexCLMat{\FBgain}{\MycT}}
	}
	{\exState{\MycT}}
	\nonumber\\&
	=
	{\exState{\MycT+1}}
	, \label{eq:VECH_to_VEC_pf}
	\end{align}
	where the third equality is derived using the relation $\VEC{\NotationMat_{1}  \NotationMat_{2} \NotationMat_{3}}=(\NotationMat_{3}^{\MyTRANSPO} \otimes \NotationMat_{1}) \VEC{\NotationMat_{2}}  $ \cite[Section 3.2.9]{Gentle07} with $\NotationMat_{1}=\NotationMat_{3}^{\MyTRANSPO}={\CLDriftMat{\MycT}}$ and $\NotationMat_{2}={\State{\MycT}}{\State{\MycT}^{\MyTRANSPO}}$.
	The fifth, seventh, and last equalities are derived based on the independence between ${\CLDriftMat{\MycT}}$ and ${\State{\MycT}}$, \eqref{eq:pf_exCLMat}, and \eqref{eq:def_exsys}, respectively.
	Since \eqref{eq:exState_is_2ndM_State} was assumed for $\MyT=0$, 
	\eqref{eq:exState_is_2ndM_State} holds for all $\MyT \in \{0,1,2,\dots\}$ by mathematical induction.
	This completes the proof.

%%%%%%%%%%%%%%%%%%%%%%%%%%%%%%%%%%%%%%%%%%%%%%%%%%%%%%%%%%%%%%%%%%%%%%%%%%%%%%%%%%%%%%%%%%%%%%%%%%%%%%%%%%%%%%%%%%%%%%%%%%%%%%%%%%%%%%%
%%%%%%%%%%%%%%%%%%%%%%%%%%%%%%%%%%%%%%%%%%%%%%%%%%%%%%%%%%%%%%%%%%%%%%%%%%%%%%%%%%%%%%%%%%%%%%%%%%%%%%%%%%%%%%%%%%%%%%%%%%%%%%%%%%%%%%%
%%%%%%%%%%%%%%%%%%%%%%%%%%%%%%%%%%%%%%%%%%%%%%%%%%%%%%%%%%%%%%%%%%%%%%%%%%%%%%%%%%%%%%%%%%%%%%%%%%%%%%%%%%%%%%%%%%%%%%%%%%%%%%%%%%%%%%%
%%%%%%%%%%%%%%%%%%%%%%%%%%%%%%%%%%%%%%%%%%%%%%%%%%%%%%%%%%%%%%%%%%%%%%%%%%%%%%%%%%%%%%%%%%%%%%%%%%%%%%%%%%%%%%%%%%%%%%%%%%%%%%%%%%%%%%%
\section{Proof of Theorem \ref{thm:Equivarence_RMS}} \label{pf:Equivarence_RMS}

We first show the sufficiency.
For any ${\State{0}} \in \mathbb{R}^{\DimX}$ and any ${\TVUnc{0}},{\TVUnc{1}},\dots  \in \DomTVUnc $,
we choose ${\exState{0}}$ and ${\exTVUnc{0}},{\exTVUnc{1}},\dots $ that satisfy 
\eqref{eq:ass_ini_condition_State} and
\eqref{eq:ass_ini_condition_TISto}. 
Theorem \ref{thm:sys_to_exsys} implies that,
for $\IDEl \in \{1,\dots, \DimX\}$, there exists $\IDbEl$ satisfying
\begin{align}
{\CondExpectAllTV{ 
		\El{\State{\MyT}}{\IDEl}^{2}
	}{\MybT}{\MyT}}
=
{\El{\exState{\MyT}}{\IDbEl}}
\leq
\|{\exState{\MyT}}\|
.
\end{align}	
If the expanded system {\eqref{eq:def_exsys}}
is robustly stable, we obtain the following asymptotic convergence for all $\IDEl \in \{1, \dots, \DimX \}$: 
\begin{align}
\lim_{\MyT \to \infty}
{\CondExpectAllTV[\big]{  
		\El{\State{\MyT}}{\IDEl}^{2}
	}{\MybT}{\MyT}}
\leq
\lim_{\MyT \to \infty}
\|{\exState{\MyT}}\|
=0
.
\end{align}			
Therefore, the SMP system {\eqref{eq:def_CLsys}} is robustly MS stable.

Next, we show the necessity.
For any ${\exTVUnc{0}},{\exTVUnc{1}},\dots  \in \ImagePolyW$, there exist ${\TVUnc{0}},{\TVUnc{1}},\dots  \in \DomTVUnc $ satisfying \eqref{eq:ass_ini_condition_TISto} because of the definition \eqref{eq:def_DompolyW} of $\ImagePolyW$.
We consider ${\exState{0}} \in \mathbb{R}^{\DimexX}$ such that \eqref{eq:ass_ini_condition_State} holds for some ${\State{0}}$.
Then, Theorem \ref{thm:sys_to_exsys} gives \eqref{eq:exState_is_2ndM_State} for all $\MyT$.
For any $\IDEl$ and $\IDbEl$, 
applying the Cauchy-Schwarz inequality \cite[Section 5.5]{Gray09} to ${\El{\State{\MyT}}{\IDEl}} {\El{\State{\MyT}}{\IDbEl}}$ yields
\begin{align}
&
|
{\CondExpectAllTV{ 
		{\El{\State{\MyT}}{\IDEl}} {\El{\State{\MyT}}{\IDbEl}}
	}{\MybT}{\MyT}}
|^{2}
\leq	
%(
{\CondExpectAllTV{ 
		{\El{\State{\MyT}}{\IDEl}}^{2}
	}{\MybT}{\MyT}}
{\CondExpectAllTV{ 
		{\El{\State{\MyT}}{\IDbEl}}^{2}
	}{\MybT}{\MyT}}
%)^{1/2}
. \label{eq:CauchySchwarz_State_to_exState}
\end{align}	
If the SMP system {\eqref{eq:def_CLsys}} is robustly MS stable, this property indicates
\begin{align}
&
\lim_{\MyT \to \infty}
%\big|
{\CondExpectAllTV[\big]{ 
		{\El{\State{\MyT}}{\IDEl}} {\El{\State{\MyT}}{\IDbEl}}
	}{\MybT}{\MyT}}
%\big|
=0
.
\end{align}	
Combining this convergence with \eqref{eq:exState_is_2ndM_State} gives	
\begin{align}
&
\forall {\exTVUnc{0}},{\exTVUnc{1}} , {\dots} \in \ImagePolyW
,
\lim_{\MyT \to \infty}
{\exState{\MyT}} =  0
, \label{eq:Convergence_exState_pf} 
\end{align}		
for any ${\exState{0}}$ that satisfies ${\exState{0}}=\VECH{ {\State{0}} {\State{0}^{\MyTRANSPO}}  }$ for some ${\State{0}}$.
Therefore, we finally show that 
\eqref{eq:Convergence_exState_pf}  holds for all ${\exState{0}} \in \mathbb{R}^{\DimexX}$.
Let us define ${\PFStateIni{\IDEl}{\IDbEl}} \in \mathbb{R}^{\DimX}$ such that ${\El{\PFStateIni{\IDEl}{\IDbEl}}{\IDEl}}={\El{\PFStateIni{\IDEl}{\IDbEl}}{\IDbEl}}=1$ and the other components are zero.
Let us define ${\vhxx{\IDEl}{\IDbEl}}:=\VECH{ {\PFStateIni{\IDEl}{\IDbEl}}  {\PFStateIni{\IDEl}{\IDbEl}^{\MyTRANSPO}} }$.
The standard basis on $\mathbb{R}^{\DimexX}$ consists of the unit vectors
${\vhxx{\IDEl}{\IDEl}}$ and $({\vhxx{\IDEl}{\IDbEl}} -{\vhxx{\IDEl}{\IDEl}} - {\vhxx{\IDbEl}{\IDbEl}})$
for  all $\IDEl \in \{1,\dots, \DimX\}$ and $\IDbEl \in \{\IDEl+1,\dots, \DimX\} $.
Then, for any ${\exState{0}} \in  \mathbb{R}^{\DimexX}$, there exist constants ${\weightPFexState{\IDEl}{\IDbEl}} \in \mathbb{R}$ satisfying
\begin{align}
{\exState{0}}
=
\sum_{\IDEl=1}^{\DimX}
\Big(
{\weightPFexState{\IDEl}{\IDEl}}
{\vhxx{\IDEl}{\IDEl}}
+
\sum_{\IDbEl=\IDEl+1}^{\DimX} {\weightPFexState{\IDEl}{\IDbEl}}
({\vhxx{\IDEl}{\IDbEl}} -{\vhxx{\IDEl}{\IDEl}} - {\vhxx{\IDbEl}{\IDbEl}})
\Big)
.\label{eq:def_weightPFexState}
\end{align}
Then, for any ${\exState{0}}  \in  \mathbb{R}^{\DimexX}$, we obtain
\begin{align}
{\exState{\MyT}} 
&= \Big(
\prod_{\MybT=0}^{\MyT-1}\Elimi{\exCLMat{\FBgain}} 
\Big)
{\exState{0}}
\VisibleColTwo{\nonumber\\&}
=
\sum_{\IDEl=1}^{\DimX}
\Big(
{\weightPFexState{\IDEl}{\IDEl}}
{\SpetialexStatevhxx{\MyT}{\IDEl}{\IDEl}}
+
\sum_{\IDbEl=\IDEl+1}^{\DimX} 
{\weightPFexState{\IDEl}{\IDbEl}}
\big(
{\SpetialexStatevhxx{\MyT}{\IDEl}{\IDbEl}}
-
{\SpetialexStatevhxx{\MyT}{\IDEl}{\IDEl}} 
-
{\SpetialexStatevhxx{\MyT}{\IDbEl}{\IDbEl}}
\big)
\Big)
,\label{eq:any_exState_linearCombi}
\end{align}	
where
${\SpetialexStatevhxx{\MyT}{\IDEl}{\IDbEl}}
:=(\prod_{\MybT=0}^{\MyT-1}\Elimi{\exCLMat{\FBgain}} )
{\SpetialexStatevhxx{0}{\IDEl}{\IDbEl}}
$
denotes ${\exState{\MyT}} $ starting from ${\exState{0}}= {\vhxx{\IDEl}{\IDbEl}}$. 
This implies \eqref{eq:Convergence_exState_pf}  for all ${\exState{0}}  \in  \mathbb{R}^{\DimexX}$ because  \eqref{eq:Convergence_exState_pf} with ${ {\exState{0}}= {\vhxx{\IDEl}{\IDbEl}}=\VECH{ {\PFStateIni{\IDEl}{\IDbEl}}  {\PFStateIni{\IDEl}{\IDbEl}^{\MyTRANSPO}} }  }$ holds for any $\IDEl$ and $\IDbEl$.
This completes the proof.

%%%%%%%%%%%%%%%%%%%%%%%%%%%%%%%%%%%%%%%%%%%%%%%%%%%%%%%%%%%%%%%%%%%%%%%%%%%%%%%%%%%%%%%%%%%%%%%%%%%%%%%%%%%%%%%%%%%%%%%%%%%%%%%%%%%%%%%
%%%%%%%%%%%%%%%%%%%%%%%%%%%%%%%%%%%%%%%%%%%%%%%%%%%%%%%%%%%%%%%%%%%%%%%%%%%%%%%%%%%%%%%%%%%%%%%%%%%%%%%%%%%%%%%%%%%%%%%%%%%%%%%%%%%%%%%
%%%%%%%%%%%%%%%%%%%%%%%%%%%%%%%%%%%%%%%%%%%%%%%%%%%%%%%%%%%%%%%%%%%%%%%%%%%%%%%%%%%%%%%%%%%%%%%%%%%%%%%%%%%%%%%%%%%%%%%%%%%%%%%%%%%%%%%
%%%%%%%%%%%%%%%%%%%%%%%%%%%%%%%%%%%%%%%%%%%%%%%%%%%%%%%%%%%%%%%%%%%%%%%%%%%%%%%%%%%%%%%%%%%%%%%%%%%%%%%%%%%%%%%%%%%%%%%%%%%%%%%%%%%%%%%
\section{Proof of Theorem \ref{thm:Equivarence_ERMS}} \label{pf:Equivarence_ERMS}

To employ \eqref{eq:exState_is_2ndM_State} for all $\MyT \in \{0,1,2,\dots\}$,
we consider the case that 
\eqref{eq:ass_ini_condition_State} and
\eqref{eq:ass_ini_condition_TISto} hold, which are derived later.
By using \eqref{eq:exState_is_2ndM_State} and the Cauchy-Schwarz inequality \cite[Lecture 1]{Bhatia09}:
\begin{align}
\Big(
\sum_{\IDEl=1}^{\DimX} \El{ \coefASchwarz }{\IDEl} \El{ \coefBSchwarz }{\IDEl}
\Big)^{2}
\leq 
\Big(
\sum_{\IDEl=1}^{\DimX} \El{ \coefASchwarz }{\IDEl}^{2}
\Big)
\Big(
\sum_{\IDEl=1}^{\DimX} \El{ \coefBSchwarz }{\IDEl}^{2}
\Big)
,
\end{align}
with the settings  
$\El{ \coefASchwarz }{\IDEl} = 1 $
and $\El{ \coefBSchwarz }{\IDEl} =
{\CondExpectAllTV{ 
		\El{\State{\MyT}}{\IDEl}^{2}
	}{\MybT}{\MyT}}
$,
we obtain	
\begin{align}
\DimX \|{\exState{\MyT}}\|^{2}
%&
=
%\Bigg(
\DimX
\sum_{\IDbEl=1}^{\DimexX}
\El{\exState{\MyT}}{\IDbEl}^{2}
%\Bigg)^{1/2}
&
\geq
\Big(
\sum_{\IDEl=1}^{\DimX} 1
\Big)
\sum_{\IDEl=1}^{\DimX}
{\CondExpectAllTV[\big]{ 
		\El{\State{\MyT}}{\IDEl}^{2}
	}{\MybT}{\MyT}}^{2}
%\Bigg)^{1/2}
\VisibleColTwo{\nonumber \\&}
\geq
\Big(
\sum_{\IDEl=1}^{\DimX}
{\CondExpectAllTV[\big]{  
		\El{\State{\MyT}}{\IDEl}^{2}
	}{\MybT}{\MyT}}
\Big)^{2}
\VisibleColTwo{\nonumber \\&}
=
{\CondExpectAllTV[\big]{   
		\|{\State{\MyT}}\|^{2}
	}{\MybT}{\MyT}}^{2}
. \label{eq:ineq_exState_to_State1}
\end{align}	
Meanwhile,
using
\eqref{eq:CauchySchwarz_State_to_exState} yields
\begin{align}	
\|{\exState{\MyT}}\|^{2}
%&
=
%\Bigg(
\sum_{\IDEl=1}^{\DimexX}
\El{\exState{\MyT}}{\IDEl}^{2}
%\Bigg)^{2}
%\nonumber \\ 
&
\leq
%\Bigg(
\sum_{\IDEl=1}^{\DimX}
\sum_{\IDbEl=1}^{\DimX}
{\CondExpectAllTV[\big]{  
		\El{\State{\MyT}}{\IDEl}
		\El{\State{\MyT}}{\IDbEl}
	}{\MybT}{\MyT}}^{2}
\nonumber \\ &
\leq
%\Bigg(
\sum_{\IDEl=1}^{\DimX}
\sum_{\IDbEl=1}^{\DimX}
{\CondExpectAllTV[\big]{  
		{\El{\State{\MyT}}{\IDEl}}^{2}
	}{\MybT}{\MyT}}
{\CondExpectAllTV[\big]{  
		{\El{\State{\MyT}}{\IDbEl}}^{2}
	}{\MybT}{\MyT}}
\nonumber \\ &
=
%\Bigg(
\sum_{\IDEl=1}^{\DimX}
{\CondExpectAllTV[\big]{  
		{\El{\State{\MyT}}{\IDEl}}^{2}
	}{\MybT}{\MyT}}
%\Bigg)
%\Bigg(
\sum_{\IDbEl=1}^{\DimX}
{\CondExpectAllTV[\big]{  
		{\El{\State{\MyT}}{\IDbEl}}^{2}
	}{\MybT}{\MyT}}
%\Bigg)
\nonumber \\&
=
{\CondExpectAllTV[\big]{  
		\|    {\State{\MyT}}    \|^{2}
	}{\MybT}{\MyT}}^{2}
.	\label{eq:ineq_exState_to_State2}
\end{align}

Now, we show the sufficiency. 
The conditions \eqref{eq:ineq_exState_to_State1} and \eqref{eq:ineq_exState_to_State2} can be used  for any ${\State{0}} \in \mathbb{R}^{\DimX}$ and any ${\TVUnc{0}},{\TVUnc{1}},\dots  \in \DomTVUnc $ because there exist ${\exState{0}}$ and ${\exTVUnc{0}},{\exTVUnc{1}},\dots$ satisfying \eqref{eq:ass_ini_condition_State} and
\eqref{eq:ass_ini_condition_TISto}.
If the expanded system \eqref{eq:def_exsys} is exponentially robustly stable, substituting \eqref{eq:ineq_exState_to_State2} with $\MyT=0$ and \eqref{eq:ineq_exState_to_State1} into \eqref{eq:def_ERstable} yields
\begin{align}
\DimX^{-1/2}
{\CondExpectAllTV[\big]{   		\|{\State{\MyT}}\|^{2}	}{\MybT}{\MyT}}
\leq
\|    {\exState{\MyT}}    \|
\leq 
\exEMScoef  \|{\exState{0}}\| \exEMSrate^{\MyT}
\leq
\exEMScoef  \|\State{0}\|^{2} \exEMSrate^{\MyT}
.
\end{align}
Taking the root of this inequality gives
\begin{align}
\sqrt{
	{\CondExpectAllTV[]{
			\|{\State{\MyT}}\|^{2}
		}{\MybT}{\MyT}}
}
&\leq 
\DimX^{1/4} {\exEMScoef}^{1/2} 
\|{\State{0}}\| {\exEMSrate}^{\MyT/2} 
.
\end{align}	
Therefore, substituting \eqref{eq:exEMSrate_to_EMSrate} and $\EMScoef\geq \DimX^{1/4} {\exEMScoef}^{1/2} $ yields the exponential robust MS stability in \eqref{eq:def_ERMSstable}.

Next, we show the necessity.
For any ${\exTVUnc{0}},{\exTVUnc{1}},\dots  \in \ImagePolyW$, there exist ${\TVUnc{0}},{\TVUnc{1}},\dots  \in \DomTVUnc $ satisfying \eqref{eq:ass_ini_condition_TISto} because of the definition \eqref{eq:def_DompolyW} of $\ImagePolyW$.
We consider ${\exState{0}} \in \mathbb{R}^{\DimexX}$ such that \eqref{eq:ass_ini_condition_State} holds for some ${\State{0}}$.
Then, \eqref{eq:ineq_exState_to_State1} and \eqref{eq:ineq_exState_to_State2} can be used.
If the SMP system {\eqref{eq:def_CLsys}} is exponentially robustly MS stable,
substituting \eqref{eq:ineq_exState_to_State1} with $\MyT=0$ and \eqref{eq:ineq_exState_to_State2} into the squared version of \eqref{eq:def_ERMSstable} yields
\begin{align}
\|{\exState{\MyT}}\|
\leq
{\CondExpectAllTV[\big]{  \|    {\State{\MyT}}    \|^{2}   }{\MybT}{\MyT}}
\leq 
\EMScoef^{2}  \|{\State{0}}\|^{2} \EMSrate^{2\MyT}
\leq 
\EMScoef^{2} \DimX^{1/2} \|{\exState{0}}\| \EMSrate^{2\MyT}
.\label{eq:def_ERMSstable_pf1}
\end{align}

Meanwhile, we recall \eqref{eq:def_weightPFexState}.
There exist ${\weightPFexState{\IDEl}{\IDbEl}}$ satisfying \eqref{eq:def_weightPFexState} and $| {\weightPFexState{\IDEl}{\IDbEl}} | \leq \| {\exState{0}} \|$ because ${\weightPFexState{\IDEl}{\IDbEl}}$ are the coefficients of the unit vectors.
In addition, $\| {\vhxx{\IDEl}{\IDbEl}}  \| \leq \sqrt{3}$ holds because of its definition.
For any ${\exState{0}}\in \mathbb{R}^{\DimexX}$ and any ${\exTVUnc{0}},{\exTVUnc{1}},\dots  \in \ImagePolyW$,
combining \eqref{eq:any_exState_linearCombi} with these bounds and \eqref{eq:def_ERMSstable_pf1} yields
\begin{align}
\|{\exState{\MyT}}  \|
&\leq 
\sum_{\IDEl=1}^{\DimX}
\Big(
|{\weightPFexState{\IDEl}{\IDEl}}|
\|{\SpetialexStatevhxx{\MyT}{\IDEl}{\IDEl}}\|
+
\sum_{\IDbEl=\IDEl+1}^{\DimX} 
|{\weightPFexState{\IDEl}{\IDbEl}}|
\big(
\|{\SpetialexStatevhxx{\MyT}{\IDEl}{\IDbEl}}\|
\VisibleColTwo{\nonumber\\&\qquad\qquad}
+
\|{\SpetialexStatevhxx{\MyT}{\IDEl}{\IDEl}} \|
+
\|{\SpetialexStatevhxx{\MyT}{\IDbEl}{\IDbEl}}\|
\big)
\Big)
\nonumber\\
&\leq 
\EMScoef^{2}  \DimX^{1/2}\EMSrate^{2\MyT}
\sum_{\IDEl=1}^{\DimX}
\Big(
|{\weightPFexState{\IDEl}{\IDEl}}|
 \| {\vhxx{\IDEl}{\IDEl}} \| 
+
\sum_{\IDbEl=\IDEl+1}^{\DimX} 
|{\weightPFexState{\IDEl}{\IDbEl}}|
\big(
\| {\vhxx{\IDEl}{\IDbEl}} \| 
\VisibleColTwo{\nonumber\\&\qquad\qquad\qquad\qquad}
+
\| {\vhxx{\IDEl}{\IDEl}} \|
+
\| {\vhxx{\IDbEl}{\IDbEl}} \|
\big)
\Big)
\nonumber\\
&\leq 
\EMScoef^{2}  \DimX^{1/2}\EMSrate^{2\MyT}
\sum_{\IDEl=1}^{\DimX}
\Big(
\sqrt{3}\| {\exState{0}} \|
+
\sum_{\IDbEl=\IDEl+1}^{\DimX} 
3\sqrt{3}\| {\exState{0}} \|
\Big)
\nonumber\\
&\leq 
\sqrt{3} 
\Big(
\DimX + 3\frac{\DimX(\DimX-1)}{2}
\Big)
\EMScoef^{2}  \DimX^{1/2} \| {\exState{0}} \| \EMSrate^{2\MyT}
.
\end{align}	
Therefore, substituting \eqref{eq:exEMSrate_to_EMSrate} and 
$\exEMScoef \geq (\sqrt{3}/2)(3 \DimX  -1)\DimX^{3/2}\EMScoef^{2} $ 
yields \eqref{eq:def_ERstable}.
This completes the proof.

%%%%%%%%%%%%%%%%%%%%%%%%%%%%%%%%%%%%%%%%%%%%%%%%%%%%%%%%%%%%%%%%%%%%%%%%%%%%%%%%%%%%%%%%%%%%%%%%%%%%%%%%%%%%%%%%%%%%%%%%%%%%%%%%%%%%%%%
%%%%%%%%%%%%%%%%%%%%%%%%%%%%%%%%%%%%%%%%%%%%%%%%%%%%%%%%%%%%%%%%%%%%%%%%%%%%%%%%%%%%%%%%%%%%%%%%%%%%%%%%%%%%%%%%%%%%%%%%%%%%%%%%%%%%%%%
%%%%%%%%%%%%%%%%%%%%%%%%%%%%%%%%%%%%%%%%%%%%%%%%%%%%%%%%%%%%%%%%%%%%%%%%%%%%%%%%%%%%%%%%%%%%%%%%%%%%%%%%%%%%%%%%%%%%%%%%%%%%%%%%%%%%%%%
%%%%%%%%%%%%%%%%%%%%%%%%%%%%%%%%%%%%%%%%%%%%%%%%%%%%%%%%%%%%%%%%%%%%%%%%%%%%%%%%%%%%%%%%%%%%%%%%%%%%%%%%%%%%%%%%%%%%%%%%%%%%%%%%%%%%%%%
\section{Proof of Theorem \ref{thm:MSstab_TISMP}} \label{pf:MSstab_TISMP}

We prove the statement based on the result in \cite[Theorems 1 and 2]{Oliveira99} and in a manner similar to \cite[Theorems 1 and 2]{HosoeTAC18}.
Using ${\expolyVMat{\IDpoly}} \succ 0$,  $\exAddMat$, and $\FBgain$ satisfying \eqref{eq:expanded_ERS}, we obtain the following result \cite[Theorems 1 and 2]{Oliveira99}:
\begin{align}
\exEMSrate^{2}
\LYAPexpolyVMat({\exTVUnc{\MyT}})
&\succeq
\Big(
\sum_{\IDpoly=1}^{\NUMpoly}
{\El{\exTVUnc{\MyT}}{\IDpoly}}
 \Elimi{ {\expolyCLMat{\IDpoly}{\FBgain}}  }^{\MyTRANSPO} 
\Big)
\VisibleColTwo{\nonumber\\&\qquad \times}
\LYAPexpolyVMat({\exTVUnc{\MyT}})
\Big(
\sum_{\IDpoly=1}^{\NUMpoly}
{\El{\exTVUnc{\MyT}}{\IDpoly}}
 \Elimi{ {\expolyCLMat{\IDpoly}{\FBgain}}  } 
\Big)
, \label{eq:Lyap_ineq1pre}
\end{align}	
where $\LYAPexpolyVMat({\exTVUnc{\MyT}}):=\sum_{\IDpoly=1}^{\NUMpoly}
{\El{\exTVUnc{\MyT}}{\IDpoly}}
{\expolyVMat{\IDpoly}} \succ 0$.
Because 
$\sum_{\IDpoly=1}^{\NUMpoly}
{\El{\exTVUnc{\MyT}}{\IDpoly}}
\Elimi{ {\expolyCLMat{\IDpoly}{\FBgain}}  } 
=
\Elimi{
\sum_{\IDpoly=1}^{\NUMpoly}
{\El{\exTVUnc{\MyT}}{\IDpoly}}
{\expolyCLMat{\IDpoly}{\FBgain}}  } 
=
\Elimi{\exCLMat{\FBgain}}
$
holds, we obtain
\begin{align}
\exEMSrate^{2}
\LYAPexpolyVMat({\exTVUnc{\MyT}})
\succeq
\Elimi{\exCLMat{\FBgain}}^{\MyTRANSPO} 
\LYAPexpolyVMat({\exTVUnc{\MyT}})
\Elimi{\exCLMat{\FBgain}}
. \label{eq:Lyap_ineq1}
\end{align}
Since the expanded system is supposed to be TI, $\LYAPexpolyVMat({\exTIUnc}) \succ 0$ is constant.
Iterating \eqref{eq:Lyap_ineq1} with multiplying by ${\exState{0}}$ on the left and right yields
\begin{align}
\exEMSrate^{2\MyT}
{\exState{0}^{\MyTRANSPO}} \LYAPexpolyVMat({\exTIUnc}) {\exState{0}}
\geq  
{\exState{\MyT}^{\MyTRANSPO}} \LYAPexpolyVMat({\exTIUnc}) {\exState{\MyT}}
. \label{eq:Lyap_ineq2}
\end{align}	
This implies the exponential stability for all $\exTIUnc \in \DomPolytope$ \cite[Theorems 1 and 2]{HosoeTAC18}.
Because $\ImagePolyW \subseteq \DomPolytope$ holds, the expanded system \eqref{eq:def_exsys} is exponentially robustly stable. 
This completes the proof.

%%%%%%%%%%%%%%%%%%%%%%%%%%%%%%%%%%%%%%%%%%%%%%%%%%%%%%%%%%%%%%%%%%%%%%%%%%%%%%%%%%%%%%%%%%%%%%%%%%%%%%%%%%%%%%%%%%%%%%%%%%%%%%%%%%%%%%%
%%%%%%%%%%%%%%%%%%%%%%%%%%%%%%%%%%%%%%%%%%%%%%%%%%%%%%%%%%%%%%%%%%%%%%%%%%%%%%%%%%%%%%%%%%%%%%%%%%%%%%%%%%%%%%%%%%%%%%%%%%%%%%%%%%%%%%%
%%%%%%%%%%%%%%%%%%%%%%%%%%%%%%%%%%%%%%%%%%%%%%%%%%%%%%%%%%%%%%%%%%%%%%%%%%%%%%%%%%%%%%%%%%%%%%%%%%%%%%%%%%%%%%%%%%%%%%%%%%%%%%%%%%%%%%%
%%%%%%%%%%%%%%%%%%%%%%%%%%%%%%%%%%%%%%%%%%%%%%%%%%%%%%%%%%%%%%%%%%%%%%%%%%%%%%%%%%%%%%%%%%%%%%%%%%%%%%%%%%%%%%%%%%%%%%%%%%%%%%%%%%%%%%%
\section{Proof of Theorem \ref{thm:RS_TIexpanded}} \label{pf:RS_TIexpanded}

The proof is similar to that of Theorem \ref{thm:MSstab_TISMP}.
Because of the finiteness of $\IDpoly\in \{1,\dots, \NUMpoly \}$, 
for any ${\expolyVMat{\IDpoly}} \succ 0$,  $\exAddMat$, and $\FBgain$ that satisfy \eqref{eq:expanded_RS},
there exists $\CMIstrictCoef>0$ such that
\begin{align}
\forall \IDpoly %\in \{1,\dots, \NUMpoly \}
,\;\;
{\CMIterms{\IDpoly}{\expolyVMat{\IDpoly}}{\exAddMat}{\FBgain}{1}}
-  
\begin{bmatrix}
\CMIstrictCoef \Identity{\DimexX} & 0 \\ 0 & 0
\end{bmatrix}
 \succeq 0
.\label{eq:expanded_RS_another}
\end{align}	
The following result is obtained in a manner similar to the proof of Theorem \ref{thm:MSstab_TISMP}:
\begin{align}
\LYAPexpolyVMat({\exTVUnc{\MyT}}) 
\succeq
\Elimi{\exCLMat{\FBgain}}^{\MyTRANSPO} 
\LYAPexpolyVMat({\exTVUnc{\MyT}})
\Elimi{\exCLMat{\FBgain}}
+ \CMIstrictCoef \Identity{\DimexX}
. \label{eq:Lyap_ineq1_another}
\end{align}
This condition with  ${\exTVUnc{\MyT}} = \exTIUnc$ indicates that, for any $\exTIUnc \in \ImagePolyW \subseteq \DomPolytope$, the absolute values of all the eigenvalues of $\Elimi{\exTICLMat{\FBgain}}$ is less than one \cite[Lemma 1]{Oliveira99}.
This implies the robust stability of the expanded system \eqref{eq:def_exsys} and thus completes the proof.

%%%%%%%%%%%%%%%%%%%%%%%%%%%%%%%%%%%%%%%%%%%%%%%%%%%%%%%%%%%%%%%%%%%%%%%%%%%%%%%%%%%%%%%%%%%%%%%%%%%%%%%%%%%%%%%%%%%%%%%%%%%%%%%%%%%%%%%
%%%%%%%%%%%%%%%%%%%%%%%%%%%%%%%%%%%%%%%%%%%%%%%%%%%%%%%%%%%%%%%%%%%%%%%%%%%%%%%%%%%%%%%%%%%%%%%%%%%%%%%%%%%%%%%%%%%%%%%%%%%%%%%%%%%%%%%
%%%%%%%%%%%%%%%%%%%%%%%%%%%%%%%%%%%%%%%%%%%%%%%%%%%%%%%%%%%%%%%%%%%%%%%%%%%%%%%%%%%%%%%%%%%%%%%%%%%%%%%%%%%%%%%%%%%%%%%%%%%%%%%%%%%%%%%
%%%%%%%%%%%%%%%%%%%%%%%%%%%%%%%%%%%%%%%%%%%%%%%%%%%%%%%%%%%%%%%%%%%%%%%%%%%%%%%%%%%%%%%%%%%%%%%%%%%%%%%%%%%%%%%%%%%%%%%%%%%%%%%%%%%%%%%
\section{Proof of Theorem \ref{thm:ERS_RS_TVexpanded}} \label{pf:ERS_RS_TVexpanded}

Considering quadratic stability, the proof is similar to those of Theorems \ref{thm:MSstab_TISMP} and \ref{thm:RS_TIexpanded}. 
For any sequence ${\exTVUnc{0}},{\exTVUnc{1}},\dots  \in \ImagePolyW$,
we obtain \eqref{eq:Lyap_ineq1}, \eqref{eq:Lyap_ineq2}, and \eqref{eq:Lyap_ineq1_another} with $\LYAPexpolyVMat({\exTVUnc{\MyT}})=\UNIexpolyVMat$.
Therefore, the exponential robust stability of the TV expanded system holds.
Next, let us denote $\Elimi{{\DETAILexCLMat{\FBgain}{\MyT}}}$ by ${\BRIEFexCLMat{\MyT}}$ for brevity of notation.
Iterating \eqref{eq:Lyap_ineq1_another} in a manner similar to \cite[Lemma 2.2]{Koning82} yields 
\begin{align}
\UNIexpolyVMat
&\succeq
{\BRIEFexCLMat{0}}^{\MyTRANSPO} 
\UNIexpolyVMat
{\BRIEFexCLMat{0}}
+ \CMIstrictCoef \Identity{\DimexX}
\nonumber\\&
\succeq
{\BRIEFexCLMat{0}}^{\MyTRANSPO} 
{\BRIEFexCLMat{1}}^{\MyTRANSPO} 
\UNIexpolyVMat
{\BRIEFexCLMat{1}}
{\BRIEFexCLMat{0}}
+ \CMIstrictCoef 
(
\Identity{\DimexX}
+
{\BRIEFexCLMat{0}}^{\MyTRANSPO} 
{\BRIEFexCLMat{0}}
)
\nonumber\\&
\succeq
{\BRIEFexCLMat{0:\MyT+1}}^{\MyTRANSPO} 
\UNIexpolyVMat
{\BRIEFexCLMat{0:\MyT+1}}
+
\CMIstrictCoef
\Big(
\Identity{\DimexX}
+
\sum_{\MybT=0}^{\MyT}
{\BRIEFexCLMat{0:\MybT}}^{\MyTRANSPO} 
{\BRIEFexCLMat{0:\MybT}}
\Big)
, \label{eq:Lyap_ineq1_another_ite}
\end{align}
where ${\BRIEFexCLMat{0:\MyT}}:={\BRIEFexCLMat{\MyT}} {\BRIEFexCLMat{\MyT-1}} \dots {\BRIEFexCLMat{0}}$.
The positive definiteness of $\UNIexpolyVMat$ in \eqref{eq:Lyap_ineq1_another_ite} leads to the boundedness of $\sum_{\MybT=0}^{\MyT}
{\BRIEFexCLMat{0:\MybT}}^{\MyTRANSPO} 
{\BRIEFexCLMat{0:\MybT}} \succeq 0$ that implies 
$\lim_{\MyT \to \infty}
{\BRIEFexCLMat{0:\MyT}}^{\MyTRANSPO} 
{\BRIEFexCLMat{0:\MyT}} =0$.
Because of ${\exState{\MyT+1}}={\BRIEFexCLMat{0:\MyT}}{\exState{0}}$, the robust stability of the expanded system holds.
This  completes the proof.

%%%%%%%%%%%%%%%%%%%%%%%%%%%%%%%%%%%%%%%%%%%%%%%%%%%%%%%%%%%%%%%%%%%%%%%%%%%%%%%%%%%%%%%%%%%%%%%%%%%%%%%%%%%%%%%%%%%%%%%%%%%%%%%%%%%%%%%
%%%%%%%%%%%%%%%%%%%%%%%%%%%%%%%%%%%%%%%%%%%%%%%%%%%%%%%%%%%%%%%%%%%%%%%%%%%%%%%%%%%%%%%%%%%%%%%%%%%%%%%%%%%%%%%%%%%%%%%%%%%%%%%%%%%%%%%
%%%%%%%%%%%%%%%%%%%%%%%%%%%%%%%%%%%%%%%%%%%%%%%%%%%%%%%%%%%%%%%%%%%%%%%%%%%%%%%%%%%%%%%%%%%%%%%%%%%%%%%%%%%%%%%%%%%%%%%%%%%%%%%%%%%%%%%
%%%%%%%%%%%%%%%%%%%%%%%%%%%%%%%%%%%%%%%%%%%%%%%%%%%%%%%%%%%%%%%%%%%%%%%%%%%%%%%%%%%%%%%%%%%%%%%%%%%%%%%%%%%%%%%%%%%%%%%%%%%%%%%%%%%%%%%
\section{Proof of Theorem \ref{thm:stability_QMI}} \label{pf:stability_QMI}

We first show that {\itemMultiPERSQMIs} implies {\itemMultiPERSCMIs}, that is, 
satisfying the QMIs \eqref{eq:expanded_ERS_QMI} implies that the CMIs \eqref{eq:expanded_ERS} hold.
Let us suppose that the QMIs \eqref{eq:expanded_ERS_QMI} are satisfied.

First, we prove the following relation \cite{Oliveira99} for any $\NotationMat \in \mathbb{R}^{\DimANotation \times \DimANotation}$ and any $\NotationSymMat \succ 0 \in {\SetSymMat{\DimANotation}}$:
\begin{align}
\NotationMat + \NotationMat^{\MyTRANSPO}  - \NotationSymMat \succeq 0
\Rightarrow
\exists \NotationMat^{-1}
.\label{eq:regularity_property}
\end{align}	
The condition  $\NotationMat + \NotationMat^{\MyTRANSPO}  - \NotationSymMat \succeq 0$ with any $\NotationVec \neq 0 \in \mathbb{R}^{\DimANotation}$ implies
\begin{align}
\NotationVec^{\MyTRANSPO}  ( \NotationMat  + \NotationMat^{\MyTRANSPO} ) \NotationVec
=
2 \NotationVec^{\MyTRANSPO}   \NotationMat  \NotationVec
\geq 
\NotationVec^{\MyTRANSPO}  \NotationSymMat \NotationVec > 0
.
\end{align}
Therefore, any $\NotationVec \neq 0$ satisfies $\NotationMat  \NotationVec \neq 0$, implying that  $\NotationMat$ is nonsingular, that is, \eqref{eq:regularity_property} holds.

Next, any $\NotationMat \in \mathbb{R}^{\DimANotation \times \DimANotation}$ satisfies the following relations \cite[Lemma 4.4]{Magnus80}:
\begin{align}	
\DupMat \EliMat ( \NotationMat \otimes \NotationMat ) \DupMat 
&= ( \NotationMat \otimes \NotationMat ) \DupMat 
,\label{eq:MagnusProperties_ExcludeDL}
\\
\det(  \EliMat ( \NotationMat \otimes \NotationMat ) \DupMat  ) &= \det(  \NotationMat )^{\DimANotation+1}
,\label{eq:MagnusProperties_Det}
\\
\exists \NotationMat^{-1}
\Rightarrow
( \EliMat ( \NotationMat \otimes \NotationMat ) \DupMat )^{-1}
&=
\EliMat ( \NotationMat^{-1} \otimes \NotationMat^{-1} ) \DupMat
.\label{eq:MagnusProperties_Inverse}
\end{align}	
Because \eqref{eq:expanded_ERS_QMI}, \eqref{eq:regularity_property}, and the positive definiteness of ${\AIMexpolyVMat{\IDpoly}}$ hold,
$\Elimi{ \AddInvMat \otimes \AddInvMat }$ is nonsingular.
Substituting $\Elimi{ \AddInvMat \otimes \AddInvMat }$ into \eqref{eq:MagnusProperties_Det} gives $\det(  \AddInvMat )\neq 0$, that is, $\AddInvMat$ is nonsingular.

Next, we use the settings \eqref{eq:AIMexpolyVMat_to_expolyVMat}, \eqref{eq:AddInvMat_to_exAddMat}, and \eqref{eq:AIMFBgain_to_FBgain}.	
By using the relation $(\NotationMat_{1} \otimes \NotationMat_{2})(\NotationMat_{3} \otimes \NotationMat_{4})=
\NotationMat_{1}\NotationMat_{3} \otimes \NotationMat_{2}\NotationMat_{4}
$ 
\cite[Section 3.2.9]{Gentle07}, \eqref{eq:MagnusProperties_ExcludeDL}, and \eqref{eq:MagnusProperties_Inverse}, we transform ${\QMILowLeftblock{\IDpoly}{\AIMFBgain,\AddInvMat}}$
as follows:
\begin{align}
&
\exAddMat^{\MyTRANSPO}
{\QMILowLeftblock{\IDpoly}{\AIMFBgain,\AddInvMat}}
\exAddMat
\nonumber\\&
=
\exAddMat^{\MyTRANSPO}
{\QMILowLeftblock{\IDpoly}{\AIMFBgain,\AddInvMat}}
\Elimi{ \AddInvMat \otimes \AddInvMat }^{-1}
\nonumber\\&
=
\exAddMat^{\MyTRANSPO}
\EliMat
\big( 
{\expolyAAMat{\IDpoly}}  (\AddInvMat \otimes \AddInvMat)
-{\expolyABMat{\IDpoly}} (\AddInvMat \otimes \AIMFBgain)
\VisibleColTwo{\nonumber\\&\quad}
-{\expolyBAMat{\IDpoly}} (\AIMFBgain \otimes \AddInvMat)
+{\expolyBBMat{\IDpoly}} (\AIMFBgain \otimes \AIMFBgain)
\big)
\DupMat
%\nonumber\\&\quad \times
\EliMat ( \AddInvMat^{-1} \otimes \AddInvMat^{-1} ) \DupMat
\nonumber\\&
=
\exAddMat^{\MyTRANSPO}
\EliMat
\big( 
{\expolyAAMat{\IDpoly}}  (\AddInvMat \otimes \AddInvMat)
-{\expolyABMat{\IDpoly}} (\AddInvMat \otimes \AIMFBgain)
\VisibleColTwo{\nonumber\\&\quad}
-{\expolyBAMat{\IDpoly}} (\AIMFBgain \otimes \AddInvMat)
+{\expolyBBMat{\IDpoly}} (\AIMFBgain \otimes \AIMFBgain)
\big)
%\nonumber\\&\quad \times
( \AddInvMat^{-1} \otimes \AddInvMat^{-1} ) \DupMat
\nonumber\\&
=
\exAddMat^{\MyTRANSPO}
\EliMat
\big( 
{\expolyAAMat{\IDpoly}}  (\Identity{\DimX} \otimes \Identity{\DimX})
-{\expolyABMat{\IDpoly}} (\Identity{\DimX} \otimes \FBgain)
\VisibleColTwo{\nonumber\\&\quad}
-{\expolyBAMat{\IDpoly}} (\FBgain \otimes \Identity{\DimX})
+{\expolyBBMat{\IDpoly}} (\FBgain \otimes \FBgain)
\big) \DupMat
\nonumber\\&
=
\exAddMat^{\MyTRANSPO} 
{\Elimi{\expolyCLMat{\IDpoly}{\FBgain}}}
.
\end{align}
Therefore, if the QMIs \eqref{eq:expanded_ERS_QMI} hold, the CMIs \eqref{eq:expanded_ERS} are satisfied as follows:
\begin{align}
\VisibleColTwo{&}
{\CMIterms{\IDpoly}{\expolyVMat{\IDpoly}}{\exAddMat}{\FBgain}{\exEMSrate}}
\VisibleColTwo{\nonumber\\}
&
=
\begin{bmatrix}
\exEMSrate^{2}  
\exAddMat^{\MyTRANSPO}{\AIMexpolyVMat{\IDpoly}}\exAddMat
&   
{\Elimi{\expolyCLMat{\IDpoly}{\FBgain}}}^{\MyTRANSPO} 
\exAddMat
\\
\exAddMat^{\MyTRANSPO} 
{\Elimi{\expolyCLMat{\IDpoly}{\FBgain}}}
&  
\exAddMat^{\MyTRANSPO} + \exAddMat  -  \exAddMat^{\MyTRANSPO}{\AIMexpolyVMat{\IDpoly}}\exAddMat
\end{bmatrix}
\nonumber\\&
=
\begin{bmatrix}
\exAddMat^{\MyTRANSPO} & 0 \\
0 & \exAddMat^{\MyTRANSPO} 
\end{bmatrix}
{\QMIterms{\IDpoly}{\AIMexpolyVMat{\IDpoly}}{\AddInvMat}{\AIMFBgain}{\exEMSrate}}
\begin{bmatrix}
\exAddMat & 0 \\
0 & \exAddMat
\end{bmatrix}
\nonumber\\&
\succeq 0
, \label{eq:QMI_to_CMI_pf}
\end{align}	
where the above block diagonal matrix using $\exAddMat$ is nonsingular. 
We can prove that {\itemMultiPRSQMIs} implies {\itemMultiPRSCMIs} in a manner similar to this proof, by taking the strict inequality in \eqref{eq:expanded_ERS_QMI} and \eqref{eq:QMI_to_CMI_pf} and substituting $\exEMSrate=1$ into them.
In addition, \eqref{eq:AIMexpolyVMat_to_expolyVMat} indicates that  ${\expolyVMat{\IDpoly}}$ are identical if  ${\AIMexpolyVMat{\IDpoly}}$ are identical.
This complete the proof.

%%%%%%%%%%%%%%%%%%%%%%%%%%%%%%%%%%%%%%%%%%%%%%%%%%%%%%%%%%%%%%%%%%%%%%%%%%%%%%%%%%%%%%%%%%%%%%%%%%%%%%%%%%%%%%%%%%%%%%%%%%%%%%%%%%%%%%%
%%%%%%%%%%%%%%%%%%%%%%%%%%%%%%%%%%%%%%%%%%%%%%%%%%%%%%%%%%%%%%%%%%%%%%%%%%%%%%%%%%%%%%%%%%%%%%%%%%%%%%%%%%%%%%%%%%%%%%%%%%%%%%%%%%%%%%%
%%%%%%%%%%%%%%%%%%%%%%%%%%%%%%%%%%%%%%%%%%%%%%%%%%%%%%%%%%%%%%%%%%%%%%%%%%%%%%%%%%%%%%%%%%%%%%%%%%%%%%%%%%%%%%%%%%%%%%%%%%%%%%%%%%%%%%%
%%%%%%%%%%%%%%%%%%%%%%%%%%%%%%%%%%%%%%%%%%%%%%%%%%%%%%%%%%%%%%%%%%%%%%%%%%%%%%%%%%%%%%%%%%%%%%%%%%%%%%%%%%%%%%%%%%%%%%%%%%%%%%%%%%%%%%%
\section{Proof of Proposition \ref{thm:dummyHMfunc}} \label{pf:dummyHMfunc}

Combining \eqref{eq:RankOneEigendecomposition} with \eqref{eq:def2_BlockHMMat} yields
	\begin{align}
	\dummyHMMat
	%&
	=\begin{bmatrix}
	\VEC{\AddInvMat} \\	 \VEC{\AIMFBgain} 
	\end{bmatrix}
	\begin{bmatrix}
	\VEC{\AddInvMat} \\	 \VEC{\AIMFBgain} 
	\end{bmatrix}^{\MyTRANSPO}
	%\nonumber\\&
	=
	\begin{bmatrix}
	{\BlockHMMat{1}{1}} & \cdots & {\BlockHMMat{1}{2\DimX}} \\
	\vdots & \ddots & \vdots \\
	{\BlockHMMat{2\DimX}{1}} & \cdots & {\BlockHMMat{2\DimX}{2\DimX}} \\
	\end{bmatrix}
	.
	\end{align}
In a manner similar to \eqref{eq:relation_M_to_AB}, for any $\IDEl \in \{1,\dots, \DimX\}$ and $\IDbEl \in \{1,\dots, \DimX\}$, we obtain 
	\begin{align}
	{\El{ \AddInvMat \otimes \AddInvMat }{\vectorWildCard, \DimX(\IDbEl-1) +\IDEl   }}
	&=\begin{bmatrix}
	{\El{\AddInvMat}{1,\IDbEl}} {\El{\AddInvMat}{\vectorWildCard,\IDEl}} \\
	\vdots \\
	{\El{\AddInvMat}{\DimX,\IDbEl}} {\El{\AddInvMat}{\vectorWildCard,\IDEl}}
	\end{bmatrix}
	\VisibleColTwo{\nonumber\\&}
	=\VEC{   {\El{\AddInvMat}{\vectorWildCard,\IDEl}} {\El{\AddInvMat}{\vectorWildCard,\IDbEl}^{\MyTRANSPO}}   }
	\VisibleColTwo{\nonumber\\&}
	=
	\VEC{ {\BlockHMMat{\IDEl}{\IDbEl}}   }
	, %\label{eq:relation_M_to_AB}
	\\
	{\El{ \AddInvMat \otimes \AIMFBgain }{\vectorWildCard, \DimX(\IDbEl-1) +\IDEl   }}
	&=\begin{bmatrix}
	{\El{\AddInvMat}{1,\IDbEl}} {\El{\AIMFBgain}{\vectorWildCard,\IDEl}} \\
	\vdots \\
	{\El{\AddInvMat}{\DimX,\IDbEl}} {\El{\AIMFBgain}{\vectorWildCard,\IDEl}}
	\end{bmatrix}
	\VisibleColTwo{\nonumber\\&}
	=\VEC{   {\El{\AIMFBgain}{\vectorWildCard,\IDEl}} {\El{\AddInvMat}{\vectorWildCard,\IDbEl}^{\MyTRANSPO}}   }
	\VisibleColTwo{\nonumber\\&}
	=
	\VEC{ {\BlockHMMat{\DimX+\IDEl}{\IDbEl}}   }
	, %\label{eq:relation_M_to_AB}
	\\
	{\El{ \AIMFBgain \otimes \AddInvMat }{\vectorWildCard, \DimX(\IDbEl-1) +\IDEl   }}
	&=\begin{bmatrix}
	{\El{\AIMFBgain}{1,\IDbEl}} {\El{\AddInvMat}{\vectorWildCard,\IDEl}} \\
	\vdots \\
	{\El{\AIMFBgain}{\DimX,\IDbEl}} {\El{\AddInvMat}{\vectorWildCard,\IDEl}}
	\end{bmatrix}
	\VisibleColTwo{\nonumber\\&}
	=\VEC{   {\El{\AddInvMat}{\vectorWildCard,\IDEl}} {\El{\AIMFBgain}{\vectorWildCard,\IDbEl}^{\MyTRANSPO}}   }
	\VisibleColTwo{\nonumber\\&}
	=
	\VEC{ {\BlockHMMat{\IDEl}{\DimX+\IDbEl}}   }
	, %\label{eq:relation_M_to_AB}
	\\
	{\El{ \AIMFBgain \otimes \AIMFBgain }{\vectorWildCard, \DimX(\IDbEl-1) +\IDEl   }}
	&=\begin{bmatrix}
	{\El{\AIMFBgain}{1,\IDbEl}} {\El{\AIMFBgain}{\vectorWildCard,\IDEl}} \\
	\vdots \\
	{\El{\AIMFBgain}{\DimX,\IDbEl}} {\El{\AIMFBgain}{\vectorWildCard,\IDEl}}
	\end{bmatrix}
	\VisibleColTwo{\nonumber\\&}
	=\VEC{   {\El{\AIMFBgain}{\vectorWildCard,\IDEl}} {\El{\AIMFBgain}{\vectorWildCard,\IDbEl}^{\MyTRANSPO}}   }
	\VisibleColTwo{\nonumber\\&}
	=
	\VEC{ {\BlockHMMat{\DimX+\IDEl}{\DimX+\IDbEl}}   }
	. %\label{eq:relation_M_to_AB}
	\end{align}
	Substituting these relations into \eqref{eq:def2_dummyHHfunc}--\eqref{eq:def2_dummyMMfunc} yields \eqref{eq:def_dummyHHfunc}--\eqref{eq:def_dummyMMfunc}.
This completes the proof.

%%%%%%%%%%%%%%%%%%%%%%%%%%%%%%%%%%%%%%%%%%%%%%%%%%%%%%%%%%%%%%%%%%%%%%%%%%%%%%%%%%%%%%%%%%%%%%%%%%%%%%%%%%%%%%%%%%%%%%%%%%%%%%%%%%%%%%%
%%%%%%%%%%%%%%%%%%%%%%%%%%%%%%%%%%%%%%%%%%%%%%%%%%%%%%%%%%%%%%%%%%%%%%%%%%%%%%%%%%%%%%%%%%%%%%%%%%%%%%%%%%%%%%%%%%%%%%%%%%%%%%%%%%%%%%%
%%%%%%%%%%%%%%%%%%%%%%%%%%%%%%%%%%%%%%%%%%%%%%%%%%%%%%%%%%%%%%%%%%%%%%%%%%%%%%%%%%%%%%%%%%%%%%%%%%%%%%%%%%%%%%%%%%%%%%%%%%%%%%%%%%%%%%%
%%%%%%%%%%%%%%%%%%%%%%%%%%%%%%%%%%%%%%%%%%%%%%%%%%%%%%%%%%%%%%%%%%%%%%%%%%%%%%%%%%%%%%%%%%%%%%%%%%%%%%%%%%%%%%%%%%%%%%%%%%%%%%%%%%%%%%%
\section{Proof of Theorem \ref{thm:stability_QMIwithRankOne}} \label{pf:stability_QMIwithRankOne}

If {\itemMultiPERSQMIs} holds, 
we use $\dummyHMMat$ given in \eqref{eq:RankOneEigendecomposition}.
Because $\AddInvMat$ is nonsingular, $\dummyHMMat \neq 0$ and thus $0 \leq \MyRank{\dummyHMMat}=1$ hold.
Substituting \eqref{eq:def_dummyHHfunc}--\eqref{eq:def_dummyMMfunc} yields ${\SDPLMIterms{\IDpoly}{\AIMexpolyVMat{\IDpoly}}{\dummyHMMat}{\exEMSrate}}={\QMIterms{\IDpoly}{\AIMexpolyVMat{\IDpoly}}{\AddInvMat}{\AIMFBgain}{\exEMSrate}} \succeq 0$ that implies {\itemMultiPERSSDPs}.

Next,
if {\itemMultiPERSSDPs} holds, using  \eqref{eq:RankOneEigendecomposition} gives $\AddInvMat$ and $\AIMFBgain$ because of $\MyRank{\dummyHMMat}=1$.
Substituting these matrices yields ${\QMIterms{\IDpoly}{\AIMexpolyVMat{\IDpoly}}{\AddInvMat}{\AIMFBgain}{\exEMSrate}}={\SDPLMIterms{\IDpoly}{\AIMexpolyVMat{\IDpoly}}{\dummyHMMat}{\exEMSrate}} \succeq 0$  that implies {\itemMultiPERSQMIs}.

In the same way, we can prove the equivalence between {\itemMultiPRSQMIs} and {\itemMultiPRSSDPs}.
This completes the proof.

%%%%%%%%%%%%%%%%%%%%%%%%%%%%%%%%%%%%%%%%%%%%%%%%%%%%%%%%%%%%%%%%%%%%%%%%%%%%%%%%%%%%%%%%%%%%%%%%%%%%%%%%%%%%%%%%%%%%%%%%%%%%%%%%%%%%%%%
%%%%%%%%%%%%%%%%%%%%%%%%%%%%%%%%%%%%%%%%%%%%%%%%%%%%%%%%%%%%%%%%%%%%%%%%%%%%%%%%%%%%%%%%%%%%%%%%%%%%%%%%%%%%%%%%%%%%%%%%%%%%%%%%%%%%%%%
%%%%%%%%%%%%%%%%%%%%%%%%%%%%%%%%%%%%%%%%%%%%%%%%%%%%%%%%%%%%%%%%%%%%%%%%%%%%%%%%%%%%%%%%%%%%%%%%%%%%%%%%%%%%%%%%%%%%%%%%%%%%%%%%%%%%%%%
%%%%%%%%%%%%%%%%%%%%%%%%%%%%%%%%%%%%%%%%%%%%%%%%%%%%%%%%%%%%%%%%%%%%%%%%%%%%%%%%%%%%%%%%%%%%%%%%%%%%%%%%%%%%%%%%%%%%%%%%%%%%%%%%%%%%%%%
\section{Proof of Lemma \ref{thm:multiSDPfunc}} \label{pf:multiSDPfunc}

For any $\dummyHMMat\succeq 0$, $\iEig{\IDEl}{  {\dummyHMMat}   } \geq 0$ holds for $\IDEl \in \{1,\dots,\DimX(\DimX+\DimU)\} $.
Using ${\iEigVec{1}{\dummyHMMat}}^{\MyTRANSPO} \dummyHMMat {\iEigVec{1}{\dummyHMMat}}
={\iEigVec{1}{\dummyHMMat}}^{\MyTRANSPO} {\iEig{1}{\dummyHMMat}} {\iEigVec{1}{\dummyHMMat}}
={\iEig{1}{\dummyHMMat}}$ and 
$\TRACE{\dummyHMMat}=\sum_{\IDEl=1}^{\DimX(\DimX+\DimU)} {\iEig{\IDEl}{\dummyHMMat}}$ yields	
\begin{align}
{\rankoneApproxErr{\dummyHMMat}{\dummyHMMat}}
= \Big(\sum_{\IDEl=1}^{\DimX(\DimX+\DimU)}| {\iEig{\IDEl}{\dummyHMMat}}| \Big)
-
| {\iEig{1}{\dummyHMMat}}|
= {\rankoneErr{\dummyHMMat}}
. %\label{eq:trace_norm}
\end{align}

Next, we choose unit eigenvectors ${\iEigVec{\IDEl}{\dummyPreHMMat}}$ such that $[{\iEigVec{1}{\dummyPreHMMat}},\dots,{\iEigVec{\DimX(\DimX+\DimU)}{\dummyPreHMMat}}]$ is orthogonal because $\dummyPreHMMat \neq 0$ is symmetric \cite[Section 3.8.7]{Gentle07}.
This choice enables the use of the spectral decomposition 
$
{\Identity{\DimX(\DimX+\DimU)}} 
=
\sum_{\IDEl=1}^{\DimX(\DimX+\DimU)} {\iEigVec{\IDEl}{\dummyPreHMMat}}  {\iEigVec{\IDEl}{\dummyPreHMMat}}^{\MyTRANSPO}
$ 
\cite[Section 3.8.7]{Gentle07} and the results in \cite[Lemma 4.2]{Liu19} with the replacement of singular values with eigenvalues.
Thus, we obtain
\begin{align}
\VisibleColTwo{&}
{\rankoneApproxErr{\dummyHMMat}{\dummyPreHMMat}}
\VisibleColTwo{\nonumber\\}
&
=
\TRACE{	\dummyHMMat }
-
\TRACE{
	{\iEigVec{1}{\dummyPreHMMat}} {\iEigVec{1}{\dummyPreHMMat}}^{\MyTRANSPO} 
	\dummyHMMat  		
}
\nonumber\\&
=
\TRACE[\big]{
	(
	{\Identity{\DimX(\DimX+\DimU)}} - 
	{\iEigVec{1}{\dummyPreHMMat}} {\iEigVec{1}{\dummyPreHMMat}}^{\MyTRANSPO} 
	)
	\dummyHMMat  		
}
\nonumber\\&
=	
\TRACE[\Big]{
	\Big(
	\sum_{\IDbEl=2}^{\DimX(\DimX+\DimU)} {\iEigVec{\IDbEl}{\dummyPreHMMat}}  {\iEigVec{\IDbEl}{\dummyPreHMMat}}^{\MyTRANSPO}
	\Big)
	\dummyHMMat  		
}
\nonumber\\&
\geq
\sum_{\IDEl =1}^{\DimX(\DimX+\DimU)  }  
%\Big(
{\iEig{\IDEl}{  
\sum_{\IDbEl=2}^{\DimX(\DimX+\DimU)} {\iEigVec{\IDbEl}{\dummyPreHMMat}}  {\iEigVec{\IDbEl}{\dummyPreHMMat}}^{\MyTRANSPO}
  }} 
%\nonumber\\& \qquad\qquad\qquad\qquad \times
{\iEig{\DimX(\DimX+\DimU)  - \IDEl +1}{ \dummyHMMat }}
%\Big)
\nonumber\\&
=
\sum_{\IDEl =1}^{\DimX(\DimX+\DimU) - 1 }  
{\iEig{\DimX(\DimX+\DimU)  - \IDEl +1}{ \dummyHMMat }}
\nonumber\\&
=
\sum_{\IDEl =2}^{\DimX(\DimX+\DimU)  }  
|{\iEig{\IDEl}{ \dummyHMMat }}|
.
\end{align}
Note that the above inequality is derived from the property 
$\sum_{\IDEl=1}^{\DimX}
{\iEig{\IDEl}{ \NotationSymMat_{1} }}
{\iEig{\DimX  - \IDEl +1}{ \NotationSymMat_{2} }}
\leq
\TRACE{  \NotationSymMat_{1} \NotationSymMat_{2} }
$ for any $\NotationSymMat_{1} \succeq 0 \in {\SetSymMat{\DimX}}$ and $\NotationSymMat_{2} \succeq 0 \in {\SetSymMat{\DimX}}$ \cite[Lemma 4.1]{Liu19}\cite[Theorem II.1]{Lasserre95}.
This completes the proof.

%%%%%%%%%%%%%%%%%%%%%%%%%%%%%%%%%%%%%%%%%%%%%%%%%%%%%%%%%%%%%%%%%%%%%%%%%%%%%%%%%%%%%%%%%%%%%%%%%%%%%%%%%%%%%%%%%%%%%%%%%%%%%%%%%%%%%%%
%%%%%%%%%%%%%%%%%%%%%%%%%%%%%%%%%%%%%%%%%%%%%%%%%%%%%%%%%%%%%%%%%%%%%%%%%%%%%%%%%%%%%%%%%%%%%%%%%%%%%%%%%%%%%%%%%%%%%%%%%%%%%%%%%%%%%%%
%%%%%%%%%%%%%%%%%%%%%%%%%%%%%%%%%%%%%%%%%%%%%%%%%%%%%%%%%%%%%%%%%%%%%%%%%%%%%%%%%%%%%%%%%%%%%%%%%%%%%%%%%%%%%%%%%%%%%%%%%%%%%%%%%%%%%%%
%%%%%%%%%%%%%%%%%%%%%%%%%%%%%%%%%%%%%%%%%%%%%%%%%%%%%%%%%%%%%%%%%%%%%%%%%%%%%%%%%%%%%%%%%%%%%%%%%%%%%%%%%%%%%%%%%%%%%%%%%%%%%%%%%%%%%%%
\section{Proof of Theorem \ref{thm:TISMP_examples}} \label{pf:TISMP_examples}

First, we prove the statement \ref{item_iid}.
Using 
the i.i.d. property and \eqref{eq:iid_NUMpoly}--\eqref{eq:iid_polySMvecAB}
gives
$
{\CondExpectTI[\big]{   {\vecTIAB{\MyT}}{\vecTIAB{\MyT}^{\MyTRANSPO}}    }}
=
{\Expect[]{ {\NonArgvecAB{\MyT}} {\NonArgvecAB{\MyT}^{\MyTRANSPO}} }}
=
{\polySMvecAB{1}}
$, which corresponds to the condition \eqref{eq:SMP_AA_AB_BB}.
Because of \eqref{eq:iid_MapPolyW}, ${\MapPolyW{\TIUnc}} \in \DomPolytope$ holds clearly.
All the conditions for the SMP systems in Definition \ref{def:TVSMP} are satisfied, implying the statement \ref{item_iid}.

Next, we prove the statement \ref{item_random_polytope}.
Using \eqref{eq:setting1_polyW} yields ${\El{\TIUnc}{\IDpoly}}{\El{\TIUnc}{\IDbpoly}} = {\El{\VEC{ {\TIUnc}  {\TIUnc^{\MyTRANSPO}}  }}{   \DimTVUnc(\IDbpoly-1) + \IDpoly    }}= {\El{\MapPolyW{\TIUnc}}{   \DimTVUnc(\IDbpoly-1) + \IDpoly    }}$.
Using this relation and \eqref{eq:setting1_polySMvecAB} yields
\begin{align}
\VisibleColTwo{&}
{\CondExpectTI[\big]{   {\vecTIAB{\MyT}}{\vecTIAB{\MyT}^{\MyTRANSPO}}    }}
\VisibleColTwo{\nonumber\\}
&=
{\CondExpectTI[\Big]{ 
		\sum_{\IDpoly  =1}^{\DimTVUnc} 
		\sum_{\IDbpoly =1}^{\DimTVUnc} 
		{\El{\TIUnc}{\IDpoly}} {\El{\TIUnc}{\IDbpoly}} 
		{\polyRandvecAB{\IDpoly}}{\polyRandvecAB{\IDbpoly}}^{\MyTRANSPO}
}}
\nonumber\\
&=
\sum_{\IDpoly  =1}^{\DimTVUnc} 
\sum_{\IDbpoly =1}^{\DimTVUnc} 
{\El{\TIUnc}{\IDpoly}} {\El{\TIUnc}{\IDbpoly}} 
\frac{\Expect{ 
		{\polyRandvecAB{\IDpoly}}{\polyRandvecAB{\IDbpoly}}^{\MyTRANSPO}
		+
		{\polyRandvecAB{\IDbpoly}}{\polyRandvecAB{\IDpoly}}^{\MyTRANSPO}		
}}{2}
\nonumber\\
&=
\sum_{\IDpoly  =1}^{\DimTVUnc} 
\sum_{\IDbpoly =1}^{\DimTVUnc} 
{\El{\MapPolyW{\TIUnc}}{   \DimTVUnc(\IDbpoly-1) + \IDpoly    }}
{\polySMvecAB{ \DimTVUnc(\IDbpoly-1) + \IDpoly  }}
\nonumber\\
&=
\sum_{\IDpoly=1}^{\NUMpoly}  {\El{\MapPolyW{\TIUnc}}{\IDpoly}}  {\polySMvecAB{\IDpoly}}
.
\label{eq:SMP_AA_AB_BB_random_polytope}
\end{align}
In addition, for any ${\TIUnc} \in \DomTVUnc$, 
$\sum_{\IDpoly=1}^{\NUMpoly}{\El{\MapPolyW{\TIUnc}}{\IDpoly}} 
=
\sum_{\IDpoly  =1}^{\DimTVUnc} 
\sum_{\IDbpoly =1}^{\DimTVUnc}
{\El{\TIUnc}{\IDpoly}} {\El{\TIUnc}{\IDbpoly}} 
=1
$
and
${\El{\MapPolyW{\TIUnc}}{\IDpoly}} \geq 0$  
hold.
This indicates that ${\MapPolyW{\TIUnc}} \in \DomPolytope$ holds for any ${\TIUnc} \in \DomTVUnc$.
This result and \eqref{eq:SMP_AA_AB_BB_random_polytope} satisfy all the conditions for the SMP systems in Definition \ref{def:TVSMP}, implying the statement \ref{item_random_polytope}.

Next, we prove the statement \ref{item_det_polytope}, where ${\MapPolyW{\TIUnc}} \in \DomPolytope$ was already proved.
We can derive the condition \eqref{eq:SMP_AA_AB_BB} in a manner similar to \eqref{eq:SMP_AA_AB_BB_random_polytope} by excluding the expectation. 
Thus, the statement \ref{item_det_polytope} holds.

Finally,  we prove the statement \ref{item_uncMeanCov_NUMpoly}, where ${\MapPolyW{\TIUnc}} \in \DomPolytope$ was already proved.
In a manner similar to \eqref{eq:SMP_AA_AB_BB_random_polytope}, we obtain
\begin{align}
\VisibleColTwo{&}
{\CondExpectTI[\big]{   {\vecTIAB{\MyT}}{\vecTIAB{\MyT}^{\MyTRANSPO}}    }}
\VisibleColTwo{\nonumber\\}
&=
{\CondExpectTI{\vecTIAB{\MyT}}} {\CondExpectTI{\vecTIAB{\MyT}}}^{\MyTRANSPO} 
+
{\CondCovTI[\big]{\vecTIAB{\MyT}}}
\nonumber\\
&=
\sum_{\IDpoly  =1}^{\DimTVUnc} 
\sum_{\IDbpoly =1}^{\DimTVUnc} 
{\El{\TIUnc}{\IDpoly}} {\El{\TIUnc}{\IDbpoly}} 
\frac{ 
{\polyMeanvecAB{\IDpoly}}{\polyMeanvecAB{\IDbpoly}}^{\MyTRANSPO}
+
{\polyMeanvecAB{\IDbpoly}}{\polyMeanvecAB{\IDpoly}}^{\MyTRANSPO}		
}{2}
\VisibleColTwo{\nonumber\\&\quad}
+
\Big( \sum_{\IDbpoly =1}^{\DimTVUnc}   {\El{\TIUnc}{\IDbpoly}}  \Big)
\sum_{\IDpoly  =1}^{\DimTVUnc} 
{\El{\TIUnc}{\IDpoly}}  {\polyCovvecAB{\IDpoly}} 
\nonumber\\
&=
\sum_{\IDpoly  =1}^{\DimTVUnc} 
\sum_{\IDbpoly =1}^{\DimTVUnc} 
{\El{\MapPolyW{\TIUnc}}{   \DimTVUnc(\IDbpoly-1) + \IDpoly    }}
{\polySMvecAB{ \DimTVUnc(\IDbpoly-1) + \IDpoly  }}
.\label{eq:SMP_AA_AB_BB_uncMeanCov}
\end{align}
Therefore, the statement \ref{item_uncMeanCov_NUMpoly} holds.
This completes the proof.

%%%%%%%%%%%%%%%%%%%%%%%%%%%%%%%%%%%%%%%%%%%%%%%%%%%%%%%%%%%%%%%%%%%%%%%%%%%%%%%%%%%%%%%%%%%%%%%%%%%%%%%%%%%%%%%%%%%%%%%%%%%%%%%%%%%%%%%
%%%%%%%%%%%%%%%%%%%%%%%%%%%%%%%%%%%%%%%%%%%%%%%%%%%%%%%%%%%%%%%%%%%%%%%%%%%%%%%%%%%%%%%%%%%%%%%%%%%%%%%%%%%%%%%%%%%%%%%%%%%%%%%%%%%%%%%
%%%%%%%%%%%%%%%%%%%%%%%%%%%%%%%%%%%%%%%%%%%%%%%%%%%%%%%%%%%%%%%%%%%%%%%%%%%%%%%%%%%%%%%%%%%%%%%%%%%%%%%%%%%%%%%%%%%%%%%%%%%%%%%%%%%%%%%
%%%%%%%%%%%%%%%%%%%%%%%%%%%%%%%%%%%%%%%%%%%%%%%%%%%%%%%%%%%%%%%%%%%%%%%%%%%%%%%%%%%%%%%%%%%%%%%%%%%%%%%%%%%%%%%%%%%%%%%%%%%%%%%%%%%%%%%

%%%%%%%%%%%%%%%%%%%%%%%%%%%%%%%%%%%%%%%%%%%%%%%%%%%%%%%%%%%%%%%%%%%%%%%%%%%%%%%%%%%%%%%%%%%%%%%%%%%%%%%%%%%%%%%%%%%%%%%%%%%%%%%%%%%%%%%
%%%%%%%%%%%%%%%%%%%%%%%%%%%%%%%%%%%%%%%%%%%%%%%%%%%%%%%%%%%%%%%%%%%%%%%%%%%%%%%%%%%%%%%%%%%%%%%%%%%%%%%%%%%%%%%%%%%%%%%%%%%%%%%%%%%%%%%
%%%%%%%%%%%%%%%%%%%%%%%%%%%%%%%%%%%%%%%%%%%%%%%%%%%%%%%%%%%%%%%%%%%%%%%%%%%%%%%%%%%%%%%%%%%%%%%%%%%%%%%%%%%%%%%%%%%%%%%%%%%%%%%%%%%%%%%	
\section{Proof of Proposition \ref{thm:SMP_vs_random_polytope}} \label{pf:SMP_vs_random_polytope}

Let us consider the TI version of the SMP system involving ${\anothervecTIAB{\MyT}}$ in Example \ref{ex:SMP} with the following settings: 
$\DimX=\DimU=1$, 
$\anotherDimTVUnc=2$,
$\anotherDomTVUnc=\DomPolytope|_{\NUMpoly=2}$,
$\ConstMeanvecAB=0$,
${\El{\polyCovvecAB{\IDpoly}}{1,1}}={\El{\polyCovvecAB{\IDpoly}}{2,2}}={\pfCovNonDiagVal{0}}$,
and
${\El{\polyCovvecAB{\IDpoly}}{1,2}}={\El{\polyCovvecAB{\IDpoly}}{2,1}}={\pfCovNonDiagVal{\IDpoly}}$,
where ${\pfCovNonDiagVal{\IDpoly}}$ are constants satisfying ${\pfCovNonDiagVal{1}}\neq {\pfCovNonDiagVal{2}}$.
Let us suppose that this SMP system is represented by a random polytope with ${\vecTIAB{\MyT}}$  given in Theorem \ref{thm:TISMP_examples} \ref{item_random_polytope}, 
that is, \eqref{eq:SMPvsRP_goalA} and \eqref{eq:SMPvsRP_goalB} are assumed to hold.
Then, we obtain
\begin{align}
&
\sum_{\IDpoly  =1}^{\DimTVUnc} 
\sum_{\IDbpoly =1}^{\DimTVUnc} 
{\El{\TIUnc}{\IDpoly}} {\El{\TIUnc}{\IDbpoly}} 
{\Expect[\Bigg]{ 
\begin{bmatrix}
{\El{\polyRandvecAB{\IDpoly}}{1}}{\El{\polyRandvecAB{\IDbpoly}}{1}} &
{\El{\polyRandvecAB{\IDpoly}}{1}}{\El{\polyRandvecAB{\IDbpoly}}{2}} \\
{\El{\polyRandvecAB{\IDpoly}}{2}}{\El{\polyRandvecAB{\IDbpoly}}{1}} &
{\El{\polyRandvecAB{\IDpoly}}{2}}{\El{\polyRandvecAB{\IDbpoly}}{2}}
\end{bmatrix}
}}
\nonumber\\
&
=
{\CondExpectTI[\big]{   {\vecTIAB{\MyT}}{\vecTIAB{\MyT}^{\MyTRANSPO}}    }}
\nonumber\\
&=
{\anotherCondExpectTI{    {\anothervecTIAB{\MyT}} {\anothervecTIAB{\MyT}^{\MyTRANSPO}}        }}
\nonumber\\
&
=
\begin{bmatrix}
{\pfCovNonDiagVal{0}} &
{\El{\anotherTIUnc}{1}} {\pfCovNonDiagVal{1}} + {\El{\anotherTIUnc}{2}} {\pfCovNonDiagVal{2}} \\ 
{\El{\anotherTIUnc}{1}} {\pfCovNonDiagVal{1}} + {\El{\anotherTIUnc}{2}} {\pfCovNonDiagVal{2}} & 
{\pfCovNonDiagVal{0}}
\end{bmatrix}
.\label{eq:pf_randopolyVSsmps_1}
\end{align}
For each $\IDpoly$, substituting ${\El{\TIUnc}{\IDpoly}}=1$ into \eqref{eq:pf_randopolyVSsmps_1} yields
\begin{align}
&
\IDpoly \in \{ 1,\dots,\DimTVUnc \}
,\quad
{\Expect[]{ \El{\polyRandvecAB{\IDpoly}}{1}^{2}  }}
%={\Expect[]{ \El{\polyRandvecAB{\IDpoly}}{2}^{2}  }}
={\pfCovNonDiagVal{0}}.
\end{align}
For each $(\IDpoly,\IDbpoly)\in \{ 1,\dots,\DimTVUnc \}^{2}$, substituting ${\El{\TIUnc}{\IDpoly}}={\El{\TIUnc}{\IDbpoly}}=1/2$ into \eqref{eq:pf_randopolyVSsmps_1} gives
\begin{align}
{\Expect[]{ 
		\El{\polyRandvecAB{\IDpoly}}{1}^{2} 
		+\El{\polyRandvecAB{\IDbpoly}}{1}^{2}
		+2 {\El{\polyRandvecAB{\IDpoly}}{1}}	 {\El{\polyRandvecAB{\IDbpoly}}{1}}
 }}
/4
={\pfCovNonDiagVal{0}}
.
\end{align}
This implies ${\Expect[]{  {\El{\polyRandvecAB{\IDpoly}}{1}}	 {\El{\polyRandvecAB{\IDbpoly}}{1}}	}}={\pfCovNonDiagVal{0}}$.
Then, 
applying the Cauchy-Schwarz inequality \cite[Section 5.5]{Gray09} to ${\El{\polyRandvecAB{1}}{2}}(  {\El{\polyRandvecAB{\IDpoly}}{1}} - {\El{\polyRandvecAB{1}}{1}} )$ gives
\begin{align}
\VisibleColTwo{&}
{\Expect[\big]{ 
{\El{\polyRandvecAB{1}}{2}}(  {\El{\polyRandvecAB{\IDpoly}}{1}} - {\El{\polyRandvecAB{1}}{1}} )
}^{2}}
\VisibleColTwo{\nonumber\\}
&
\leq 
{\Expect[]{ {\El{\polyRandvecAB{1}}{2}^{2}}     }}
{\Expect[]{ (  {\El{\polyRandvecAB{\IDpoly}}{1}} - {\El{\polyRandvecAB{1}}{1}} )^{2}  }}
\nonumber\\&
 \leq 
{\Expect[]{ {\El{\polyRandvecAB{1}}{2}^{2}}     }}
\big(
{\Expect[]{ {\El{\polyRandvecAB{\IDpoly}}{1}^{2}}     }}
+{\Expect[]{ {\El{\polyRandvecAB{1}}{1}^{2}}     }}
-
2{\Expect[]{ {\El{\polyRandvecAB{\IDpoly}}{1}}{\El{\polyRandvecAB{1}}{1}}     }}
\big)
\nonumber\\&
=0
. 
\end{align}
In a similar manner, ${\Expect[]{  {\El{\polyRandvecAB{\IDpoly}}{1}}(  {\El{\polyRandvecAB{\IDbpoly}}{2}} - {\El{\polyRandvecAB{1}}{2}} )  }}=0$ is derived.
Therefore, substituting these results into \eqref{eq:pf_randopolyVSsmps_1} yields
\begin{align}
&\sum_{\IDpoly  =1}^{\DimTVUnc} 
\sum_{\IDbpoly =1}^{\DimTVUnc} 
{\El{\TIUnc}{\IDpoly}} {\El{\TIUnc}{\IDbpoly}} 
{\Expect[\big]{ 
	{\El{\polyRandvecAB{\IDpoly}}{1}}{\El{\polyRandvecAB{\IDbpoly}}{2}}
}}
\nonumber\\
&=
\sum_{\IDpoly  =1}^{\DimTVUnc} 
\sum_{\IDbpoly =1}^{\DimTVUnc} 
{\El{\TIUnc}{\IDpoly}} {\El{\TIUnc}{\IDbpoly}} 
\Expect[\big]{ 
{\El{\polyRandvecAB{1}}{1}}{\El{\polyRandvecAB{1}}{2}}
\VisibleColTwo{\nonumber\\&\qquad}
+
{\El{\polyRandvecAB{\IDpoly}}{1}}(  {\El{\polyRandvecAB{\IDbpoly}}{2}} - {\El{\polyRandvecAB{1}}{2}} )
+
{\El{\polyRandvecAB{1}}{2}}(  {\El{\polyRandvecAB{\IDpoly}}{1}} - {\El{\polyRandvecAB{1}}{1}} )
}
\nonumber\\
&=
{\Expect[\big]{ 
		{\El{\polyRandvecAB{1}}{1}}{\El{\polyRandvecAB{1}}{2}}
}}
\nonumber\\
&=
{\El{\anotherTIUnc}{1}} {\pfCovNonDiagVal{1}} + (1-{\El{\anotherTIUnc}{1}} ) {\pfCovNonDiagVal{2}} 
.
\end{align}
The result contradicts the condition ${\pfCovNonDiagVal{1}} \neq {\pfCovNonDiagVal{2}} $ because ${\El{\anotherTIUnc}{1}}$ is not constant.
Therefore, for any setting of  ${\polyRandvecAB{\IDpoly}}$, there exists $\anotherTIUnc$ such that every $\TIUnc$ cannot satisfy \eqref{eq:pf_randopolyVSsmps_1}.
This completes the proof.

%%%%%%%%%%%%%%%%%%%%%%%%%%%%%%%%%%%%%%%%%%%%%%%%%%%%%%%%%%%%%%%%%%%%%%%%%%%%%%%%%%%%%%%%%%%%%%%%%%%%%%%%%%%%%%%%%%%%%%%%%%%%%%%%%%%%%%%
%%%%%%%%%%%%%%%%%%%%%%%%%%%%%%%%%%%%%%%%%%%%%%%%%%%%%%%%%%%%%%%%%%%%%%%%%%%%%%%%%%%%%%%%%%%%%%%%%%%%%%%%%%%%%%%%%%%%%%%%%%%%%%%%%%%%%%%
%%%%%%%%%%%%%%%%%%%%%%%%%%%%%%%%%%%%%%%%%%%%%%%%%%%%%%%%%%%%%%%%%%%%%%%%%%%%%%%%%%%%%%%%%%%%%%%%%%%%%%%%%%%%%%%%%%%%%%%%%%%%%%%%%%%%%%%
%%%%%%%%%%%%%%%%%%%%%%%%%%%%%%%%%%%%%%%%%%%%%%%%%%%%%%%%%%%%%%%%%%%%%%%%%%%%%%%%%%%%%%%%%%%%%%%%%%%%%%%%%%%%%%%%%%%%%%%%%%%%%%%%%%%%%%%

\iffalse{\section{END}}\fi

%%%%%%%%%%%%%%%%%%%%%%%%%%%%%%%%%%%%%%%%%%%%%%%%%%%%%%%%%%%%%%%%%%%%%%%%%%%%%%%%
% Generated by IEEEtran.bst, version: 1.14 (2015/08/26)

%%%%%%%%%%%%%%%%%%%%%%%%%%%%%%%%%%%%%%%%%%%%%%%%%%%%%%%%%%%%%%%%%%%%%%%%%%%%%%%%

\end{document}